\documentclass[a4paper,11pt]{article}
\usepackage[latin1]{inputenc}
\usepackage[hmargin={25mm,25mm},vmargin={25mm,30mm}]{geometry}
\usepackage{amsmath,amsfonts,amsthm,amssymb}
\usepackage{cancel}
\usepackage{comment}
\usepackage{enumerate}
\usepackage{graphicx}
\usepackage[ocgcolorlinks,allcolors={blue}]{hyperref}

\usepackage{xcolor}

\usepackage{authblk}
\usepackage[percent]{overpic}
\usepackage{epstopdf}

\usepackage{tabularx}
\usepackage{multirow}
\usepackage{booktabs}

\usepackage{refcheck}

\newtheorem{theorem}{Theorem}
\newtheorem{proposition}[theorem]{Proposition}
\newtheorem{lemma}[theorem]{Lemma}
\newtheorem{corollary}[theorem]{Corollary}
\theoremstyle{remark}
\newtheorem{remark}[theorem]{Remark}
\theoremstyle{definition}
\newtheorem{assumption}{Assumption}

\usepackage{stackengine}

\usepackage{newtxtext,newtxmath}
\usepackage[normalem]{ulem}
\newcounter{corr}
\definecolor{violet}{rgb}{0.580,0.,0.827}
\newcommand{\corr}[3]{\typeout{Warning : a correction remains in page \thepage}
\stepcounter{corr}        
{\color{blue}\ifmmode\text{\,\sout{\ensuremath{#1}}\,}\else\sout{#1}\fi}{\color{red}#2}{\color{violet} [#3]}
}
\definecolor{mygreen}{rgb}{0,0.4,0.2}
\def\b{\boldsymbol}

\newcommand{\R}{\mathbb{R}}    
\newcommand{\N}{\mathbb{N}} 
\newcommand{\Pk}{\mathbb{P}} 
\newcommand{\M}{\mathbb{M}} 
\newcommand{\Pkw}{\widehat{\Pk}}

\newcommand{\GRAD}{\b\nabla}                    %gradient
\newcommand{\GRADs}{\b{\epsilon}}         %symmetric gradient
\newcommand{\divs}{\nabla{\cdot}} 
\newcommand{\DIV}{\b\nabla{\cdot}}               %divergence
\newcommand{\norm}[2][]{\|#2\|_{#1}}

\newcommand{\VCG}{\b U}
\newcommand{\VDL}{\b U_h(E)}
\newcommand{\VDG}{\b U_h}
\newcommand{\ZCG}{\b Z}

\newcommand{\ZDG}{\b Z_h}
\newcommand{\QCG}{P}

\newcommand{\QDG}{P_h}

\def\P0{{\Pi^{0, E}_k}}

\def\PP0{{\boldsymbol{\Pi}^{0, E}_{k-1}}}

\newcommand{\email}[1]{\href{mailto:#1}{#1}}

\def\V{{\b V}}
\def\w{{\b w}}
\def\NE{N_E} % number of element DoFs
\def\DOF{\boldsymbol{\chi}}  % DoF vector and DoF operators 
\def\DOFi{{\chi}_i} 
\def\mesh{\Omega_h}
\def\Proj{\P0}
\def\alf{{\boldsymbol{\alpha}}}
\title{A Virtual Element method for non-Newtonian fluid flows}
\author[1]{Paola F. Antonietti\footnote{\email{paola.antonietti@polimi.it}}}
\author[2]{Louren\c{c}o Beir\~{a}o da Veiga\footnote{\email{lourenco.beirao@unimib.it}}}
\author[1]{Michele Botti \footnote{\email{michele.botti@polimi.it}}}
\author[3]{Giuseppe Vacca \footnote{\email{giuseppe.vacca@uniba.it}}}
\author[1]{Marco Verani \footnote{\email{marco.verani@polimi.it}}}

\affil[1]{MOX-Laboratory for Modeling and Scientific Computing,  Dipartimento di Matematica, Politecnico di Milano, 
Piazza Leonardo da Vinci 32 - 20133 Milano, Italy}

\affil[3]{Dipartimento di Matematica e Applicazioni, 
Universit\`a degli Studi di Milano-Bicocca, 
Via Roberto Cozzi 55  - 20125 Milano, Italy}

\affil[3]{Dipartimento di Matematica, 
Universit\`a degli Studi di Bari, 
Via Edoardo Orabona 4  - 70125 Bari, Italy}

\begin{document}
%% REMOVE KEYS%%%%
\norefnames

\maketitle

\abstract{In this paper, we design and analyze a Virtual Element discretization for the steady motion of non-Newtonian, incompressible fluids. A specific stabilization, tailored to mimic the monotonicity and boundedness properties of the continuous operator, is introduced and theoretically investigated. 
The proposed method has several appealing features, including the exact enforcement of the divergence free condition and the possibility of making use of fully general polygonal meshes. 
A complete well-posedness and convergence analysis of the proposed method is presented under mild assumptions on the non-linear laws, encompassing common examples such as the Carreau--Yasuda model. Numerical experiments validating the theoretical bounds as well as demonstrating the practical capabilities of the proposed formulation are presented.}

\section{Introduction} 

In recent years, a novel approach known as Polytopal Finite Element methods has emerged. Polytopal Finite Element methods are Galerkin-type approximation schemes where the discretization space can support computational grids composed of arbitrarily polygonal or polyhedral (polytopal, for short) elements.
Several methodologies have been developed in the last decade to extend the classical (i.e. based on tetrahedral, hexahedral and prismatic meshes) finite element paradigm to non-conventional elements, see, e.g., \cite{Antonietti_et_al_review,AntoniettiGianiHouston_2013, BeiraodaVeigaBrezziMarini_2023, BeiraodaVeiga-Manzini-Lipnikov_2014, volley, CangianiGeorgoulisHouston_2014, Cockburn-Gopalakrishnan-Lazarov_2009, Cockburn2010, DiPietroDroniou_2020, Di-Pietro.Ern:15, DiPietroErnLemaire_2015} and the references therein. 
Among them, the Virtual Element Method (VEM), first introduced in~\cite{volley} for  second-order elliptic problems and subsequently expanded to cover various other differential equations, has emerged as one of the most promising polytopal approaches. 

Virtual Element Methods have been extensively developed in the last decade to approximate Newtonian fluid flows modeled by Stokes and Navier-Stokes equations. In \cite{AntoniettiBeiraoMoraVerani_2014} a novel stream formulation of the VEM for the solution of the Stokes problem, based on  suitable stream function space characterizing the divergence free subspace of discrete velocities has been proposed and analyzed. Still for the Stokes problem, VEM have been further studied in \cite{Beirao-da-Veiga.Lovadina.ea:17, Cangiani-Gyrya-Manzini:2016,Caceres-Gatica_2017,BeiraodaVeiga-Dassi-Vacca_2020,Chernov-Marcati-Mascotto_2021,Lepe-Rivera_2021,Bevilacqua-Scacchi_2022,VEM_Oseen_stab,Bevilacqua-Dassi-Scacchi-Zampini_2024}, see also \cite{Frerichs-Merdon_2022,BeiraodaVeiga-Dassi-DiPietro-Droniou_2022} where arbitrary-order pressure-robust VEM methods have been studied.
Virtual Element discretization of Navier-Stokes equations has been first studied in ~\cite{BLV:2018} and then further investigated in \cite{Gatica-Munar-Sequeira_2018,BeiraodaVeiga-Mora-Vacca_2019,Adak-Mora-Natarajan_2021,Gatica-Sequeira_2021}, cf. also the recent works \cite{Adak-Mora-Silgado_2024,Li-Hu-Feng_2024} and \cite{Caceres-Gatica-Sequeira_2017} for quasi-Newtonian Stokes flows. VEM for the coupled Navier-Stokes and heat equations have been proposed and analyzed in \cite{Antonietti-Vacca-Verani_2023,BMS-23}, whereas Least-squares type VEMs for the Stokes and Navier-Stokes problems have been recently addressed in \cite{Li-Hu-Feng_2022,Wang-Wang_2024}, respectively. We refer to the book \cite{bookVEM} for a comprehensive overview of the recent developments in Virtual Element Methods. 

On the other hand, complex fluids relevant to engineering and life-science applications often behave as non-Newtonian fluids, involving shear-rate dependent non-linear viscosity. 
Indeed, non-Newtonian fluids' rheological properties, which differ from the linear relationship between shear stress and shear rate observed in Newtonian fluids, make them applicable in a variety of applications, including for example (bio)polymer manufacturing processes, biomedical engineering (blood flow modeling, tissue engineering, and drug delivery), energy production systems through complex geological formations, and modeling of food processing to mention a few.\\

In this paper, we consider the steady motion of a non-Newtonian incompressible fluid, where the constitutive law obeys the so-called Carreau-Yasuda model, usually employed to describe pseudoplastic flow with asymptotic viscosities at zero and infinite shear rates.
From the numerical viewpoint, the numerical discretization of non-Newtonian incompressible fluid flows has been extensively addressed, starting from the seminal work of \cite{BN:1990} where a Finite Element approximation of a non-Newtonian flow model where the viscosity obeys the Carreau law or the power law is considered. Still in the framework of  Finite Element methods, improved error bounds have been proved by \cite{Sandri:93, BL:1993, Barrett.Liu:94}. More precisely,  in the pioneering papers \cite{BL:1993} and \cite{Barrett.Liu:94},  (in some case optimal) velocity and pressure error bounds in suitable  quasi-norms for a non-Newtonian flow model where the viscosity obeys the Carreau or the power law model are proved. The case of a non-degenerate power law is considered in \cite{DG_1990}.
Other notable recent contributions can be found, e.g., in \cite{Belenki.Berselli.ea:12,Hirn:13,Kreuzer.Suli:16,Kalte.Ruzi:23}.

The approximation of non-Newtonian fluid flow problems requires  numerical schemes able to represent the local features emerging from the nonlinear stress-strain relation and to handle complex unstructured and highly-adapted meshes. To this end, discretization methods that can possibly support general polyhedral have been recently explored. We mention in particular the discontinuous Galerkin and hybridizable discontinuous Galerkin methods of \cite{Malkmus.Ruzi.ea:18} and \cite{Gatica.Sequeira:15}, respectively. Additionally, Hybrid High-Order methods have been applied to the simulation of non-Newtonian fluids governed by the Stokes equations in \cite{Botti.Castanon-Quiroz.ea:21} and Navier--Stokes equations with nonlinear convection in \cite{Castanon.Di-Pietro.ea:21}.
Here, we consider the steady motion of a non-Newtonian incompressible fluid, where the constitutive law obeys the so-called Carreau-Yasuda model 
%and we focus the analysis on the (more challenging from the analysis viewpoint) degenerate case (i.e. when $\delta=0$ in \eqref{eq:Carreau}) and 
with the exponent $r$ in \eqref{eq:Carreau} such that $r\in (1,2]$ (shear thinning behavior). Our aim is to develop a Virtual Element approximation of such a problem and present a comprehensive theoretical analysis. 
More precisely,  we consider as starting point the divergence-free Virtual spaces of \cite{Beirao-da-Veiga.Lovadina.ea:17,BLV:2018} and propose a new (divergence-free) Virtual Element formulation for the numerical discretization. 
Such approach exhibits two important advantages. The first one is the possibility of using general polygonal meshes, which allow, for instance, easier descriptions of complex domains and more efficient mesh adaptive strategies. Furthermore, the proposed approach yields an exactly divergence free discrete velocity solution and error estimates for the velocity field which do not depend on the pressure variable; this is a recognized advantage already for the standard linear Stokes problem, which becomes even more relevant in the current nonlinear setting (see for instance \cite{Kreuzer_et_al} and references therein).
We develop a-priori error bounds for both the velocity and the pressure, focusing on the degenerate model with $\delta=0$ in \eqref{eq:Carreau}, which correspond to the power-law equation and which is recognized as the most complex case. 
The theoretical results could be extended to the case $\delta>0$, cf. \eqref{eq:Carreau}, combining the present results with, e.g.,  the approaches in \cite{BN:1990, BL:1993, Barrett.Liu:94}.
Our analysis requires, in particular, the development of many technical results, such as the inf-sup stability of the discrete velocities-pressures coupling in non-hilbertian norms, and the proof of the continuity and coercivity properties of our novel stabilization form, specifically tailored for the current problem. 
To the best of our knowledge, this is the first work in the literature where $O(h^{kr/2})$ velocity error bounds are proved 
for a polytopal approximation method for the degenerate case and regular solution, with $k$ denoting the polynomial order and $r \in (1,2]$ the model exponent, cf. \eqref{eq:Carreau}.  
Our theoretical estimates agree with the classical results presented, e.g., in \cite{Barrett.Liu:94} in the context of classical finite element discretizations. We demonstrate the practical capabilities of the proposed formulation in a wide set of numerical experiments.  On the one hand, the numerical results confirm the theoretical estimates, and on the other hand, we show that, despite not being covered by our theoretical analysis, the proposed formulation can be successfully employed also in the case $\delta\not=0$ in \eqref{eq:Carreau}. \\

The rest of the manuscript is organized as follows. 
The next section introduces the notation we are going to employ throughout the manuscript. The weak formulation of the model problem together with its well-posedness is discussed in Section~\ref{sec:model}. The proposed divergence-free Virtual Element discretization is described in Section~\ref{sec:VEM} and the a-priori error analysis is carried out in Section~\ref{sec:error_analysis}.
A wide set of numerical experiments validating the theoretical bounds as well as showing in practice the capabilities of the proposed formulation are presented in Section~\ref{sec:numerics}. Finally, Appendix \ref{sec:appendix} contains some technical results. 

\subsection{Notation}
The vector spaces considered hereafter are over $\mathbb{R}$. We denote by $\mathbb{R}_{+}$ the set of non-negative real numbers. Given a vector space $V$ with norm $\norm[V]{\cdot}$, the notation $V'$ denotes its dual space and ${}_{V'}\langle\cdot,\cdot\rangle_V$ the duality between $V$ and $V'$. The notation $\b v\cdot\b w$ and $\b v\times\b w$ designate the scalar and vector products of two vectors $\b v,\b w\in\mathbb{R}^d$, and $\vert \b v \vert$ denotes the Euclidean norm of $\b v$ in  $\mathbb{R}^d$.
The inner product in $\mathbb{R}^{d\times d}$ is defined for $\b\tau,\b\eta\in\mathbb{R}^{d\times d}$ by $\b\tau:\b\eta\coloneq\sum_{i,j=1}^d \tau_{i,j}\eta_{i,j}$ and the induced norm is given by $|\b\tau|:= \sqrt{\b\tau:\b\tau}$.

Let $\Omega\subset\mathbb{R}^d$ denote a bounded, connected, polyhedral open set with Lipschitz boundary $\partial\Omega$ and let $\b n$ be the outward unit normal to $\partial\Omega$. To simplify the exposition, we restrict the presentation to the two-dimensional case, i.e. $d=2$, but the analysis path remains valid in the three-dimensional case $d=3$ as well, with obvious minor technical differences.
We denote with $\b x := (x_1, x_2)$ the independent variable.
We assume that the boundary is partitioned in two disjoint subsets $\partial\Omega:=\Gamma_D\cup\Gamma_N$, 
with $|\Gamma_D|>0$, such that a Dirichlet condition is given on $\Gamma_D$ and a Neumann condition on $\Gamma_N$. 

Throughout the article, spaces of functions, vector fields, and tensor fields, defined over any $X\subset\overline{\Omega}$ are denoted by italic capitals, boldface Roman capital, and special Roman capitals, respectively. The subscript $\rm{s}$ 
%appended to a special  capital 
denotes a space of symmetric tensor fields. For example, $L^2(X)$, $\b L^2(X)$, and $\mathbb{L}^2_s(X)$ denote the spaces of square-integrable functions, vector fields, and symmetric tensor fields, respectively. 
The notation $W^{m,r}(X)$, for $m \geq 0$ and $r\in[1,+\infty]$, with the convention that $W^{0,r}(X) :=L^r(X)$ and $W^{m,2}(X) := H^m(X)$, designate the classical Sobolev spaces.  
%In the following we will also make use of the shorthand notation $\norm[r,X]{{\cdot}}$ to denote the norm in $L^r(X)$.
%We adopt the convention that the subscript $X$ can be omitted in the case $X=\Omega$. The same notation is used for the vector- and tensor-valued spaces.
The trace map is denoted by $\gamma:W^{1,r}(\Omega)\to W^{1-\frac1r,r}(\partial\Omega)$.
Finally, given $\Gamma\subset\partial\Omega$, we denote by $W^{1,r}_{0,\Gamma}(\Omega)$ the subspace of $W^{1,r}(\Omega)$ spanned by functions having zero-trace on $\Gamma$.

The symbol $\nabla$ denotes the gradient for scalar functions, while  
$\GRAD$, $\b{\epsilon} := \frac{\GRAD + \GRAD^{\rm T}}2$ , and $\divs$ denote  the gradient, the symmetric gradient operator, and the divergence operator, whereas $\DIV$ denotes the vector-valued divergence operator for tensor fields.

%----------------------------------------------------------------------------------------------------------------
\section{Model problem}\label{sec:model}

The incompressible flow of a non-Newtonian fluid occupying $\Omega$ and subjected to a volumetric force field $\b{f}:\Omega\to\mathbb{R}^d$ and a normal stress $\b{g}:\Gamma_N\to\mathbb{R}^d$ is described by the non-linear Stokes equations
\begin{equation}\label{eq:stokes.continuous}
  \begin{aligned} 
    -\DIV\b{\sigma}(\cdot,\b{\epsilon}(\b{u}))+\GRAD p &= 
    \b{f} &\qquad& \text{ in } \Omega,  \\
    \divs \b{u} &= 0 &\qquad& \text{ in }  \Omega, \\
    \b{\sigma}(\cdot,\b\epsilon(\b{u}))\b{n} - p\b{n} &= \b{g} &\qquad& \text{ on } \Gamma_N, \\
    \b{\gamma}(\b{u}) &= \b{0} &\qquad& \text{ on } \Gamma_D, 
  \end{aligned}
\end{equation} 
where $\b{u}:\Omega\to\mathbb{R}^d$ and $p:\Omega\to\mathbb{R}$ denote the velocity field and the pressure field, respectively. The homogeneous Dirichlet boundary condition \eqref{eq:stokes.continuous} can be generalized to non-homogeneous data by minor modifications.

In this work, we consider as a reference model for the non-linear shear stress-strain rate relation the Carreau--Yasuda model introduced in \cite{Yasuda.Armstrong.ea:81}, i.e.
\begin{equation}\label{eq:Carreau}
\b{\sigma}(\b{x},\GRADs(\b{v})) = 
\mu(\b x) (\delta^\alpha + |\GRADs(\b{v})|^\alpha)^{\frac{r-2}\alpha} \GRADs(\b{v}),
\end{equation}
where $\mu:\Omega\to[\mu_{-}, \mu_{+}]$, with $0<\mu_{-}<\mu_{+}<+\infty$, $\alpha\in [1,\infty)$, and $\delta\ge 0$. 
The Carreau--Yasuda law is a generalization of the Carreau model which corresponds to the case $\alpha=2$.
The case $\delta = 0$ corresponds to the classical power-law model. Most real fluids that can be described by a constitutive relation of type \eqref{eq:Carreau} exhibits shear thinning behavior corresponding to the case $r<2$. For $r=2$, problem \eqref{eq:stokes.continuous} reduces to the standard Stokes system for Newtonian fluids. For the sake of conciseness, in what follows, we only consider the pseudoplastic case $r\le 2$ that is the most common in practical applications and the one presenting more challenges in terms of numerical analysis. Nevertheless, the following arguments can be easily adapted to establish the results for the dilatant case $r>2$.

According to \cite[Appendix A]{Botti.Castanon-Quiroz.ea:21}, the stress-strain law \eqref{eq:Carreau} satisfy the following assumption:
\begin{assumption}\label{ass:hypo}
The shear stress-strain rate law $\b\sigma:\Omega\times\mathbb{R}^{d\times d}_{\rm s} \to \mathbb{R}^{d\times d}_{\rm s}$ appearing in \eqref{eq:stokes.continuous} is a Caratheodory function satisfying $\b{\sigma}(\cdot,\b{0})=\b{0}$ and for a fixed $r\in(1,2]$ there exist real numbers $\delta \in [0,+\infty)$ and $\sigma_{\rm{c}}, \sigma_{\rm{m}}\in(0,+\infty)$ such that the following conditions hold:
\begin{subequations}\label{eq:hypo}
  \begin{alignat}{2} 
  \label{eq:hypo.continuity}
  &|\b{\sigma}(\b{x},\b{\tau})-\b{\sigma}(\b{x},\b{\eta})| \le \sigma_{\rm c} \left(\delta^r + |\b\tau|^r + |\b\eta|^r \right)^\frac{r-2}{r}| \b\tau-\b\eta|,
  &\qquad&\text{(H\"older continuity)}
  \\
  \label{eq:hypo.monotonicity}
  &\left(\b{\sigma}(\b x,\b\tau)-\b{\sigma}(\b x,\b\eta)\right):\left(\b\tau-\b\eta\right) \ge \sigma_{\rm m}\left(\delta^r+|\b\tau|^r + |\b\eta|^r \right)^\frac{r-2}{r}|\b\tau-\b\eta|^{2},
  &\qquad& \text{(strong monotonicity)}
  \end{alignat}
\end{subequations}
for almost every $\b{x}\in\Omega$ and all $\b{\tau},\b{\eta}\in\mathbb{R}^{d\times d}_{\rm{s}}$.
\end{assumption}
We observe that the positive constants $\sigma_{\rm{c}}, \sigma_{\rm{m}}$ in \eqref{eq:hypo} for the Carreau-Yasuda model \eqref{eq:Carreau} depends on $\mu,\alpha$ and $r$ and are such that 
$$
\mu_{-}(r-1) 2^{\frac{(r-1)(r-2)}{r}} \le  \sigma_{\rm{m}} \le \sigma_{\rm{c}} \le \frac{\mu_{+}}{r-1} 
2^{\frac{2r(2-r)+1}{r}}.
$$
Finally, for further use, we recall the strong monotonicity bound
   \begin{equation}\label{eq:bound_hyp_model}
       \vert \b x- \b y\vert^2 (\delta^r+ \vert x\vert^r+\vert y \vert^r)^{\frac{r-2}{r}} \lesssim
       \left\{ 
       (\delta^r + \vert \b x\vert^r)^{\frac{r-2}{r}}\b x - 
       (\delta^r + \vert \b y\vert^r)^{\frac{r-2}{r}}\b y
       \right\} \cdot (\b x -\b y)
   \end{equation}
   for $\b x, \b y\in \mathbb{R}^n$ and $\delta\geq 0$.
Here and in the following, to avoid the proliferation of constants, we adopt the notation $\mathsf{a} \lesssim  \mathsf{b}$ to denote the inequality $\mathsf{a} \leq C \mathsf{b}$, for a positive constant $C$ that might depend on $\sigma_{\rm{c}}, \sigma_{\rm{m}}$ (or related parameters) in Assumption~\ref{ass:hypo} and on $r$, but is independent of the discretization parameter $h$.
The obvious extensions $\mathsf{a} \gtrsim  \mathsf{b}$ and $\mathsf{a} \simeq  \mathsf{b}$ hold.

\subsection{Weak formulation}

Before deriving the variational formulation of problem \eqref{eq:stokes.continuous} and discussing its well-posedness, we introduce additional notation and recall some basic results concerning Sobolev spaces. First of all, the expressions of the conjugate index and the Sobolev index are given by 
$$
r':= \begin{cases} \frac{r}{r-1} &\text{if } r>1, \\ +\infty  &\text{if } r= 1, \end{cases} \qquad 
r^* := \begin{cases} \frac{dr}{d-r} &\text{if } r<d, \\ +\infty  &\text{if } r\ge d, \end{cases} 
$$
respectively.
From the classical Sobolev embedding theorems \cite[Section 9.3]{Brezis:10} it is inferred that $W^{1,r}(\Omega)\subset L^q(\Omega)$ for all $q\in[1,r^*]$ (excluding $\{+\infty\}$ if $r=d$) and the embedding is compact for $q<r^*$. 
We also recall Korn's first inequality (see, e.g., \cite[Theorem 1.2]{Ciarlet.Ciarlet:05} and \cite[Theorem 1]{Geymonat.Suquet:86}) that will be needed in the analysis: there is $C_{\rm K} > 0$ depending only on $\Omega$ and $r$ such that for all $\b v \in \b{W}^{1,r}_{0,\Gamma_D}(\Omega)$,
\begin{equation}\label{eq:Korn}
\|\b v\|_{\b{W}^{1,r}(\Omega)} \le C_{\rm K} \|\GRADs (\b v)\|_{\mathbb{L}^r(\Omega)}.
\end{equation}

From this point on, we omit both the integration variable and the measure from integrals, as they can be in all cases inferred from the context. Let $r\in(1,2]$ be the Sobolev exponent dictated by the non-linear stress-strain law characterizing problem \eqref{eq:stokes.continuous} and satisfying Assumption~\ref{ass:hypo}.
We define the following velocity and pressure spaces incorporating the homogeneous boundary condition on $\Gamma_D$ and the zero-average constraint in the case $\Gamma_D:=\partial\Omega$, respectively: 
\[
\b U := \b{W}^{1,r}_{0,\Gamma_D}(\Omega) %\coloneqq \left\{\b v \in W^{1,r}(\Omega,\mathbb{R}^d)\ : \ \b v_{|{\Gamma_D}} = \b 0 \right\},
\qquad
P = \begin{cases} L^{r'}(\Omega)&\text{if } |\Gamma_D|<|\partial\Omega| \\ L^{r'}_0(\Omega) \coloneqq \left\{q \in L^{r'}(\Omega)\ : \ \textstyle\int_\Omega q = 0 \right\} &\text{if } \Gamma_D =\partial\Omega. \end{cases} 
\]
Assuming $\b f \in \b{L}^{r'}(\Omega)$ and $\b g \in \b{L}^{r'}(\Gamma_N)$, the weak formulation of problem \eqref{eq:stokes.continuous} reads:
Find $(\b u,p) \in \b U \times P$ such that
\begin{equation}\label{eq:stokes.weak}
  \begin{aligned}
     a(\b u,\b v)+b(\b v,p) &= \int_\Omega \b f \cdot \b v + \int_{\Gamma_N} \b g \cdot \b\gamma(\b v)&\qquad \forall \b v \in \b U, \\
     -b(\b u,q) &= 0 &\qquad \forall q \in P,   
  \end{aligned}
\end{equation}
where the function $a : \b U \times \b U \to \mathbb{R}$ and the bilinear form $b : \b U \times P \to \mathbb{R}$ are defined  for all $\b v,\b w \in \b U$ and all $q \in L^{r'}(\Omega)$ by
\begin{equation}\label{eq:a.b}
  a(\b w,\b v) \coloneqq \int_\Omega \b\sigma(\cdot,\b\epsilon(\b w)) : \b\epsilon(\b v),\qquad
  b(\b v,q)  \coloneqq -\int_\Omega (\divs \b v) q.
\end{equation}

Let us introduce the kernel of the  bilinear form $b(\cdot,\cdot)$
that corresponds to the functions in $\b U$ with vanishing divergence, i.e.
\begin{equation*}
\b Z := \{ \b v \in \b U \quad \text{s.t.} \quad \divs  \b v = 0  \}\,.
\end{equation*}
Then, Problem~\eqref{eq:stokes.weak} can be formulated in the equivalent kernel form:
Find $\b u \in \b Z$ such that
\begin{equation}\label{eq:stokes.weak.Z}
     a(\b u, \b v) = \int_\Omega \b f \cdot \b v + \int_{\Gamma_N} \b g \cdot \b\gamma(\b v) \qquad \forall \b v \in \b Z \,.
\end{equation}

\subsection{Well-posedness}

In this section, after reporting the properties of the viscous function $a(\cdot,\cdot)$ and the coupling bilinear form $b(\cdot,\cdot)$ defined in \eqref{eq:a.b}, we prove the well-posedness of problem \eqref{eq:stokes.weak}.  
\begin{lemma}[Continuity and strong monotonicity of $a$]
%\label{lemm:properties_a_continuous}
For all $\b u, \b v, \b w \in \b U$, setting $\b e := \b u - \b v$, it holds
\begin{subequations}%\label{eq:a_properties}
  \begin{alignat}{2} 
  \label{eq:a.continuity}
  |a(\b u, \b v) - a(\b w, \b v)| &\le \sigma_{\rm c}
  \norm[\b{W}^{1,r}(\Omega)]{\b e}^{r-1} \norm[\b{W}^{1,r}(\Omega)]{\b v}
  %\left(\delta^r +\norm[\b{W}^{1,r}(\Omega)]{\b u}^r + \norm[\b{W}^{1,r}(\Omega)]{\b w}^r\right)^{\frac{r-2}r} \norm[\b{W}^{1,r}(\Omega)]{\b e} \norm[\b{W}^{1,r}(\Omega)]{\b v},
  \\
  \label{eq:a.monotonicity}
  a(\b u, \b e) - a(\b w, \b e) &\gtrsim \sigma_{\rm m} \left(\delta^r +\norm[\b{W}^{1,r}(\Omega)]{\b u}^r + \norm[\b{W}^{1,r}(\Omega)]{\b w}^r\right)^{\frac{r-2}r} \norm[\b{W}^{1,r}(\Omega)]{\b e}^2.
  \end{alignat}
\end{subequations}
\end{lemma}
\begin{proof} Let $\b u, \b v, \b w \in \b U$ and set $\b e := \b u - \b w$.

\noindent
\emph{(i) H\"older continuity.}
First, we make a preliminary observation. For every $\b{\tau},\b{\eta}\in\mathbb{R}^{d\times d}_{\rm{s}}$ the triangle inequality implies that
$$
2^{1-r} \vert\b\tau -\b\eta\vert^r \le 
\vert\b\tau\vert^r +\vert\b\eta\vert^r.
$$
Therefore, since $\delta\ge0$ and $r-2<0$ we have
\begin{equation}\label{eq:pre_holdercont}
\left(\delta^r +\vert\b\tau\vert^r +\vert\b\eta\vert^r \right)^{\frac{r-2}{r}}
\le \left(2^{1-r} \vert\b\tau -\b\eta\vert^r\right)^{\frac{r-2}{r}}
\le 2^{\frac{(1-r)(r-2)}{r}} \vert\b\tau -\b\eta\vert^{r-2}
\lesssim \vert\b\tau -\b\eta\vert^{r-2}. 
\end{equation}
Recalling the H\"older continuity property \eqref{eq:hypo.continuity} and using \eqref{eq:pre_holdercont} followed by the H\"older inequality with exponents $(r', r)$, we have (here ${\b e}:= {\b u} - {\b w}$)
\begin{equation}
\label{eq:for-holder}
\begin{aligned}
|a(\b u, \b v) - a(\b w, \b v)| &\le
\int_\Omega \left|[\b\sigma(\cdot,\b\epsilon(\b u)) - \b\sigma(\cdot,\b\epsilon(\b w)) ] : \b\epsilon(\b v)  \right| \\
&\le 
\sigma_{\rm c} \int_\Omega \left(\delta^r + |\b\epsilon(\b u)|^r + |\b\epsilon(\b w)|^r \right)^\frac{r-2}{r} |\b\epsilon(\b e)| |\b\epsilon(\b v)| \\
&\lesssim 
\sigma_{\rm c} \int_\Omega |\b\epsilon(\b e)|^{r-1} |\b\epsilon(\b v)|  
\\
%&\le
%\sigma_{\rm c} \left(\int_\Omega \delta^r + |\b\epsilon(\b u)|^r + %|\b\epsilon(\b w)|^r \right)^\frac{r-2}{r} \norm[\b{W}^{1,r}(\Omega)]{\b e} \norm[\b{W}^{1,r}(\Omega)]{\b v}\\
&\lesssim \sigma_{\rm c} \norm[\b{W}^{1,r}(\Omega)]{\b e}^{r-1} \norm[\b{W}^{1,r}(\Omega)]{\b v}.
\end{aligned}
\end{equation}
\emph{(ii) Strong monotonicity.} 
Using the Korn's inequality \eqref{eq:Korn} together with the monotonicity property of $\b \sigma$ in \eqref{eq:hypo.monotonicity} and the H\"older inequality with exponents $(\frac2{2-r},\frac2r)$, we obtain
$$
\begin{aligned}
\sigma_{\rm m} \norm[\b{W}^{1,r}(\Omega)]{\b e}^2 &\le C_{\rm K}^2
\sigma_{\rm m} \left(\int_\Omega \vert \b\epsilon(\b u-\b w) \vert^r\right)^\frac2r \lesssim
\sigma_{\rm m} \left(\int_\Omega \left(\vert\b\epsilon(\b u) -\b\epsilon(\b w)\vert^2\right)^\frac{r}2 \right)^\frac2r
\\
&\lesssim
\left(
\int_\Omega \left(\delta^r+|\b\epsilon(\b u)|^r + |\b\epsilon(\b w)|^r \right)^\frac{2-r}{2}
\bigl((\b\sigma(\cdot,\b\epsilon(\b u)) - \b\sigma(\cdot,\b\epsilon(\b w)) ) : \b\epsilon(\b u - \b w) \bigr)^\frac{r}2 \right)^\frac2r \\
&\lesssim 
\left(\int_\Omega \bigl(\delta^r+|\b\epsilon(\b u)|^r + |\b\epsilon(\b w)|^r \bigr) \right)^\frac{2-r}r
\left(\int_\Omega (\b\sigma(\cdot,\b\epsilon(\b u)) - \b\sigma(\cdot,\b\epsilon(\b w)) ) : \b\epsilon(\b u - \b w) \right) \\
&\lesssim
\left(|\Omega|\delta^r +\norm[\b{W}^{1,r}(\Omega)]{\b u}^r + \norm[\b{W}^{1,r}(\Omega)]{\b w}^r \right)^\frac{2-r}{r} 
(a(\b u, \b e) - a(\b w, \b e)).
\end{aligned}
$$
Rearranging the previous inequality, yields the conclusion.
\end{proof}

The following result is needed to infer the existence of a unique pressure $p\in P$ solving problem \eqref{eq:stokes.weak} from the well-posedness of problem \eqref{eq:stokes.weak.Z}. For its proof we refer to \cite[Theorem 1]{Bogovski:79}.
\begin{lemma}[Inf-sup condition]%\label{lem:inf_sup}
For any $r\in(1,\infty)$ there exists a positive constant $\beta(r)$ such that the bilinear form $b(\cdot,\cdot)$ defined in \eqref{eq:a.b} satisfies
\begin{equation}\label{eq:inf_sup}
\inf_{q\in P}\;\sup_{\b w\in \b{U}\setminus\{\b 0\}}\; 
\frac{b(\b w,q)}{\norm[L^{r'}(\Omega)]{q}\norm[\b{W}^{1,r}(\Omega)]{\b w}} \ge \beta(r) > 0.
\end{equation}
\end{lemma}

We are now ready to establish the well-posedness of problem \eqref{eq:stokes.weak}.
%\begin{proposition}[Well-posedness]\label{prop:well-posed_cont}
%For any $r\in(1,2)$, there exists a unique solution $(\b u, p)\in \b{U}\times P$ to problem problem \eqref{eq:stokes.weak} satisfying the a-priori estimates
%\begin{align}\label{eq:a-priori_cont.u}
%\norm[\b{W}^{1,r}(\Omega)]{\b u}&\lesssim
%\left(\| \b f \|_{L^{r'}(\Omega,\R^d)} +
%\| \b g \|_{L^{r'}(\Gamma_N,\R^d)}\right)^{\frac1{r-1}} + \delta.
%\\
%\label{eq:a-priori_cont.p}
%\norm[L^{r'}(\Omega)]{p} &\lesssim
%\| \b f \|_{L^{r'}(\Omega,\R^d)} +
%\| \b g \|_{L^{r'}(\Gamma_N,\R^d)}+ \delta^{r-1}.
%\end{align}
%\end{proposition}
%

\begin{proposition}[Well-posedness]\label{prop:well-posed_cont}
For any $r\in(1,2)$, there exists a unique solution $(\b u, p)\in \b{U}\times P$ to problem problem \eqref{eq:stokes.weak} satisfying the a-priori estimates
\begin{align}\label{eq:a-priori_cont.u}
    \norm[\b{W}^{1,r}(\Omega)]{\b u}&\lesssim
\delta^{2-r} \mathcal{N}(\b f, \b g) + \mathcal{N}(\b f, \b g)^{\frac{1}{r-1}}  \,,
\\
\label{eq:a-priori_cont.p}
\norm[L^{r'}(\Omega)]{p} &\lesssim
\mathcal{N}(\b f, \b g) + \delta^{(2-r)(r-1)} \mathcal{N}(\b f, \b g)^{r-1}  \,,
\end{align}
where 
\begin{equation}
\label{eq:Cfg}
\mathcal{N}(\b f, \b g)\coloneq 
2^{\frac{2-r}r} \sigma_{\rm m}^{-1}
\left(\| \b f \|_{\b L^{r'}(\Omega)} + \| \b g \|_{\b L^{r'}(\Gamma_N)}\right).
\end{equation}
\end{proposition}

\begin{proof}
We focus on the proof of the a-priori bounds \eqref{eq:a-priori_cont.u} and \eqref{eq:a-priori_cont.p}. For the uniqueness and existence we refer to \cite{Beirao-da-Veiga:09, Berselli.Diening.ea:10}. Using \eqref{eq:a.monotonicity} with $\b w = \b 0$ and taking $\b v = \b u$ in \eqref{eq:stokes.weak.Z}, it is inferred that
$$
\sigma_{\rm m} \left(\delta^r +\norm[\b{W}^{1,r}(\Omega)]{\b u}^r \right)^{\frac{r-2}r} \norm[\b{W}^{1,r}(\Omega)]{\b u}^2 \lesssim 
a(\b u, \b u) = \int_\Omega \b f \cdot \b u + \int_{\Gamma_N} \b g \cdot \b\gamma(\b u).
$$
Thus owing to the H\"older inequality and the continuity of the trace map, we obtain
\begin{equation}
    \label{eq:well0}
\sigma_{\rm m} \left(\delta^r +\norm[\b{W}^{1,r}(\Omega)]{\b u}^r \right)^{\frac{r-2}r} \norm[\b{W}^{1,r}(\Omega)]{\b u}^{\bcancel{2}} \lesssim 
\left(\| \b f \|_{\b L^{r'}(\Omega)} +
\| \b g \|_{\b L^{r'}(\Gamma_N)}\right)
\bcancel{\norm[\b{W}^{1,r}(\Omega)]{\b u}}.
\end{equation}
If $\norm[\b{W}^{1,r}(\Omega)]{\b u}\ge \delta$, from the previous bound it follows that
\begin{equation}\label{eq:a-priori.1}
\norm[\b{W}^{1,r}(\Omega)]{\b u}^{r-1} \lesssim
2^{\frac{2-r}r} \sigma_{\rm m}^{-1}
\left(\| \b f \|_{L^{r'}(\Omega,\R^d)} +
\| \b g \|_{L^{r'}(\Gamma_N,\R^d)}\right)
\coloneq \mathcal{N}(\b f, \b g).
%2^{\frac{2-r}r} \sigma_{\rm m}^{-1}
\end{equation}
Otherwise, owing to $\norm[\b{W}^{1,r}(\Omega)]{\b u}^r\ge \delta^r$ and \eqref{eq:well0} we obtain
\begin{equation}\label{eq:a-priori.2}
\norm[\b{W}^{1,r}(\Omega)]{\b u} \lesssim 
\delta^{2-r}\mathcal{N}(\b f, \b g).
\end{equation}
Therefore, from \eqref{eq:a-priori.1} and \eqref{eq:a-priori.2} we infer that
$$
\norm[\b{W}^{1,r}(\Omega)]{\b u} \lesssim 
\delta^{2-r} \mathcal{N}(\b f, \b g) + \mathcal{N}(\b f, \b g)^{\frac{1}{r-1}}.
$$

We now move to the estimate of the pressure. Owing to the inf-sup condition \eqref{eq:inf_sup} and equation \eqref{eq:stokes.weak}, it holds
$$
\beta(r) \norm[L^{r'}(\Omega)]{p} \le 
\sup_{\b v\in\b U\setminus\{\b 0\}}\frac{b(\b v, p)}{\norm[\b{W}^{1,r}(\Omega)]{\b v}} = 
\sup_{\b v\in\b U\setminus\{\b 0\}}\frac{\int_\Omega \b f \cdot \b v + \int_{\Gamma_N} \b g \cdot \b\gamma(\b v) - a(\b u, \b v)}{\norm[\b{W}^{1,r}(\Omega)]{\b v}}.
$$
Applying the H\"older inequality, the continuity of $a$ in \eqref{eq:a.continuity} with $\b{w}=\b0$, and the a-priori estimate of the velocity \eqref{eq:a-priori_cont.u}, we obtain 
$$
\begin{aligned}
\norm[L^{r'}(\Omega)]{p} &\lesssim
\| \b f \|_{\b L^{r'}(\Omega)} + \| \b g \|_{\b L^{r'}(\Gamma_N)} +
\sigma_{\rm c} \norm[\b{W}^{1,r}(\Omega)]{\b u}^{r-1}  \\ &\lesssim 
\sigma_{\rm c}\mathcal{N}(\b f, \b g) + \sigma_{\rm c} \delta^{(2-r)(r-1)} \mathcal{N}(\b f, \b g)^{r-1}.
\end{aligned}
$$
\end{proof}

%----------------------------------------------------------%

\section{Virtual Elements discretization}
\label{sec:VEM}

In the present section, we present the divergence-free Virtual Elements for non-Newtonian fluid flows.
In Subsection \ref{sub:preliminaries}, we introduce 
some basic tools and notations useful in the construction and the theoretical analysis of Virtual Element Methods.
In Subsections \ref{sub:v-p spaces}, we outline an overview of the inf-sup stable divergence--free velocities-pressures pair of spaces.
In Subsection \ref{sub:forms}, we define the discrete computable form, in particular we design a VEM stabilizing form that is suited in the proposed non-linear setting. 
Finally, in Subsection \ref{sub:vem problem}, we introduce the virtual elements discretization of Problems \eqref{eq:stokes.weak} and \eqref{eq:stokes.weak.Z} and we establish the well-posedness of the discrete problem.

%-------------------------------------------------------------------
\subsection{Preliminaries}
\label{sub:preliminaries}
Let $\{\Omega_h\}_h$ be a sequence of decompositions of the domain $\Omega \subset \R^2$ into general polytopal elements $E$ 
where $h := \sup_{E \in \Omega_h} h_{E}$. 
%
%We set $\mathcal{E}_h$ the set of the mesh edges.
%
We suppose that $\{\Omega_h\}_h$ fulfils the following assumption.
\begin{assumption}\label{ass:mesh}
\textbf{(Mesh assumptions).} There exists a positive constant $\rho$ such that for any $E \in \{\Omega_h\}_h$ 
\begin{itemize}
\item $E$ is star-shaped with respect to a ball $B_E$ of radius $ \geq\, \rho \, h_E$;
\item any edge $e$ of $E$ has length  $ \geq\, \rho \, h_E$.
\end{itemize}
\end{assumption}
We remark that the hypotheses above, though not too restrictive in many practical cases, could possibly be further relaxed, combining the present analysis with the studies in~\cite{BLR:2017,brenner-guan-sung:2017,brenner-sung:2018,BdV-Vacca:2022}.
Using standard VEM notations, for $n \in \N$ and  $m \in \R_+$ and $p=1, \dots, +\infty$
let us introduce the spaces:
\begin{itemize}
\item $\Pk_n(\omega)$: the set of polynomials on $\omega \subset \Omega$ of degree $\leq n$  (with $\Pk_{-1}(\omega)=\{ 0 \}$),
\item $\Pk_n(\Omega_h) := \{q \in L^2(\Omega) \quad \text{s.t} \quad q|_{E} \in  \Pk_n(E) \quad \text{for all $E \in \Omega_h$}\}$,
\item $W^{m,p}(\Omega_h) := \{v \in L^p(\Omega) \quad \text{s.t} \quad v|_{E} \in  W^{m,p}(E) \quad \text{for all $E \in \Omega_h$}\}$,
\end{itemize}
equipped with the broken norm and seminorm
\begin{equation}
\label{eq:normebroken}
\begin{aligned}
\norm[W^{m,p}(\Omega_h)]{v}^p
&:= \sum_{E \in \Omega_h} \norm[W^{m,p}(E)]{v}^p\,,
&\quad
|v|^p_{W^{m,p}(\Omega_h)} &:= \sum_{E \in \Omega_h} |v|^p_{W^{m,p}(E)}\,, 
&\qquad &\text{if $1 \leq p < \infty$.}
%\\
%\|v\|_{\b W^{m,\infty}(\Omega_h)} &:= \sup_{E \in \Omega_h} \|v\|_{\b W^{m,\infty}(E)}\,,
%&\quad
%|v|_{\b W^{m,\infty}(\Omega_h)} &:= \sup_{E \in \Omega_h} |v|_{\b W^{m,\infty}(E)}\,, 
%&\qquad &\text{if $p= \infty$.}
\end{aligned}
\end{equation}
Let $E \in \Omega_h$, we denote with $h_{E}$ the diameter, with $|E|$ the area, with $\b x_{E} := (x_{E, 1}, x_{E, 2})$ the centroid. 
A natural basis associated with the space $\Pk_n(E)$ is the set of normalized
monomials
\begin{equation*}
\M_n(E) := \left\{ 
m_{\boldsymbol{\alpha}},
\,\,\,  \text{with} \,\,\,
|\boldsymbol{\alpha}| \leq n
\right\}
\end{equation*}
where, for any multi-index $\boldsymbol{\alpha} := (\alpha_1, \alpha_2) \in \N^2$
\[
m_{\boldsymbol{\alpha}} :=
\prod_{i=1}^2  
\left(\frac{x_i - x_{E, i}}{h_{E}} \right)^{\alpha_i}
\qquad \text{and} \qquad
|\boldsymbol{\alpha}|:= \sum_{i=1}^2 \alpha_i \,.
\]
For any $e$ edge of $\Omega_h$, the normalized monomial set $\M_n(e)$ is defined analogously as the span of all one-dimensional normalized monomials of degree up to $n$.
Moreover for any $m \leq n$ we denote with
\[
\Pkw_{n \setminus m}(E) := {\rm span}  
\left\{ 
m_{\boldsymbol{\alpha}},
\,\,\,  \text{with} \,\,\,
m +1 \leq |\boldsymbol{\alpha}| \leq n
\right\} \,.
\]

%\noindent 
%Let $\omega \subset \Omega$, for any $m \leq n$, 
%we denote with $\Pkw_{n \setminus m}(\omega) \subset \Pk_m(\omega)$ the algebraic complement of $\Pk_m(\omega)$ in $\Pk_n(\omega)$, i.e.
%$\Pk_n(\omega) = \Pk_m(\omega) \oplus \Pkw_{n \setminus m}(\omega)$.

\noindent
For any $E$,  
let us define the following polynomial projections:
\begin{itemize}
\item the $L^2$-projection $\Pi_n^{0, E} \colon L^2(E) \to \Pk_n(E)$, given by
\begin{equation}
\label{eq:P0_k^E}
\int_{E} q_n (v - \, {\Pi}_{n}^{0, E}  v) \, {\rm d} E := 0 \qquad  \text{for all $v \in L^2(E)$  and $q_n \in \Pk_n(E)$,} 
\end{equation} 
with obvious extension for vector functions $\Pi^{0, E}_{n} \colon \b L^2(E) \to [\Pk_n(E)]^d$ and 
tensor functions $\boldsymbol{\Pi}^{0, E}_{n} \colon \mathbb{L}^2(E) \to [\Pk_n(E)]^{d\times d}$;

\item the $H^1$-seminorm projection ${\Pi}_{n}^{\nabla,E} \colon H^1(E) \to \Pk_n(E)$, defined by 
\begin{equation*}
%\label{eq:Pn_k^E}
\left\{
\begin{aligned}
& \int_{E} \nabla  \,q_n \cdot \nabla ( v - \, {\Pi}_{n}^{\nabla,E}   v)\, {\rm d} E = 0 \quad  \text{for all $v \in H^1(E)$ and  $q_n \in \Pk_n(E)$,} \\
& \int_{\partial E}(v - \,  {\Pi}_{n}^{\nabla, E}  v) \, {\rm d}s:= 0 \, ,
\end{aligned}
\right.
\end{equation*}
with extension for vector fields $\Pi^{\nabla, E}_{n} \colon \b H^1(E) \to [\Pk_n(E)]^d$.
\end{itemize}
Further we define the operator $\Pi^0_{n}\colon L^2(\Omega) \to \Pk_n(\Omega_h)$ such that $\Pi^0_{n}\vert_E:=\Pi^{0,E}_n$ for any $E \in \Omega_h$.
%%
% In the following the symbol $\lesssim$ will denote a bound up to a generic positive constant,
% independent of the mesh size $h$, but which may depend on 
% $\Omega$, on the ``polynomial'' order of the
% method $k$, and on the mesh regularity constant appearing in Assumption~\ref{ass:mesh}.
%%
We finally recall the following well know useful results:
\begin{itemize}
\item Trace inequality with scaling \cite{brenner-scott:book}:
For any $E \in \Omega_h$ and for any  function $v \in W^{1,r}(E)$ it holds 
\begin{equation}
\label{eq:cont_trace}
\|v\|^r_{L^r(\partial E)} \lesssim h_E^{-1}\|v\|^r_{L^r(E)} + h_E^{r-1} \|\nabla v\|^r_{\b{L}^r(E)} \,.
\end{equation}

\item Polynomial inverse estimate \cite[Theorem 4.5.11]{brenner-scott:book}: 
Let $1 \leq q,p \leq \infty$, then for any $E \in \Omega_h$
\begin{equation}
\label{eq:inverse}
\Vert  p_n \Vert_{L^q(E)} \lesssim h_E^{2/q - 2/p} \Vert  p_n \Vert_{L^p(E)}
\quad \text{for any $p_n \in \Pk_n(E)$.}
\end{equation}

%\item Bramble-Hilbert Lemma \cite[Lemma 4.3.8]{brenner-scott:book}: Let $0\leq t \leq s \leq n+1$, and $1 \leq q,p \leq \infty$ such that $s - 2/p > t - 2/q$, then for any $E \in \Omega_h$
%\begin{equation}
%\label{eq:bramble}
%\vert  \phi - \Pi^{0,E}_n \phi \vert_{W^t_q(E)} \lesssim 
%h_E^{(s-t) +2\frac{p-q}{pq}} \vert  \phi \vert_{W^s_p(E)}
%\quad \text{for any $\phi \in W^s_p(E)$.}
%\end{equation}
\end{itemize}

% -----------------------------------------------------------------
\subsection{Virtual Elements velocity and pressure spaces}
\label{sub:v-p spaces}

Let $k \geq 2$ be the  polynomial order of the method.
We consider on each polygonal element $E \in \Omega_h$ the ``enhanced'' virtual space \cite{BLV:2017,vacca:2018}:
\begin{equation}
\label{eq:v-loc}
\begin{aligned}
\VDL := \biggl\{  
\b v_h \in [C^0(\overline{E})]^2 \,\,\, \text{s.t.} \,\,\,
(i)& \,\,  
  \boldsymbol{\Delta}    \b v_h  +  \nabla s \in \b x ^\perp \Pk_{k-1}(E), 
\,\,\text{ for some $s \in L_0^2(E)$,} 
\\
(ii)&  
\,\,   \divs \, \b v_h \in \Pk_{k-1}(E) \,, 
\\
(iii) &
\,\,  {\b v_h}_{|e} \in [\Pk_k(e)]^2 \,\,\, \forall e \in \partial E, 
\\
(iv) &
\,\,   (\b v_h - \Pi^{\nabla,E}_k \b v_h, \, \b x ^\perp \, \widehat{p}_{k-1} )_E = 0
\,\,\, \text{$\forall \widehat{p}_{k-1} \in \widehat{\Pk}_{k-1 \setminus k-3}(E)$}
\,\,\biggr\} \,,
\end{aligned}
\end{equation}
where $\b x ^\perp := (x_2, -x_1)$.
Next, we summarize the main properties of the space $\VDL$
(we refer to \cite{BLV:2018} for a deeper analysis).

\begin{itemize}
\item [\textbf{(P1)}] \textbf{Polynomial inclusion:} $[\Pk_k(E)]^2 \subseteq \VDL$;
\item [\textbf{(P2)}] \textbf{Degrees of freedom:}
the following linear operators $\mathbf{D_{\boldsymbol{U}}}$ constitute a set of DoFs for $\VDL$:
\begin{itemize}
\item[$\mathbf{D_{\boldsymbol{U}}1}$] the values of $\b v_h$ at the vertexes of the polygon $E$,
\item[$\mathbf{D_{\boldsymbol{U}}2}$] the values of $\b v_h$ at $k-1$ distinct points of every edge $e \in \partial E$,
\item[$\mathbf{D_{\boldsymbol{U}}3}$] the moments of $\b v_h$ 
$$
\frac{1}{|E|}\int_E  \b v_h \cdot \boldsymbol{m}^{\perp} m_{\boldsymbol{\alpha}} \, {\rm d}E  
\qquad 
\text{for any $m_{\boldsymbol{\alpha}} \in \M_{k-3}(E)$,}
$$
where $\boldsymbol{m}^{\perp} := \frac{1}{h_E}(x_2 - x_{2,E}, -x_1 + x_{1,E})$,
\item[$\mathbf{D_{\boldsymbol{U}}4}$] the moments of $\divs \b v_h$ 
$$
\frac{h_E}{|E|}\int_E (\divs \b v_h) \, m_{\boldsymbol{\alpha}} \, {\rm d}E  
\qquad \text{for any $m_{\boldsymbol{\alpha}} \in \M_{k-1}(E)$ with $|\boldsymbol{\alpha}| > 0$;}
$$
\end{itemize}
\item [\textbf{(P3)}] \textbf{Polynomial projections:}
the DoFs $\mathbf{D_{\boldsymbol U}}$ allow us to compute the following linear operators:
\[
%\PN \colon \VDL \to [\Pk_k(E)]^2, \qquad
\P0 \colon \VDL \to [\Pk_{k}(E)]^2, \qquad
\PP0 \colon \GRAD \VDL \to [\Pk_{k-1}(E)]^{2 \times 2} \,.
\]
\end{itemize} 

The global velocity space $\VDG$ is defined by gluing the local spaces with the obvious
associated sets of global DoFs: 
\begin{equation}
\label{eq:v-glo}
\VDG := \{\b v_h \in \VCG \quad \text{s.t.} \quad {\b v_h}|_E \in \VDL \quad \text{for all $E \in \Omega_h$} \} \,.
\end{equation}
The discrete pressure space $\QDG$ is given by the piecewise polynomial functions of degree $k-1$, i.e.
\begin{equation}
\label{eq:q-glo}
\QDG := \{q_h \in \QCG \quad \text{s.t.} \quad {q_h}_|E \in \Pk_{k-1}(E) \quad \text{for all $E \in \Omega_h$} \}\,.
\end{equation}
The couple of spaces $(\VDG, \, \QDG)$ is well known to be inf-sup stable in the classical Hilbertian setting. The inf-sup stability for $r \neq 2$ is proven below in Section~\ref{sec:infsup}.
Let us introduce the discrete kernel
\begin{equation}
\label{eq:z-glo}
\ZDG := \{ \b v_h \in \VDG \quad \text{s.t.} \quad  b(\b v_h, q_h) = 0 \quad \text{for all $q_h \in \QDG$}\}
\end{equation}
then recalling $(ii)$ in \eqref{eq:v-loc} and \eqref{eq:q-glo}, the following kernel inclusion holds
\begin{equation*}
%\label{eq:kernel}
\ZDG \subseteq \ZCG \,,
\end{equation*}
i.e. the functions in the discrete kernel are exactly divergence-free.

%We close this section reviewing the following result on polynomial approximation (see for instance \cite{brenner-scott:book}) and Virtual Element interpolation (see \cite{MBM_23}). 
%%
Uniquely for the purpose of defining our interpolant in the space $\VDG$, we consider also the alternative set of edge degrees of freedom (which can substitute $\mathbf{D_{\boldsymbol{U}}2}$)
\begin{itemize}
\item[$\mathbf{D_{\boldsymbol{U}}2'}$] the moments of $\b v_h$
\begin{equation}
\label{eq:per-infsup2}
\frac{1}{|e|} \int_e {\b v_h} \cdot {\boldsymbol{t}}_e m_\alpha \,{\rm d}s \,, \qquad 
\frac{1}{|e|} \int_e {\b v_h} \cdot {\boldsymbol{n}}_e m_\alpha \,{\rm d}s \qquad 
\text{for any $m_\alpha$} \in \M_{k-2}(e) \, ,
\end{equation}
where $ {\boldsymbol{t}}_e$ and  ${\boldsymbol{n}}_e$ denote the tangent and the normal vectors to the edge $e$ respectively.
\end{itemize}
Given any $\b v \in \b W^{s,p}(E)$, with $p \in (1,\infty)$ and $s \in {\mathbb R}_+$, $s>2/p$, we define its approximant 
${\b v_I} \in \VDG$ as the unique function in $\VDG$ that interpolates $\b v$ with respect to the DoF set  
$\mathbf{D_{\boldsymbol{U}}1}$, $\mathbf{D_{\boldsymbol{U}}2}'$, $\mathbf{D_{\boldsymbol{U}}3}$, $\mathbf{D_{\boldsymbol{U}}4}$.
It is easy to check that, whenever $\textrm{div} \b v = 0$, then $\b v_I \in \ZDG$.
Furthermore, the following approximation property is a trivial generalization of the results in \cite{MBM_23}.

\begin{lemma}\label{lem:approx-interp}
Let $E \in \Omega_h$, $n \in \mathbb{N}$, $p \in [1,\infty]$, $s \in {\mathbb R}_{+}$ and $\b v \in \b W^{s,p}(E)$.
%Let ${\boldsymbol{\Pi}^{0, E}_{n}}$ denote the classical $L^2(E)$ projection operator onto $[\Pk_{n}(E)]^2$.
For $s>2/p$, let ${\b v_I} \in \VDG$ be the interpolant of $\b v$ defined above. It holds
$$
\begin{aligned}
& | \b v - \Pi^{0, E}_{n} \b v |_{\b W^{m,p}(E)} \lesssim h_E^{s-m} |\b v|_{\b W^{s,p}(E)} 
&\quad &\text{for $0 \le m \le s \le n+1$,} \\
& | \b v - \b v_I |_{\b W^{m,p}(E)} \lesssim h_E^{s-m} |\b v|_{\b W^{s,p}(E)} 
&\quad  &\text{for $2/p < s \le k+1 , \ m \in \{0,1\}$.}
\end{aligned}
$$
The first bound above extends identically to the scalar and tensor-valued case.
\end{lemma}

% -----------------------------------------------------------------
\subsection{Virtual Element form: stabilization and inf-sup condition}
\label{sub:forms}

The next step in the construction of the method is the definition of a discrete version of the form $a(\cdot, \cdot)$ in \eqref{eq:a.b} and the approximation of the right-hand side of \eqref{eq:stokes.weak}.
In the present analysis, the main issue is the design of a VEM stabilizing form that is suited for the non-linearity under consideration.
Following the usual procedure in the VEM setting, we need to construct discrete forms that are computable employing the DoF values only.
In the light of property \textbf{(P3)} we define the computable discrete form

\[
a_h^E(\b v_h,  \b w_h) := 
\int_E \b\sigma(\cdot, \PP0 \b \epsilon(\b v_h)) : \PP0 \b\epsilon(\b w_h)
+ S^E((I - \P0 ) \b v_h, \, (I - \P0 ) \b w_h) \,,
\]
where $S^E(\cdot, \cdot) \colon \VDL \times \VDL \to \R$
is the VEM stabilizing term.
Many examples for the linear case can be found in the VEM literature \cite{volley,BDR:2017,BdV-Vacca:2022}.
In the present paper, we consider in the non-linear setting the so-called \texttt{dofi-dofi} stabilization defined as follows.
Let $N_E$ be the dimension of $\VDL$ and let, for $i=1, \dots, N_E$, $\chi_i \colon \VDL \to \R$ be the function that associates to each $\b v_h \in \VDL$ the value of the $i$-th local degree of freedom in {\bf (P2)}, and let $\DOF \in {\mathbb R}^{N_E}$ be the corresponding vector. 
Then, we propose the following two choices for $S^E$ (coherently with the choice in \eqref{eq:Carreau})
\\
\begin{equation}
\label{eq:dofi}
\begin{aligned}
S_1^E(\b v_h,\w_h) & := \overline{\mu}_E \, \big( \delta^\alpha + h_E^{-\alpha}\vert \DOF(\b v_h) \vert^\alpha \big)^{\frac{r-2}\alpha}  \DOF (\b v_h) \cdot \DOF (\w_h) \, , \\
S_2^E(\b v_h,\w_h) & :=  \overline{\mu}_E \, \sum_{i=1}^{\NE} \big( \delta^\alpha + h_E^{-\alpha}|\DOFi (\b v_h)|^\alpha \big)^{\frac{r-2}\alpha} \DOFi (\b v_h) \DOFi (\w_h) \, ,
\end{aligned}
\end{equation}
where $\overline{\mu}_E = \Pi^{0,E}_0 \mu$.
%where $\sigma_m$ is the constant appearing in \eqref{eq:hypo.monotonicity}. 
%
In Section~\ref{sec:stab-form} we will show that the above choice indeed satisfies a suitable stability property.
Since the two forms above are equivalent, see Lemma \ref{Lem:vemstab-prel}, the following results apply to each of them. 
Whenever a technical difference arises, for ease of exposition we will restrict the discussion to the $S^E_1$ choice.

Let $S(\cdot,\cdot) \colon \VDG \times \VDG \to \R$ be defined by 
\begin{equation}
\label{eq:Sglobal}    
S(\b v_h, \b w_h):=\sum_{E\in\Omega_h} S^E(\b v_h, \b w_h) \qquad \text{for all $\b v_h$, $\b w_h \in \VDG$.}
\end{equation} 
Then the global forms $a: \VDG \times \VDG \to \R$ is defined by summing the local contributions, i.e.
\begin{equation}
\label{eq:forma ah}
a_h(\b v_h, \b w_h) := 
\int_\Omega \b \sigma(\cdot, \b \Pi^0_{k-1} \b \epsilon(\b v_h)) : 
\b \Pi^0_{k-1} \b \epsilon(\b \w_h) + 
S((I - \Pi^0_k) \b v_h, (I - \Pi^0_k) \b w_h)
\qquad \text{for all $\b v_h$, $\b w_h \in \VDG$.}
\end{equation}
We finally define the discrete external force
\begin{equation}
\label{eq:forma fh}
\b f_h:= \Pi^0_k \b f \,,
\end{equation}
and observe that the ensuing right-hand side $(	\b f_h, \b v_h)$ is computable by property \textbf{(P3)}.

\begin{remark}{(Choiche of $\delta$)}\label{rem:deltazero}
To avoid a more cumbersome analysis, in the following, we tailor our theoretical developments for the more challenging and interesting power-law model, i.e. \eqref{eq:Carreau} with $\delta=0$. Our derivations could be extended to the case $\delta>0$ combining the present results with, e.g.,  the approaches in \cite{BN:1990, BL:1993,Barrett.Liu:94}.
% Our error estimates hold for all $\delta \ge 0$, and are optimal in the case $\delta=0$.  
\end{remark}

%{\color{red} SPIEGARE BENE NELLA INTRO: metodo numerico sempre ottimale. L'analisi e' stata svolta per il caso piu challenging delta=0, tecniche ad hoc, l'altro caso e' piu standard e non e' stato inserito per non rendere troppo cumbersome...}

%-----------------------------
\subsubsection{Properties of the stabilization form}
\label{sec:stab-form}

The following lemma will be useful in the sequel. 
\begin{lemma}\label{Lem:vemstab-prel}
Let the mesh regularity assumptions stated in Assumption~\ref{ass:mesh} hold. For any $E \in \mesh$ and for all $\b v_h \in \VDL$ we have
$$
S_1^E(\b v_h,\b v_h) \simeq S_2^E(\b v_h,\b v_h) \simeq 
h_E^{2-r} | \DOF(\b v_h) |^r \simeq 
h_E^{2-r} \max_{1 \le i \le \NE} | \DOFi (\b v_h) |^r \,.
$$
\end{lemma}
\begin{proof}
By the equivalence of vector norms, since $\NE$ is uniformly bounded for any $E \in \mesh$, we have
$$
S_1^E(\b v_h,\b v_h) := \overline{\mu}_E h_E^{2-r} \vert \DOF(\b v_h) \vert^r \simeq h_E^{2-r} \max_{1 \le i \le \NE} | \DOFi (\b v_h) |^r \, .
$$
Analogously,
$$
S_2^E(\b v_h,\b v_h) := \overline{\mu}_E h_E^{2-r} \sum_{i=1}^{\NE} |\DOFi (\b v_h)|^r \simeq
h_E^{2-r} \max_{1 \le i \le \NE} | \DOFi (\b v_h) |^r 
\simeq h_E^{2-r} \vert \DOF(\b v_h) \vert^r\, .
$$
\end{proof}

We now show the main coercivity and continuity bounds in $\b W^{1,r}(E)$ for the stabilization forms; later we will collect such results into a single Corollary taking into account also the presence of the symmetric gradient.

\begin{lemma}\label{Lem:vemstab-1}
Let the mesh regularity assumptions stated in Assumption~\ref{ass:mesh} hold.
For any $E \in \mesh$ we have
$$
|\b v_h|_{\b W^{1,r}(E)}^r \lesssim S^E(\b v_h,\b v_h) 
\qquad \text{for all $\b v_h \in \VDL$.}
$$
\end{lemma}
\begin{proof}
Let $E \in \mesh$ and $\b v_h \in \VDL$. Then first by a H\"older inequality with exponents $(\frac{2}{r},\frac{2}{2-r})$ and recalling $|E| \simeq h_E^2$, then applying \cite[Theorem 2]{MBM_23}, we obtain   
$$
|\b v_h|_{\b W^{1,r}(E)}^r \lesssim h_E^{2-r} |\b v_h|_{\b W^{1,2}(E)}^r 
\lesssim h_E^{2-r} \vert \DOF(\b v_h) \vert^r \, .
$$
%Furthermore, by equivalence of norms ($\NE$ is uniformly bounded) and Lemma \ref{Lem:vemstab-prel},
%$$
%h_E^{2-r} \vert \DOF(\b v_h) \vert^r \lesssim h_E^{2-r} \max_{1 \le i \le \NE} | \DOFi (\b v_h) |^r 
%\lesssim S^E(\b v_h,\b v_h) \, .
%$$
%The result follows combining the two bounds above.
The result follows combining the bound above with Lemma \ref{Lem:vemstab-prel}.
\end{proof}

\begin{lemma}\label{Lem:vemstab-2}
Let the mesh regularity assumptions stated in Assumption~\ref{ass:mesh} hold.
For any $E \in \mesh$ we have
$$
S^E(\b v_h,\b v_h) \lesssim |\b v_h|_{\b W^{1,r}(E)}^r 
\qquad \text{for all $\b v_h \in \VDL$ s.t. $\ \Proj\b v_h=0$.}
$$
\end{lemma}
\begin{proof}
Let $E \in \mesh$ and $\b v_h \in \VDL$.
Recalling Lemma \ref{Lem:vemstab-prel}, it is sufficient to show that 
$$
\max_{1 \le i \le N_E}h_E^{2-r} | \DOFi (\b v_h) |^r  \lesssim |\b v_h|_{W^{1,r}(E)}^r \qquad \text{for  $1 \le i \le \NE$,}
$$
uniformly in $E$ and $\b v_h$.
We need to handle boundary and bulk types of degrees of freedom separately. 

\smallskip\noindent
\emph{Boundary DoFs.} 
Let $\DOF_i$ be a degree of freedom of type ${\bf D}_U{\bf 1}$ or ${\bf D}_U{\bf 2}$, that is a pointwise evaluation at a generic point $\nu$ on $\partial E$. Then, by recalling that $\b v_h$ is piecewise polynomial on $\partial E$ and applying an inverse estimate, we get
$$
|\DOF_i(\b v_h)|^r = |\b v_h(\nu)|^r \le \| \b v_h \|_{\b L^\infty(\partial E)}^r \le h_E^{-1} \| \b v_h \|_{\b L^r(\partial E)}^r \, .
$$
Since $\Proj\b v_h=0$, we have $\int_{E} \b v_h = 0$. Therefore, we obtain the bound by applying the trace inequality \eqref{eq:cont_trace} followed by Poincar\'e's inequality
$$
h_E^{2-r} |\DOF_i(\b v_h)|^r \le 
h_E^{1-r} \| \b v_h \|_{\b L^r(\partial E)}^r
\le h_E^{1-r} \big( h_E^{-1} \| \b v_h \|_{\b L^r(E)}^r + h_E^{r-1} | \b v_h |_{\b W^{1,r}(E)}^r \big)
\lesssim | \b v_h |_{\b W^{1,r}(E)}^r \, .
$$

\noindent
\emph{Bulk DoFs.} Let now $\DOF_i$ be a degree of freedom of type $\mathbf{D_{\boldsymbol{U}}3}$. We focus on this case only since the proof for $\mathbf{D_{\boldsymbol{U}}4}$ follows with very similar steps. Then, for $m_{\alf} \in {\mathbb M}_{k-3}(E)$,
$$
|\DOF_i(\b v_h)| = \Big| \frac{1}{|E|} \int_E \b v_h \cdot {\bf m}^\perp m_{\alf} {\rm d}E \Big|  \ .
$$
We apply the H\"older inequality, recall that $|E|\simeq h_E^2$, and finally the inverse estimate \eqref{eq:inverse} on the polynomial $m_{\alf}$, yielding
\begin{equation}
\label{eq:per-infsup}
\begin{aligned}
|\DOF_i(\b v_h)| 
& \lesssim h_E^{-2} \| \b v_h \|_{\b L^r(E)} \| {\b m}^\perp \|_{\b L^{\infty}(E)} \| m_{\alf} \|_{ L^{r'}(E)}
\lesssim h_E^{-2} \| \b v_h \|_{\b L^r(E)} h_E^{2/r'} \| m_\alf \|_{L^{\infty}(E)} \\
& \lesssim h_E^{-2} h_E^{2(1-1/r)} \| \b v_h \|_{\b L^r(E)} = h_E^{-2/r} \| \b v_h \|_{\b L^r(E)} \, ,
\end{aligned}
\end{equation}
where we also used that $\| {\bf m}^\perp \|_{\b L^{\infty}(E)}$ and  $\| m_\alf \|_{\b L^{\infty}(E)} \lesssim 1$. 
As mentioned above, we recall that $\int_E \b v_h = 0$. 
First using the above bound, then applying a (scaled) Poincar\'e inequality, finally gives
$$
h_E^{2-r} | \DOFi (\b v_h) |^r \lesssim 
h_E^{-r} \| \b v_h \|_{\b L^r(E)}^r 
%= h_E^{-r} \| \b v_h \|_{L^r(E)}^r
\lesssim |\b v_h|_{\b W^{1,r}(E)}^r \, .
$$
\end{proof}
The next instrumental result follows from Lemma \ref{Lem:vemstab-prel}, Lemma \ref{Lem:vemstab-1}, and Lemma \ref{Lem:vemstab-2}.
\begin{corollary}%\label{cor:inverse}
Let the mesh regularity assumptions stated in Assumption~\ref{ass:mesh} hold. Then for any $E \in \mesh$ we have the following inverse estimate
\begin{equation}
\label{eq:inv_norms}
|\b v_h|_{\b W^{1,2}(E)} \lesssim h_E^{(r-2)/r} |\b v_h|_{\b W^{1,r}(E)}  
\qquad \text{for all $\b v_h \in \VDL$.}
\end{equation}
\end{corollary}
\begin{proof}
We start by combining Lemma \ref{Lem:vemstab-1} (for $r=2$) and Lemma \ref{Lem:vemstab-prel} (still for $r=2$), obtaining
$$
|\b v_h|_{\b W^{1,2}(E)}^2 \lesssim \max_{1 \le i \le \NE} | \DOFi (\b v_h) |^2 \, .
$$
We now take the square root and, using that $\NE$ is uniformly bounded, manipulate as follows 
$$
|\b v_h|_{\b W^{1,2}(E)} \lesssim \max_{1 \le i \le \NE} | \DOFi (\b v_h) | 
\lesssim h_E^{(r-2)/r} \Big( h_E^{2-r} \max_{1 \le i \le \NE} | \DOFi (\b v_h) |^r \Big)^{1/r} \, .
$$
We now apply Lemma \ref{Lem:vemstab-prel} and Lemma \ref{Lem:vemstab-2}, yielding
$$
|\b v_h|_{\b W^{1,2}(E)} \lesssim h_E^{(r-2)/r} S_E(\b v_h,\b v_h)^{1/r} \lesssim h_E^{(r-2)/r} |\b v_h|_{\b W^{1,r}(E)} \, ,
$$
where we observe that in Lemma \ref{Lem:vemstab-2} it is sufficient that $\int_E \b v_h = 0$, something which is not restrictive to assume here since this result only involves semi-norms.
\end{proof}
\begin{corollary}\label{Cor:vemstab}
Let the mesh regularity assumptions stated in Assumption~\ref{ass:mesh} hold.
For any $E \in \mesh$ we have 
$$
\| {\boldsymbol\epsilon} (\b v_h) \|_{\mathbb{L}^r(E)}^r  
\lesssim S^E(\b v_h,\b v_h) \lesssim \| {\boldsymbol\epsilon} (\b v_h) \|_{\mathbb{L}^r(E)}^r 
\qquad \text{for all $\b v_h \in \VDL$ s.t. $\Proj\b v_h=0$.}
$$
\end{corollary}
\begin{proof}
The first bound follows immediately from Lemma \ref{Lem:vemstab-1}. The second bound is a consequence of 
Lemma \ref{Lem:vemstab-2} combined with Korn's second inequality, cf. \cite[Theorem 2]{Wang:03}. 
Indeed, the assumption $\Proj\b v_h=0$ for $k\ge1$ implies that $\b v_h\in \VDL$ is rigid body motion-free and, as a result of \cite[Theorem 3.3 and Remark 3.4]{Lewintan.Neff:21}, we have
$$
S^E(\b v_h,\b v_h) \lesssim
|\b v_h|_{W^{1,r}(E)}^r \lesssim \| {\boldsymbol\epsilon} (\b v_h) \|_{L^r(E)}^r,
$$
with hidden constants depending on $r$ but independent of $E$ due to Assumption 2.
\end{proof}

We close this section with the following results regarding the stabilization form.

\begin{lemma}[Strong monotonicity of $S^E(\cdot,\cdot)$]\label{lm:monotonicity:local_stab}
Let the mesh regularity assumptions stated in Assumption~\ref{ass:mesh} hold.
Let $\b u_h$, $\b w_h \in \VDL$  and set $\b e_h:=\b u_h - \b w_h$. Then there holds
\begin{equation}\label{eq:monotonicity:local_stab}
    S^E(\b u_h,\b e_h)-S^E(\b w_h,\b e_h) \gtrsim 
    S^E(\b e_h,\b e_h)^{\frac 2 r}\left( 
    h_E^{2-r}( 
    \vert \DOF(\b u_h) \vert^{r} +  \vert \DOF(\b w_h) \vert^{r} )
    \right)^{\frac{r-2}{r}}.
\end{equation}
Moreover if $\b u_h$ and $\b w_h$ are s.t. $\Pi_k^{0,E}\b u_h = \Pi_k^{0,E}\b w_h = \b 0$, then
\begin{equation}\label{eq:monotonicity:local_stab2}
    S^E(\b u_h,\b e_h)-S^E(\b w_h,\b e_h) \gtrsim 
    \Vert \b \epsilon(\b e_h) \Vert^2_{\mathbb{L}^r(E)} \left(  
    \Vert \b \epsilon(\b u_h) \Vert^r_{\mathbb{L}^r(E)} + 
    \Vert \b \epsilon(\b w_h) \Vert^r_{\mathbb{L}^r(E)}
    \right)^{\frac{r-2}{r}}.
\end{equation}
\end{lemma}

\begin{proof}
  Employing \eqref{eq:bound_hyp_model} with 
$\b x= \DOF(\b u_h)$,  $\b y= \DOF(\b w_h) $ and $\delta=0$ and recalling the definition of $S_1^E(\cdot, \cdot)$ in \eqref{eq:dofi} we infer
   \begin{equation}
   \label{eq:strong1}
   \begin{aligned}
       \vert \DOF(\b e_h)\vert^2 (\vert \DOF(\b u_h)\vert^r+\vert \DOF(\b w_h) \vert^r)^{\frac{r-2}{r}} &\lesssim 
       \left(
       (\vert \DOF(\b u_h)\vert^{r-2} \DOF(\b u_h) - 
       \vert \DOF(\b w_h)\vert^{r-2} \DOF(\b w_h)
       \right) \cdot \DOF(\b e_h)
       \\
       &\lesssim h_E^{r-2} 
       \bigl(S_1^E(\b u_h, \b e_h) - S_1^E(\b w_h, \b e_h) \bigr) \,.
  \end{aligned}     
  \end{equation}
Owing to Lemma \ref{Lem:vemstab-prel} it holds
%and the equivalence of norms in finite dimensional spaces, it is easy to derive that 
$S_1^E(\b e_h,\b e_h)^{2/r} \simeq h_E^{\frac{2(2-r)}{r}}\vert\DOF(\b e_h)\vert^2$. Therefore from \eqref{eq:strong1} we derive
$$
h_E^{\frac{2(r-2)}{r}} \, S_1^E(\b e_h,\b e_h)^{\frac{2}{r}}
 (\vert \DOF(\b u_h)\vert^r+\vert \DOF(\b w_h) \vert^r)^{\frac{r-2}{r}}
\lesssim h_E^{r-2} 
       \bigl(S_1^E(\b u_h, \b e_h) - S_1^E(\b w_h, \b e_h) \bigr) \,.
$$
Bound \eqref{eq:monotonicity:local_stab} easily follows from the bound above.
Bound \eqref{eq:monotonicity:local_stab2} follows combining in \eqref{eq:monotonicity:local_stab} Lemma \ref{Lem:vemstab-prel} and Corollary \ref{Cor:vemstab}.

The bounds for the stabilization $S_2^E(\cdot, \cdot)$ in \eqref{eq:dofi} can be derived using analogous arguments.
\end{proof}

\begin{corollary}[Strong monotonicity of $S(\cdot,\cdot)$]
%\label{cor:monotonicity:global_stab}
Let the mesh regularity assumptions stated in Assumption~\ref{ass:mesh} hold.
Let $\b u_h$, $\b w_h \in \VDG$
and set $\b e_h:=\b u_h - \b w_h$. Then there holds
\begin{equation}\label{eq:monotonicity:global_stab}
    S(\b u_h,\b e_h)-S(\b w_h,\b e_h) 
    \gtrsim S(\b e_h,\b e_h)^{\frac 2 r}\left( 
    \sum_{E\in\Omega_h} h_E^{2-r}( 
    \vert \DOF(\b u_h) \vert^{r} +  \vert \DOF(\b w_h) \vert^{r} )
    \right)^{\frac{r-2}{r}}.
\end{equation}
Moreover if $\b u_h$ and $\b w_h$ are s.t. $\Pi_k^{0}\b u_h = \Pi_k^{0}\b w_h = \b 0$ %for any $E \in \mesh$, 
then
\begin{equation}\label{eq:monotonicity:global_stab2}
    S(\b u_h,\b e_h)-S(\b w_h,\b e_h) \gtrsim 
    \Vert \b \epsilon(\b e_h) \Vert^2_{\mathbb{L}^r(\Omega_h)} \left(  
    \Vert \b \epsilon(\b u_h) \Vert^r_{\mathbb{L}^r(\Omega_h)} + 
    \Vert \b \epsilon(\b w_h) \Vert^r_{\mathbb{L}^r(\Omega_h)}
    \right)^{\frac{r-2}{r}}\,.
\end{equation}
\end{corollary}
\begin{proof}
Applying Lemma \ref{lm:monotonicity:local_stab} and employing the H\"older inequality with exponents $(\frac 2 r, \frac{2}{2-r})$,
direct computations yield
$$
\begin{aligned}
S(\b e_h,\b e_h)& =\sum_{E\in\Omega_h} S^E(\b e_h,\b e_h)  \lesssim
\sum_{E\in\Omega_h} 
\bigl(S^E(\b u_h, \b e_h) -  S^E(\b w_h, \b e_h)\bigr)^{\frac{r}{2}}
\left( 
    h_E^{2-r}( 
    \vert \DOF(\b u_h) \vert^{r} +  \vert \DOF(\b w_h) \vert^{r} )
    \right)^{\frac{2-r}{2}}
    \\
&\lesssim
 \bigl(S(\b u_h, \b e_h) -  S(\b w_h, \b e_h)\bigr)^{\frac{r}{2}}
\left( \sum_{E\in\Omega_h}
    h_E^{2-r}( 
    \vert \DOF(\b u_h) \vert^{r} +  \vert \DOF(\b w_h) \vert^{r} )
    \right)^{\frac{2-r}{2}}
\end{aligned}
$$
Raising both sides of the bound above to $2/r$ we obtain \eqref{eq:monotonicity:global_stab}.
Recalling definition \eqref{eq:normebroken} and  employing again the H\"older inequality with exponents $(\frac 2 r, \frac{2}{2-r})$, bound \eqref{eq:monotonicity:global_stab2} can be derived from \eqref{eq:monotonicity:local_stab2} as follows
$$
\begin{aligned}
\Vert \b \epsilon(\b e_h) \Vert_{\mathbb{L}^r(\Omega_h)}^r & 
\lesssim
\sum_{E \in \Omega_h} \bigl(S^E(\b u_h, \b e_h) - S^E(\b w_h, \b e_h) \bigr)^{\frac{r}{2}} \bigl( \Vert \b \epsilon(\b u_h) \Vert_{\mathbb{L}^r(E)}^r + \Vert \b \epsilon(\b w_h) \Vert_{\mathbb{L}^r(E)}^r\bigr)^{\frac{2-r}{2}}
\\
& \lesssim
\bigl(S(\b u_h, \b e_h) - S(\b w_h, \b e_h) \bigr)^{\frac{r}{2}} 
\bigl( \Vert \b \epsilon(\b u_h) \Vert_{\mathbb{L}^r(\Omega_h)}^r + \Vert \b \epsilon(\b w_h) \Vert_{\mathbb{L}^r(\Omega_h)}^r\bigr)^{\frac{2-r}{2}} \,.
\end{aligned}
$$
Raising both sides of the bound above to $2/r$ we obtain \eqref{eq:monotonicity:global_stab2}.
\end{proof}

\begin{lemma}[H\"older continuity of $S^E(\cdot,\cdot)$]
%\label{Lemma:S_holder}
    Let the mesh regularity assumptions stated in Assumption~\ref{ass:mesh} hold.
    Let $\b u_h$, $\b w_h \in \VDL$  and set $\b e_h:=\b u_h - \b w_h$. Then there holds
    \begin{equation}
    \label{eq:continuity:local_stab1}
        \vert S^E(\b u_h,\b v_h)-S^E(\w_h,\b v_h)\vert
        \lesssim h_E^{2-r} \vert \DOF(\b e_h) \vert^{r-1} \vert \DOF(\b v_h) \vert \quad \text{for all $\b v_h \in \VDL$.}
    \end{equation}
Moreover if $\b u_h$, $\b w_h$ and $\b v_h$ are s.t. $\Pi_k^{0,E}\b u_h = \Pi_k^{0,E}\b w_h = \Pi_k^{0,E}\b v_h = \b 0$, then
\begin{equation}\label{eq:continuity:local_stab22}
    |S^E(\b u_h,\b v_h)-S^E(\b w_h,\b v_h)| \lesssim
    \Vert \b e_h \Vert^{r-1}_{\b W^{1,r}(E)}
    \Vert \b v_h \Vert_{\b W^{1,r}(E)} \,.
\end{equation}  
\end{lemma}
\begin{proof}
    Recalling the definition of $S^E_1(\cdot,\cdot)$  in \eqref{eq:dofi} with $\delta=0$, employing \eqref{eq:hypo.continuity} and \eqref{eq:pre_holdercont}, bound \eqref{eq:continuity:local_stab1} can be derived as follows
    \begin{equation*}
    \begin{aligned}
    S^E_1(\b u_h,\b v_h)-S^E_1(\b w_h,\b v_h) 
    &= h_E^{2-r} \left(\vert \DOF(\b u_h) \vert^{r-2} \DOF(\b u_h) 
      - \vert \DOF(\b w_h) \vert^{r-2} \DOF (\b w_h))\right)\cdot
      \DOF(\b v_h) 
      \\
     &\lesssim
     h_E^{2-r} \left(\vert \DOF(\b u_h) \vert^{r} 
      + \vert \DOF(\b w_h) \vert^{r}\right)^{\frac{r-2}{r}} \vert\DOF (\b e_h) \vert \,
      \vert\DOF(\b v_h)\vert  
      \\
     &\lesssim
     h_E^{2-r} \vert \DOF(\b e_h) \vert^{r-2} \vert\DOF (\b e_h) \vert \,
      \vert\DOF(\b v_h)\vert
= h_E^{2-r} \vert \DOF(\b e_h) \vert^{r-1} \,\vert\DOF(\b v_h)\vert        \,.
    \end{aligned}
    \end{equation*}
Bound \eqref{eq:continuity:local_stab22} is a direct consequence of \eqref{eq:continuity:local_stab1}, Lemma \ref{Lem:vemstab-prel} and  Lemma \ref{Lem:vemstab-2}.
    The result for $S^E_2(\cdot,\cdot)$ follows by the same argument applying the H\"older inequality.
\end{proof}
\begin{lemma}[H\"older continuity of $S(\cdot,\cdot)$]
\label{Lemma:S_holderG}
    Let the mesh regularity assumptions stated in Assumption~\ref{ass:mesh} hold.
    Let $\b u_h$, $\b w_h \in \VDG$  and set $\b e_h:=\b u_h - \b w_h$. Then there holds
    \begin{equation}
    \label{eq:continuity:local_stabG1}
        \vert S(\b u_h,\b v_h)-S(\w_h,\b v_h)\vert
        \lesssim 
        \sum_{E \in \Omega_h}h_E^{2-r} \vert \DOF(\b e_h) \vert^{r-1} \vert \DOF(\b v_h) \vert \quad \text{for all $\b v_h \in \VDL$.}
    \end{equation}
Moreover if $\b u_h$ and $\b w_h$ are s.t. $\Pi_k^{0}\b u_h = \Pi_k^{0}\b w_h = \b 0$, then
\begin{equation}\label{eq:continuity:local_stabG2}
    |S(\b u_h,\b v_h)-S(\b w_h,\b v_h)| \lesssim
    \Vert \b e_h \Vert^{r-1}_{\b W^{1,r}(\Omega_h)}
    \Vert \b v_h \Vert_{\b W^{1,r}(\Omega_h)} \,.
\end{equation}  
\end{lemma}
\begin{proof}
Bound \eqref{eq:continuity:local_stabG1}  easily follows from  \eqref{eq:continuity:local_stab1}. 
Employing the the H\"older inequality with exponents $(r', r)$,  bound \eqref{eq:continuity:local_stabG2}  can be derived from \eqref{eq:continuity:local_stab22} as follows
\[
\begin{aligned}
|S(\b u_h,\b v_h)-S(\b w_h,\b v_h)| &\le
\sum_{E \in \Omega_h} |S^E(\b u_h,\b v_h)-S^E(\b w_h,\b v_h)|
 \lesssim
 \sum_{E \in \Omega_h}
    \Vert \b e_h \Vert^{r-1}_{\b W^{1,r}(E)}
    \Vert \b v_h \Vert_{\b W^{1,r}(E)}    
  \\
  & \lesssim
  \Vert \b e_h \Vert^{r-1}_{\b W^{1,r}(\Omega_h)}
    \Vert \b v_h \Vert_{\b W^{1,r}(\Omega_h)} \,.   
\end{aligned}
\]
\end{proof}
%----------------------------------------------------------------------
\subsubsection{Discrete inf-sup condition}\label{sec:infsup}
\def\ww{{\bf w}}
We here prove a discrete inf-sup condition analogous to the continuous one \eqref{eq:inf_sup}. Note that the argument, especially in the first part, has important differences from that developed in \cite{BLV:2017} for $r=2$, because here we cannot exploit a ``minimum energy'' argument on the elements. We use a more direct approach based on the previous lemmas regarding the discrete stability form.  

\begin{lemma}[Discrete inf-sup]\label{inf-sup:vem}
Let the mesh regularity assumptions stated in Assumption~\ref{ass:mesh} hold. Then, for any $r\in(1,\infty)$ it exists a constant $\overline{\beta}(r)$, such that
$$
\inf_{q_h\in \QDG}\;\sup_{\b w_h\in \VDG}\; 
\frac{b(\b w_h,q_h)}{\norm[L^{r'}(\Omega)]{q_h}\norm[\b{W}^{1,r}(\Omega)]{\b w_h}} 
\ge \overline{\beta}(r) > 0.
$$
\end{lemma}
\begin{proof} Due to \eqref{eq:inf_sup}, it is sufficient to show the existence of a Fortin operator, see for example \cite{boffi-brezzi-fortin:book}. The proof is divided into two parts. 

\smallskip\noindent
\emph{Part 1.}
We start by introducing a suitable lowest-order Cl\'ement type interpolant in $\VDG$ which, given any 
$\ww \in \b{W}^{1,r}_{0,\Gamma_D}(\Omega)$, we will denote by $\ww_c$. Given any vertex $\nu$ of the mesh, we denote by $\omega_\nu$ the set given by the union of all elements in $\Omega_h$ sharing $\nu$ as a vertex. Given $E \in \Omega_h$, we denote by $\omega_E$ the union of all the $\omega_\nu$ for $\nu$ vertexes of $E$.
We will prove that, for any $E \in \Omega_h$ and $\ww \in \b{W}^{1,r}_{0,\Gamma_D}(\Omega)$, the quasi-interpolant satisfies
\begin{equation}\label{clement:main}
\| \ww - \ww_c \|_{\b L^r(E)} + h_E | \ww - \ww_c |_{\b W^{1,r}(E)} \lesssim h_E | \ww |_{\b W^{1,r}(\omega_E)} .
\end{equation}
We start by defining $\ww_c \in \VDG$ through its DoF values as follows. We initially impose
$$
\ww_c (\nu) = \frac{1}{|\omega_\nu|} \int_{\omega_\nu} \ww \, \textrm{d} \omega_\nu
\qquad \textrm{for any } \nu \textrm{ vertex of } \Omega_h \textrm{ not in } \Gamma_D .
$$
The remaining skeletal DoFs, which are the point-wise evaluations on edges, are simply set by linearly interpolating the corresponding vertex values of each edge.
{
We further set the DoFs $\mathbf{D_{\boldsymbol{U}}3}$ as follows
$$
\frac{1}{|E|}\int_E  \b w_c \cdot \boldsymbol{m}^{\perp} m_{\boldsymbol{\alpha}} \, {\rm d}E  
:=
\frac{1}{|E|}\int_E  \b w \cdot \boldsymbol{m}^{\perp} m_{\boldsymbol{\alpha}} \, {\rm d}E
\qquad 
\text{for any $m_{\boldsymbol{\alpha}} \in \M_{k-3}(E)$,}
$$
and we enforce all degrees of freedom $\mathbf{D_{\boldsymbol{U}}4}$ of $\b w_c$ equal to zero.

It is clear that the above operator preserves constants, in the sense that if $E \in {\Omega_h}$ is an element without vertexes on $\Gamma_D$, then if $\ww|_{\omega_E} = {\bf p}_0 \in [\Pk_0(\omega_E)]^2$ it will hold $\ww_{c|E} = {\bf p}_0$. Furthermore, from the DoFs definition and employing \eqref{eq:per-infsup}, it is easy to check  that for any $E \in {\Omega_h}$
\begin{equation}\label{L:first}
| \DOF (\ww_c) | \lesssim \max_{\nu \in \partial E \textrm{ vertex}} | \ww_c (\nu) | + h_E^{-2/r} \Vert \ww_c \Vert_{\b L^r(E)}
\,.
\end{equation}
We now start from a well known (scaled) Poincar\`e type inequality (the uniformity of the involved constant on $E$ following from Assumption~\ref{ass:mesh}) and afterwards apply Lemma \ref{Lem:vemstab-1} and a trivial calculation
$$
\| \ww_c \|_{\b L^r(E)} \lesssim h_E | \ww_c |_{\b W^{1,r}(E)} + h_E^{2/r - 1} \left| \int_{\partial E} \ww_c \right|
\lesssim h_E S^E(\ww_c,\ww_c)^{1/r} + h_E^{2/r} \max_{\nu \in \partial E \textrm{ vertex}} | \ww_c (\nu) | .
$$
From the above bound, since the maximum of all the vertex value norms is clearly bounded by 
$\vert \DOF (\ww_c) \vert$, using Lemma \ref{Lem:vemstab-prel} we obtain
\begin{equation}\label{L:second}
\| \ww_c \|_{\b L^r(E)} \lesssim h_E S^E(\ww_c,\ww_c)^{1/r}  + 
h_E^{2/r} h_E^{(r-2)/r} S^E(\ww_c,\ww_c)^{1/r} \lesssim h_E S^E(\ww_c,\ww_c)^{1/r}  .
\end{equation}
To avoid repetition of similar ideas, we will prove only the bound for the first term in the left-hand side of \eqref{clement:main}. Furthermore, we will consider only the case $\overline{E} \cap \Gamma_D = \emptyset$, the other one following very similarly but using a classical Poincar\`e inequality instead of the constant preserving property. By a triangle inequality and recalling that the quasi-interpolant preserves constants (in the sense detailed above), for any ${\bf p}_0 \in [\Pk_0(\Omega)]^2$ we have
\begin{equation}\label{clem:tria}
\| \ww - \ww_c \|_{\b L^r(E)} \le \| \ww - {\bf p}_0 \|_{\b L^r(E)} + \| (\ww - {\bf p}_0)_c \|_{\b L^r(E)} .
\end{equation}
For the second term, we first apply \eqref{L:second}, then Lemma \ref{Lem:vemstab-prel} and finally 
\eqref{L:first}:
$$
\begin{aligned}
\| (\ww - {\bf p}_0)_c \|_{\b L^r(E)} & \lesssim 
h_E S^E( (\ww - {\bf p}_0)_c , (\ww - {\bf p}_0)_c )^{1/r}
\lesssim
h_E^{1 + (2-r)/r} \max_{1 \le i \le \NE} | \DOFi ((\ww - {\bf p}_0)_c) | \\
& \lesssim
h_E^{2/r} \max_{\nu \in \partial E \textrm{ vertex}} | (\ww - {\bf p}_0)_c (\nu) | +  \Vert \b w - {\bf p}_0\Vert_{\b L^r(E)}
\,.
\end{aligned}
$$
By definition of the quasi-interpolant and due to the shape regularity of the mesh (yielding $|\omega_\nu| \simeq h_E^2$ for any $\nu$ vertex of $E$) we obtain from above
$$
\| (\ww - {\bf p}_0)_c \|_{\b L^r(E)} \lesssim 
\max_{\nu \in \partial E \textrm{ vertex}} 
|\omega_\nu|^{1/r}
\left| {\frac{1}{|\omega_\nu|} \int_{\omega_\nu} (\ww - {\bf p}_0) }\right| +  \Vert \b w - {\bf p}_0\Vert_{\b L^r(E)} \,.
$$
Using an Holder inequality $(r,r')$ we obtain
$$
\| (\ww - {\bf p}_0)_c \|_{\b L^r(E)} \lesssim \max_{\nu \in \partial E \textrm{ vertex}} 
\| \ww - {\bf p}_0 \|_{\b L^r(\omega_\nu)}  +  \Vert \b w - {\bf p}_0\Vert_{\b L^r(E)} \le \| \ww - {\bf p}_0 \|_{\b L^r(\omega_E)} .
$$
Combining the above bound with \eqref{clem:tria} and standard polynomial approximation estimates we finally obtain the desired bound
$$
\| \ww - \ww_c \|_{\b L^r(E)} \lesssim h_E  \, | \ww |_{\b W^{1,r}(\omega_E)} .
$$
}
\noindent  %%
\emph{Part 2.} % We present briefly this second part, which is more standard. 
We define a preliminary Fortin operator $\widehat\Pi : \b W^{1,r}_{0,\Gamma_D}(\Omega) \rightarrow \VDG$ by setting its DoF values as follows. For any $\ww \in \b W^{1,r}_{0,\Gamma_D}(\Omega)$, we set to zero all DOFs of type $\mathbf{D_{\boldsymbol{U}}1}$, $\mathbf{D_{\boldsymbol{U}}3}$ and of type $\mathbf{D_{\boldsymbol{U}}2'}$ (cf. \eqref{eq:per-infsup2}) apart from   
$$
\frac{1}{|e|}\int_e \widehat\Pi \ww \cdot {\bf n}_e  = \frac{1}{|e|} \int_e \ww \cdot {\bf n}_e \,,
$$
on each edge $e$ (not in $\Gamma_D$).
We set the DoFs $\mathbf{D_{\boldsymbol{U}}4}$ as 
$$
\frac{h_E}{|E|}\int_E (\divs \widehat\Pi \ww) \, m_{\boldsymbol{\alpha}} \, {\rm d}E  
:= \frac{h_E}{|E|}\int_E (\divs \ww) \, m_{\boldsymbol{\alpha}} \, {\rm d}E
\quad  \text{for all $m_{\boldsymbol{\alpha}} \in \M_{k-1}(E)$ with $|\boldsymbol{\alpha}| > 0$.}
$$
It is trivial to check that the above operator satisfies the first Fortin condition by construction, that is 
$b(\widehat\Pi \ww,q_h) = b(\ww,q_h)$ for all $q_h \in \QDG$. Furthermore, it also easily follows from the above definition
$$
\vert \DOF (\widehat\Pi \ww) \vert \lesssim 
\frac{h_E}{|E|} \max_{\text{ $m_{\boldsymbol{\alpha}} \in \M_{k-1}(E)$}} 
\left|\int_E (\divs \ww) \, m_{\boldsymbol{\alpha}} \, {\rm d}E \right|
\ + \ \max_{e \in \partial E} \frac{1}{|e|} \left| \int_e \ww \cdot {\bf n}_e \, \textrm{d} \, s \right| 
$$
which combined with a scaled trace inequality and a Holder inequality $(r,r')$ yields
\begin{equation}\label{pirupiru}
\vert \DOF (\widehat\Pi \ww) \vert \lesssim h_E^{- 2/r} 
\Big( \| \ww \|_{\b L^{r}(E)} + h_E |\ww|_{\b W^{1,r}(E)} \Big) .
\end{equation}
We can now define the Fortin operator
$\Pi^{\cal F} : \b W^{1,r}_{0,\Gamma_D}(\Omega) \rightarrow \VDG$.
We set
$$
\Pi^{\cal F} \ww = \ww_c + \widehat\Pi (\ww - \ww_c) 
\qquad \forall \ \ww \in \b W^{1,r}_{0,\Gamma_D}(\Omega) .
$$
It is immediate to check that also $\Pi^{\cal F}$ satisfies the first Fortin condition, that is 
$b(\Pi^{\cal F} \ww,q_h) = b(\ww,q_h)$ for all $q_h \in \QDG$. In order to conclude, we are left to check the second condition for a Fortin operator, which is its continuity in the $\b W^{1,r}(\Omega)$ norm (see, eg. \cite{boffi-brezzi-fortin:book}).

By the triangle inequality and recalling \eqref{clement:main}, for all $E \in \Omega_h$
\begin{equation}\label{triaF}
| \Pi^{\cal F} \ww |_{\b W^{1,r}(E)} \lesssim |\ww|_{\b W^{1,r}(\omega_E)} 
+  |\widehat\Pi (\ww - \ww_c)|_{\b W^{1,r}(E)} \,.
\end{equation}
For the second term on the right hand side, we first apply again Lemma \ref{Lem:vemstab-prel} and \ref{Lem:vemstab-1}, then recall \eqref{pirupiru} and finally make use of \eqref{clement:main}, yielding
$$
\begin{aligned}
|\widehat\Pi (\ww - \ww_c)|_{\b W^{1,r}(E)} 
& \lesssim h_E^{(2-r)/r} \vert \DOF (\widehat\Pi (\ww - \ww_c)) \vert \\
& \lesssim h_E^{-1} \| \ww - \ww_c \|_{\b L^{r}(E)} 
+ | \ww - \ww_c |_{\b W^{1,r}(E)} 
\lesssim  | \ww |_{\b W^{1,r}(\omega_E)} .
\end{aligned}
$$
The continuity in $\b W^{1,r}(\Omega)$ of $\Pi^{\cal F}$ follows trivially from  \eqref{triaF} and the bound here above, summing on all elements $E \in \Omega_h$ and noting that, due to Assumption \ref{ass:mesh}, 
$\{ \omega_E \}_{E \in \Omega_h}$ overlaps every element a uniformly bounded number of times. 

\end{proof}

%---------------------------------------------------------------------------
\subsection{Virtual Element problem}
\label{sub:vem problem}

Having in mind the spaces \eqref{eq:v-glo} and \eqref{eq:q-glo}, the discrete form  \eqref{eq:forma ah}, the form \eqref{eq:a.b}, the discrete loading term \eqref{eq:forma fh}, the virtual element discretization of Problem \eqref{eq:stokes.weak} is given by:
Find $(\b u_h, p_h) \in \VDG \times \QDG$ such that
\begin{equation}\label{eq:stokes.vem}
  \begin{aligned}
     a_h(\b u_h,\b v_h)+b(\b v_h, p_h) &= \int_\Omega \b f_h \cdot \b v_h + \int_{\Gamma_N} \b g \cdot \b\gamma(\b v_h)&\qquad \forall \b v_h \in \VDG , \\
     b(\b u_h, q_h) &= 0 &\qquad \forall q_h \in \QDG \,. 
  \end{aligned}
\end{equation}
Recalling the definition of the discrete kernel $\ZDG$ in \eqref{eq:z-glo}, the previous problem can be also written in the kernel formulation:
Find $\b u_h \in \b \ZDG$ such that
\begin{equation}\label{eq:stokes.vem.Z}
     a_h(\b u_h, \b v_h) = \int_\Omega \b f_h \cdot \b v_h + \int_{\Gamma_N} \b g \cdot \b\gamma(\b v_h) \qquad \forall \b v_h \in \b \ZDG \,.
\end{equation}

In order to prove the well-posedness of the discrete problem \eqref{eq:stokes.vem}, we establish the strong monotonicity and the H\"older continuity of the non-linear function $a_h(\cdot, \cdot)$.

\begin{lemma}[Strong monotonicity of $a_h(\cdot,\cdot)$] 
\label{lem:mono_ah}
Let $\b u_h$, $\b w_h \in \VDG$ and set $\b e_h:=\b u_h - \b w_h \in \VDG$. Then there holds
\begin{equation*}
\begin{aligned}
    a_h(\b u_h, \b e_h) - a_h(\b w_h, \b e_h) & \gtrsim
    \left(\|\b u_h\|^r_{\b W^{1,r}(\Omega)} +
    \| \b w_h\|^r_{\b W^{1,r}(\Omega)} \right)^{\frac{r-2}{r}} \, \Vert  \b e_h \Vert^2_{\b W^{1,r}(\Omega)}\,.
\end{aligned}
\end{equation*}
\end{lemma}
\begin{proof}
Recalling the definition of $a_h(\cdot, \cdot)$ in \eqref{eq:forma ah} and definition \eqref{eq:Sglobal} we have
\begin{equation}
    \label{eq:lemma15-1}
    a_h(\b u_h, \b e_h) - a_h(\b w_h, \b e_h) := T_1 + T_2 
\end{equation}
where
$$
    \begin{aligned}
    T_1 &: =
    \int_\Omega \b\sigma(\cdot, \b\Pi^0_{k-1} \b \epsilon(\b u_h)) : \b\Pi^0_{k-1} \b\epsilon(\b e_h) - 
\int_\Omega \b\sigma(\cdot, \b\Pi^0_{k-1} \b \epsilon(\b w_h)) : \b\Pi^0_{k-1} \b\epsilon(\b e_h)  \,,
\\
T_2 &:= S((I - \Pi^0_k) \b u_h, (I -  \Pi^0_k) \b e_h) -
    S((I -  \Pi^0_k) \b w_h, (I -  \Pi^0_k) \b e_h) \,.
    \end{aligned}
$$
We estimate separately each term above.
The same identical proof used to derive \eqref{eq:a.monotonicity} (with $\delta=0$) and the continuity of the $L^2$-projection with respect to the $L^r$-norm yield
\begin{equation}
\label{eq:lemma15-2}
\begin{aligned}
T_1
&\gtrsim \sigma_{\rm m} 
\left(\norm[\mathbb{L}^{r}(\Omega)]{\b\Pi^0_{k-1} \b \epsilon (\b u_h)}^r 
+ \norm[\mathbb{L}^{r}(\Omega)]{\b\Pi^0_{k-1} \b \epsilon (\b w_h)}^r\right)^{\frac{r-2}r} 
\norm[\mathbb{L}^{r}(\Omega)]{\b\Pi^0_{k-1} \b \epsilon (\b e_h)}^2 
\\
&
\gtrsim \sigma_{\rm m} 
\left(\norm[\mathbb{L}^{r}(\Omega)]{ \b \epsilon (\b u_h)}^r 
+ \norm[\mathbb{L}^{r}(\Omega)]{\b \epsilon (\b w_h)}^r\right)^{\frac{r-2}r} 
\norm[\mathbb{L}^{r}(\Omega)]{\b\Pi^0_{k-1} \b \epsilon (\b e_h)}^2
\\
&
\gtrsim \sigma_{\rm m} 
\left(\norm[\b W^{1, r}(\Omega)]{ \b u_h}^r 
+ \norm[\b W^{1, r}(\Omega)]{ \b w_h}^r \right)^{\frac{r-2}r} 
\norm[\mathbb{L}^{r}(\Omega)]{\b\Pi^0_{k-1} \b \epsilon (\b e_h)}^2 \,.
\end{aligned}
\end{equation}
Employing bound \eqref{eq:monotonicity:global_stab2} and   the properties of the  $L^2$-projection   we infer
\begin{equation}
\label{eq:lemma15-3}
\begin{aligned}
T_2 &\gtrsim 
    \left(  
    \Vert \b \epsilon((I - \Pi^0_k)\b u_h) \Vert^r_{\mathbb{L}^r(\Omega_h)} + 
    \Vert \b \epsilon((I -  \Pi^0_k)\b w_h) \Vert^r_{\b L^r(\Omega_h)}
    \right)^{\frac{r-2}{r}}
    \Vert \b \epsilon((I -  \Pi^0_k)\b e_h) \Vert^2_{\mathbb{L}^r(\Omega_h)}
    \\
    & \ge
    \left(  
    \vert (I -  \Pi^0_k) \b u_h \vert^r_{\b W^{1,r}(\Omega_h)} + 
    \vert (I -  \Pi^0_k)\b w_h \vert^r_{\b W^{1,r}(\Omega_h)}
    \right)^{\frac{r-2}{r}}
    \Vert (I -  \b \Pi^0_{k-1})\b \epsilon(\b e_h) \Vert^2_{\mathbb{L}^r(\Omega)}
    \\
    &\gtrsim 
    \left(  
    \Vert  \b u_h \Vert^r_{\b W^{1,r}(\Omega_h)} + 
    \Vert  \b w_h \Vert^r_{\b W^{1,r}(\Omega_h)}
    \right)^{\frac{r-2}{r}}
    \Vert (I - \b \Pi^0_{k-1})\b \epsilon(\b e_h) \Vert^2_{\mathbb{L}^r(\Omega)}  \,.
\end{aligned}    
\end{equation}
The proof follows collecting \eqref{eq:lemma15-2} and \eqref{eq:lemma15-3} in \eqref{eq:lemma15-1} and observing that by Korn inequality \eqref{eq:Korn} we have
$$
\Vert \b \Pi^0_{k-1}\b \epsilon(\b e_h) \Vert^2_{\mathbb{L}^r(\Omega)} +
\Vert (I - \b \Pi^0_{k-1})\b \epsilon(\b e_h) \Vert^2_{\mathbb{L}^r(\Omega)} 
\gtrsim
\Vert  \b \epsilon(\b e_h) \Vert^2_{\mathbb{L}^r(\Omega)}
\gtrsim
\Vert  \b e_h \Vert^2_{\b W^{1,r}(\Omega)} \,.
$$
\end{proof}
%%%%%%%%
\begin{lemma}[H\"older continuity  of $a_h(\cdot,\cdot)$] 
\label{lem:holder_ah}
Let $\b u_h$, $\b w_h \in \VDG$ and set $\b e_h:=\b u_h - \b w_h \in \VDG$. Then for any $\b v_h \in \VDG$ there holds
\begin{equation*}
\begin{aligned}
    |a_h(\b u_h, \b v_h) - a_h(\b w_h, \b v_h)| & \lesssim
    \|\b e_h\|^{r-1}_{\b W^{1,r}(\Omega)} 
    \| \b v_h\|_{\b W^{1,r}(\Omega)} \,.
\end{aligned}
\end{equation*}
\end{lemma}

\begin{proof}
Recalling the definition of $a_h(\cdot, \cdot)$ in \eqref{eq:forma ah} and definition \eqref{eq:Sglobal} we have
\begin{equation}
    \label{eq:lemmaholdh-0}
    |a_h(\b u_h, \b v_h) - a_h(\b w_h, \b v_h)| \le |T_1| + |T_2| 
\end{equation}
where
$$
    \begin{aligned}
    T_1 &: =
    \int_\Omega \b\sigma(\cdot, \b\Pi^0_{k-1} \b \epsilon(\b u_h)) : \b\Pi^0_{k-1} \b\epsilon(\b v_h) - 
\int_\Omega \b\sigma(\cdot, \b\Pi^0_{k-1} \b \epsilon(\b w_h)) : \b\Pi^0_{k-1} \b\epsilon(\b v_h)  \,,
\\
T_2 &:= S((I - \Pi^0_k) \b u_h, (I -  \Pi^0_k) \b v_h) -
    S((I -  \Pi^0_k) \b w_h, (I -  \Pi^0_k) \b v_h) \,.
    \end{aligned}
$$    
Using the same computations as in \eqref{eq:for-holder} and  the continuity of the $L^2$-projection w.r.t. the $L^2$-norm we have
\begin{equation}
    \label{eq:lemmaholdh-1}
    |T_1| \lesssim \|\b e_h\|^{r-1}_{\b W^{1,r}(\Omega)} 
    \| \b v_h\|_{\b W^{1,r}(\Omega)} \,.
\end{equation}
From \eqref{eq:continuity:local_stabG2} and Lemma \ref{lem:approx-interp} we infer
\begin{equation}
    \label{eq:lemmaholdh-2}
    |T_2| \lesssim \|(I - \Pi^0_k) \b e_h\|^{r-1}_{\b W^{1,r}(\Omega)} 
    \| (I - \Pi^0_k) \b  v_h\|_{\b W^{1,r}(\Omega)}
    \lesssim \|\b e_h\|^{r-1}_{\b W^{1,r}(\Omega)} 
    \| \b  v_h\|_{\b W^{1,r}(\Omega)}\,.
\end{equation}
The proof follows inserting \eqref{eq:lemmaholdh-1} and \eqref{eq:lemmaholdh-2} in \eqref{eq:lemmaholdh-0}.
\end{proof}
%%%%%%%%
%%%%%%%
\begin{theorem}[Well-posedness of \eqref{eq:stokes.vem.Z}]
    For any $r\in(1,2)$, there exists a unique solution $\b u_h \in \ZDG$ to the discrete problem \eqref{eq:stokes.vem.Z} satisfying the a-priori estimate (cf. \eqref{eq:Cfg})
\begin{equation}
\label{eq:a-priori.disc.u}
    \norm[\b{W}^{1,r}(\Omega)]{\b u_h}\lesssim
\mathcal{N}(\b f, \b g)^{\frac{1}{r-1}}  \,.
\end{equation}

\end{theorem}
\begin{proof}
\textit{(i) Existence.} 
Let the mesh $\Omega_h$ be fixed. We equip the space $\b \ZDG$ with the $\b H^1$-inner product, denoted by $(\cdot,\cdot)_{\b H^1}$, and corresponding induced norm $\norm[\b H^1(\Omega)]{\cdot}$. 
%defined by
%$$
%(\b v_h, \b w_h)_{\b \ZDG} = 
%\sum_{E\in \Omega_h} \b \DOF(\b v_h)\cdot \b \DOF(\b w_h), 
%\qquad \norm[\b \ZDG]{\b v_h}^2 = (\b v_h, \b v_h)_{\b \ZDG}, 
%\qquad\forall\ \b v_h, \b w_h\in \b \ZDG.
%$$
Owing to the Poincar\'e inequality, the equivalence of vector norms in finite-dimensional spaces, and the inverse estimate \eqref{eq:inv_norms}, we infer 
\begin{equation}\label{eq:inv_wellp}
    \norm[\b H^1(\Omega)]{\b v_h} \lesssim
    \left(\sum_{E\in\Omega_h} |\b v_h|_{\b H^1(E)}^2 \right)^\frac12 \lesssim 
    \left(\sum_{E\in\Omega_h} h_E^{r-2} |\b v_h|_{\b W^{1,r}(E)}^r \right)^\frac1r \lesssim
    \left(\min_{E\in\Omega_h} h_E\right)^{\frac{r-2}r}\norm[\b W^{1,r}(\Omega)]{\b v_h}.
\end{equation}
Let now $\b \Phi_h:\b \ZDG\to\b \ZDG$ be defined such that
$$
(\b \Phi_h(\b v_h), \b w_h)_{\b H^1(\Omega)} \coloneq a_h(\b v_h, \b w_h), 
\qquad\forall\ \b v_h, \b w_h\in \b \ZDG.
$$
Thus, the strong monotonicity of $a_h(\cdot, \cdot)$ established in Lemma \ref{lem:mono_ah} together with  the bound in \eqref{eq:inv_wellp} leads to, for any $\b v_h \in \b \ZDG$, 
$$
\begin{aligned}
\lim_{\norm[\b H^1(\Omega)]{\b v_h}\to\ \infty} 
\frac{(\b \Phi_h(\b v_h), \b v_h)_{\b H^1(\Omega)}}{\norm[\b H^1(\Omega)]{\b v_h}}
&\gtrsim
\lim_{\norm[\b H^1(\Omega)]{\b v_h}\to\ \infty} 
\frac{\norm[\b W^{1,r}(\Omega)]{\b v_h}^r}{\norm[\b H^1(\Omega)]{\b v_h}}
\\
&\gtrsim
\lim_{\norm[\b H^1(\Omega)]{\b v_h}\to\ \infty} 
\left(\min_{E\in\Omega_h} h_E\right)^{2-r}
\norm[\b H^1(\Omega)]{\b v_h}^{r-1} \to \infty.
\end{aligned}
$$
By applying \cite[Theorem 3.3]{Deimling:85}, the previous result shows that the operator $\b \Phi_h$ is surjective.
Let $\b z_h\in \b \ZDG$ be such that 
$$(\b z_h, \b w_h)_{\b H^1(\Omega)} = \int_\Omega \b f_h \cdot \b w_h + \int_{\Gamma_N} \b g\cdot \b\gamma(\b w_h).
$$
Hence, as a result of the surjectivity of $\b \Phi_h$ and the definition of $\b z_h$, there exists $\b u_h\in \b \ZDG$ such that $\b \Phi_h(\b u_h) = \b z_h$, namely $\b u_h$ is a solution to the discrete problem \eqref{eq:stokes.vem.Z}.

\medskip\noindent 
\textit{(ii) Uniqueness.}
Let $\b u_{h,1}, \b u_{h,2}\in \b \ZDG$ solve \eqref{eq:stokes.vem.Z}.
Subtracting \eqref{eq:stokes.vem.Z} for $\b u_{h,2}$ from \eqref{eq:stokes.vem.Z} for $\b u_{h,1}$ and then taking $\b v_h = \b u_{h,1} - \b u_{h,2}$ as test function, we obtain 
$$
a_h(\b u_{h,1},\b u_{h,1} - \b u_{h,2}) - a_h(\b u_{h,2},\b u_{h,1} - \b u_{h,2}) = 0.
$$
Thus, applying the strong monotonicity of $a_h(\cdot, \cdot)$ in Lemma \ref{lem:mono_ah} with $\b e_h = \b u_{h,1} - \b u_{h,2}$ and owing to Korn's first inequality \eqref{eq:Korn}, we get $\norm[\b W^{1,r}(\Omega)]{\b u_{h,1} - \b u_{h,2}} = 0$ and, as a result, $\b u_{h,1}=\b u_{h,2}$.

\medskip\noindent 
\textit{(iii) A-priori estimate.} We follow the same lines of Proposition \ref{prop:well-posed_cont} (with $\delta=0$), based on the strong monotonicity of $a_h(\cdot, \cdot)$ in Lemma \ref{lem:mono_ah}.
%$$
%  \begin{aligned}
%  \norm[\b{W}^{1,r}(\Omega)]{\b u_h}^r & \le
%  \norm[\b{W}^{1,r}(\Omega)]{\b u_h}^2
%  \norm[\b{W}^{1,r}(\Omega)]{\b u_h}^{r-2} \\
%  &\lesssim
%    \| {\boldsymbol\epsilon} (\b u_h) \|^2_{\b L^r(\Omega)} 
%    \left( 
%    \Vert \b \Pi^0_{k-1} {\boldsymbol\epsilon} (\b u_h)\Vert_{\b L^r(\Omega)}^{r-2} +
%    \sum_{E\in\Omega_h} h_E^{r-2} \| \DOF((I - \P0 )\b u_h) \|^{r} 
%    \right)^{\frac{r-2}{r}}
%    \\
%   \lesssim &\; a_h(\b u_h, \b u_h) \lesssim  
%    \left(\| \b f_h \|_{\b L^{r'}(\Omega)} +
%\| \b g \|_{\b L^{r'}(\Gamma_N)}\right)
%\norm[\b{W}^{1,r}(\Omega)]{\b u_h}.
%  \end{aligned}
%$$
%Then, the conclusion follows proceeding as in \eqref{eq:a-priori.1} and \eqref{eq:a-priori.2}.
\end{proof}
Thanks to the discrete inf-sup condition established in Lemma \ref{inf-sup:vem} and the equivalence of the discrete problems \eqref{eq:stokes.vem} and \eqref{eq:stokes.vem.Z}, the following result hold:
\begin{corollary}[Well-posedness of \eqref{eq:stokes.vem}]
For any $r\in(1,2)$, there exists a unique solution $(\b u_h, p_h) \in \b \VDG \times \QDG$ to the discrete problem \eqref{eq:stokes.vem}. Additionally, $p_h$ satisfies the a-priori bound (cf. \eqref{eq:Cfg})
    \begin{equation}\label{eq:discrete_apriori.p}
    \norm[L^{r'}(\Omega)]{p_h} \lesssim
    \mathcal{N}(\b f, \b g)  \,.
    \end{equation}
\end{corollary}
\begin{proof}
First, we observe that, by definition, the unique solution $\b u_h$ to \eqref{eq:stokes.vem.Z} satisfies the second equation in \eqref{eq:stokes.vem}.
Then, the existence and uniqueness of the discrete pressure solution is obtained as in the linear case $r=2$, cf. \cite[Theorem 4.2.1]{boffi-brezzi-fortin:book}.
The a-priori estimate \eqref{eq:discrete_apriori.p} follows, proceeding as in Proposition \ref{prop:well-posed_cont}, from the discrete inf-sup stability in Lemma \ref{inf-sup:vem} and the H\"older continuity of $a_h(\cdot, \cdot)$ in Lemma \ref{lem:holder_ah}.
\end{proof}

%----------------------------------------------------------%
\section{A-priori error analysis}
\label{sec:error_analysis}

\subsection{Additional properties of the stress-strain law}

We recall some important results regarding the stress-strain relation that are instrumental for the a-priori analysis of the scheme.
We mainly follow \cite[Section 3]{Berselli.Diening.ea:10} and \cite[Section 2]{Hirn:13}. First, we associate to the Sobolev exponent $r\in(1,2]$ the convex functions $\varphi, \varphi^*\in C^1(\mathbb{R}_0^+, \mathbb{R}_0^+)$ defined by
$$
\varphi(t) := \frac{t^r}r \qquad\text{and}\qquad
\varphi^*(t) := \frac{t^{r'}}{r'}.
$$
We also introduce, for $a\ge0$, the shifted functions $\varphi_a(t) := \int_0^t (a+s)^{r-2}s\, {\rm d} s$, which satisfy 
$$
(r-1)(a+t)^{r-2}\le \varphi_a''(t)\le (a+t)^{r-2},
$$
from which the convexity of $\varphi_a$ can be inferred and, as a result, $\varphi_a(t+s)\lesssim \varphi_a(t) +\varphi_a(s)$ uniformly in $t,s,a\in\mathbb{R}^+_0$. The following Lemma provides important properties of the shifted functions $\varphi_a$. We refer the reader to \cite[Lemmata 28--32]{Diening.Ettwein:08} for the detailed proof.

\begin{lemma}[Young type inequalities]
\label{lem:young_shifted}
For all $\varepsilon>0$ there exists $C(\varepsilon)>0$ only depending on $r$ and $\delta$ such that for all $s,t,a,\delta \ge 0$ and all $\b\tau,\b\eta\in\mathbb{R}^{d\times d}$ there holds
\begin{align}
\label{eq:Young1}
s\varphi_a'(t) +t\varphi_a'(s) &\le 
\varepsilon\varphi_a(s) + c(\varepsilon) \varphi_a(t),
\\
%\label{eq:Young2}
\varphi_{\delta+|\b\tau|}(t) &\le
\varepsilon \varphi_{\delta+|\b\eta|}(|\b\tau-\b\eta|)
+c(\varepsilon) \varphi_{\delta+|\b\eta|}(t).
\end{align}
\end{lemma}

The next result showing the equivalence of several quantities is strictly related to the continuity and monotonicity assumptions given in Assumption~\ref{ass:hypo}. 
%%
% First, we introduce a function $\b{\mathcal{F}}:\mathbb{R}^{d\times d}_{\rm s}\to \mathbb{R}^{d\times d}_{\rm s}$ which is closely related to the stress tensor in the Carreau--Yasuda model \eqref{eq:Carreau}:
% \begin{equation}\label{eq:addfun}
% \b{\mathcal{F}}(\b\eta) := (\delta + |\b\eta|)^{\frac{r-2}2} \b\eta.
% \end{equation}
%%
The proofs of the next lemma can be found in \cite[Section 2.3]{Hirn:13}. 
 The lemma here below applies to any (scalar or tensor valued) function $\b{\sigma}$ which satisfies Assumption~\ref{ass:hypo}. In the following, with a slight abuse of notation, we will apply such lemma both to the constitutive law $\b{\sigma}$ but also to the auxiliary scalar function $\sigma(\tau):= |\tau|^{r-2} \tau$.
\begin{lemma}\label{lm:Hirn}
Let $\b\sigma$ satisfy \eqref{eq:hypo} for $r\in(1,2]$ and $\delta \geq0$. 
% and let $\b{\mathcal{F}}$ be defined by \eqref{eq:addfun}. 
Then, uniformly for all $\b\tau,\b\eta\in \mathbb{R}^{d\times d}_{\rm s}$ and all $\b{v},\b{w}\in\b{U}$ there hold 
\begin{subequations}
%\label{eq:extrass}
  \begin{alignat}{2} 
  \label{eq:extrass.continuity}
  |\b{\sigma}(\cdot,\b{\tau})-\b{\sigma}(\cdot,\b{\eta})| & \simeq 
  \left(\delta+|\b\tau|+|\b\eta| \right)^{r-2}
  |\b\tau-\b\eta| \simeq \varphi_{\delta+|\b\tau|}'(|\b\tau - \b\eta|),
  \\
  \label{eq:extrass.monotonicity}
  \left(\b{\sigma}(\cdot,\b\tau)-\b{\sigma}(\cdot, \b\eta)\right):\left(\b\tau-\b\eta\right) &\simeq 
  \left(\delta+|\b\tau|+|\b\eta|\right)^{r-2} |\b\tau-\b\eta|^{2} \simeq 
  \varphi_{\delta+|\b\tau|}(|\b\tau - \b\eta|) ,
  % \simeq |\b{\mathcal{F}}(\b\tau)-\b{\mathcal{F}}(\b\eta)|^2,
  % \\
  % \label{eq:extrass.add1}
  % \norm[\mathbb{L}^r(\Omega)]{\b\epsilon(\b v -\b w)}^2 &\lesssim
  % \norm[\mathbb{L}^2(\Omega)]{\b{\mathcal{F}}(\b\epsilon(\b v)) -\b{\mathcal{F}}(\b\epsilon(\b w))}^2
  % \norm[\mathbb{L}^r(\Omega)]{\delta+|\b\epsilon(\b v)|+|\b\epsilon(\b w)|}^{2-r},
  % \\
  % \label{eq:extrass.add2}
  % \norm[\mathbb{L}^2(\Omega)]{\b{\mathcal{F}}(\b\epsilon(\b v)) -\b{\mathcal{F}}(\b\epsilon(\b w))}^2
  % &\lesssim \norm[\mathbb{L}^r(\Omega)]{\b\epsilon(\b v -\b w)}^r,
  % \\
  % \label{eq:extrass.add3}
  % \norm[\mathbb{L}^{r'}(\Omega)]{\b\sigma(\cdot,\b\epsilon(\b v)) -\b\sigma(\cdot,\b\epsilon(\b w))}^2
  % &\lesssim 
  % \norm[\mathbb{L}^2(\Omega)]{\b{\mathcal{F}}(\b\epsilon(\b v)) -\b{\mathcal{F}}(\b\epsilon(\b w))}^{\frac2{r'}},
  \end{alignat}
\end{subequations}
where the hidden constants only depend on  $\sigma_{\rm{c}}, \sigma_{\rm{m}}$ in Assumption~\ref{ass:hypo} and on $r$.
\end{lemma}

% ------------------------------------------
\subsection{Convergence}
In the present section, we derive the rate of convergence for the proposed VEM scheme \eqref{eq:stokes.vem} in the case $\delta=0$, see Remark \ref{rem:deltazero}.

\begin{proposition}\label{prop:main} 
Let $\b u$ be the solution of problem \eqref{eq:stokes.weak.Z} and let $\b u_h$ be the solution of problem \eqref{eq:stokes.vem.Z}. 
Assume that $\b u \in \b W^{k_1+1,r}(\Omega_h)$, 
$\b\sigma(\cdot, \b \epsilon(\b u)) \in \mathbb{W}^{k_2,r'}(\Omega_h)$, $\b f \in \b W^{k_3+1,r'}(\Omega_h)$ for some
$k_1$, $k_2$, $k_3 \le k$. 
Let the mesh regularity assumptions stated in Assumption~\ref{ass:mesh} hold. Then, we have
\begin{equation}\label{falcao}
    \|\b u -\b u_h\|_{\b W^{1,r}(\Omega_h)} \lesssim h^{k_1 r/2} R_1^{r/2} + h^{k_1} R_1 + h^{k_2} R_2 + h^{k_3+2} R_3\, ,
\end{equation}
where the hidden constant depends on $\mathcal{N}(\b f, \b g)$ (cf. \eqref{eq:Cfg}) and the regularity terms are
\begin{equation}\label{eq:reg:terms}
R_1 := | \b u |_{\b W^{k_1+1,r}(\Omega_h)} \, , \qquad
R_2 := |\b\sigma(\cdot, \b \epsilon(\b u))|_{\mathbb{W}^{k_2,r'}(\Omega_h)} \,, \qquad
R_3 := |\b f |_{\b W^{k_3+1,r'}(\Omega_h)} \, .
\end{equation}
\end{proposition}
\begin{proof}
We set $\b \xi_h:=\b u_h -\b u_I$. For the sake of brevity in the following we employ the notation $\widetilde{\b v}:=(I - \Pi^0_k) \b v$ for any $\b v \in \b L^2(\Omega)$ (that is, a tilde symbol over an $\b L^2$ function denotes the application of the $(I - \Pi^0_k)$ operator).
Manipulating \eqref{eq:stokes.vem.Z} and \eqref{eq:stokes.weak.Z} and recalling \eqref{eq:forma ah} we have
\begin{equation}
\label{eq:conv0}
\begin{aligned}
    &a_h(\b u_h, \b \xi_h) - a_h(\b u_I, \b \xi_h)
    =  
    a(\b u, \b \xi_h) - a_h(\b u_I, \b \xi_h) +
    (\b f_h -\b f , \b \xi_h )  
    \\
    &= 
    \left(\int_\Omega \b\sigma(\cdot, \b \epsilon(\b u)) : \b\epsilon( \b \xi_h) - \int_\Omega \b\sigma(\cdot, \b \Pi^0_{k-1}\b \epsilon(\b u_I)) : \b \Pi^0_{k-1} \b\epsilon( \b \xi_h)  \right) - S(\widetilde{\b u}_I, \widetilde{\b \xi}_h) 
    + (\b f_h -\b f , \b \xi_h )
    \\
    &=: T_1+T_2+T_3 \,.
\end{aligned}    
\end{equation}
We next estimate each term on the right-hand side above.
In the following $C$  will denote a generic positive constant independent of $h$ that may change at each occurrence, whereas the positive parameter $\theta$ adopted in \eqref{eq:T1A} and \eqref{eq:T3} will be specified later.

\noindent 
$\bullet$ \, Estimate of $T_1$. Employing the definition of $L^2$-projection \eqref{eq:P0_k^E} we have
\begin{equation}
\label{eq:T10}
\begin{aligned}
T_1 &= \int_\Omega \bigl(\b\sigma(\cdot, \b \epsilon(\b u)) - \b \Pi^0_{k-1}\b\sigma(\cdot, \b \Pi^0_{k-1} \b \epsilon(\b u_I)) \bigr) : \b\epsilon(\b \xi_h)
\\
&=
\int_\Omega \bigl((I - \b \Pi^0_{k-1})\b\sigma(\cdot, \b \epsilon(\b u))  \bigr) : \b\epsilon(\b \xi_h)+
\int_\Omega \bigl(\b\sigma(\cdot, \b \epsilon(\b u)) - \b\sigma(\cdot, \b \Pi^0_{k-1} \b \epsilon(\b u_I)) \bigr) : \b \Pi^0_{k-1} \b\epsilon(\b \xi_h)
\\
&=:T_1^A+ T_1^B \,.
\end{aligned}
\end{equation}
The term $T_1^A$ can be bounded as follows
\begin{equation}
\label{eq:T1A}
\begin{aligned}
    T_1^A &= \sum_{E \in \Omega_h}
    \int_E \bigl((I - \b \Pi^{0,E}_{k-1})\b\sigma(\cdot, \b \epsilon(\b u))  \bigr) : \b\epsilon(\b \xi_h) 
    \\
   & \le C \sum_{E\in \Omega_h} h_E^{k_2} |\b\sigma(\cdot, \b \epsilon(\b u))|_{\mathbb{W}^{k_2,r'}(E)} |{\b \xi}_h|_{\b W^{1,r}(E)} 
    & \quad & \text{($(r',r)$-H\"older ineq. \& Lemma \ref{lem:approx-interp})}
    \\
   &
   \le \frac{C}{\theta}  h^{2k_2} |\b\sigma(\cdot, \b \epsilon(\b u))|_{\mathbb{W}^{k_2,r'}(\Omega_h)}^2 + 
   \frac{\theta}{2} \Vert{\b \xi}_h\Vert_{\b W^{1,r}(\Omega)}^2 
    & \quad & \text{($(r',r)$-H\"older ineq. \& Young ineq.)}
\end{aligned}
\end{equation}
Employing \eqref{eq:extrass.continuity} with $\delta = 0$ we obtain 
   \[
   \begin{aligned}
    T_1^B &= 
    \sum_{E\in\Omega_h}\int_E \vert \b\sigma(\cdot, \b \epsilon(\b u)) - \b\sigma(\cdot, \PP0 \b \epsilon(\b u_I))\vert \, \vert \PP0 \b\epsilon(\b \xi_h)\vert 
    \\
    &\leq 
    \sum_{E\in\Omega_h}\int_E
    C\varphi_{|\PP0 \b \epsilon(\b u_I)|}'(|\PP0 \b \epsilon(\b u_I) - \b \epsilon(\b u)|) \,\vert \PP0 \b\epsilon(\b \xi_h)\vert \,.
    \end{aligned}
    \]
Employing \eqref{eq:Young1} we get that for every $\varepsilon>0$ there exists a positive constant $C(\varepsilon)$ such that 
\[
\begin{aligned}
     T_1^B &\le
    \varepsilon \sum_{E\in\Omega_h}\int_E
    \varphi_{|\PP0 \b \epsilon(\b u_I)|}(|\PP0 \b \epsilon(\b \xi_h)|) + 
    C(\varepsilon) \sum_{E\in\Omega_h}\int_E\varphi_{|\PP0 \b \epsilon(\b u_I)|}(|\PP0 \b \epsilon(\b u_I) - \b \epsilon(\b u)|) .
\end{aligned}
\]     
Using \eqref{eq:extrass.monotonicity}  we obtain 
(with $\gamma$ denoting the associated uniform hidden constant)
\[
\begin{aligned}
T_1^B &\le
    \gamma \varepsilon
    (\b\sigma(\cdot, \b\Pi_{k-1}^0 \b \epsilon(\b u_h))- \b\sigma(\cdot, \b\Pi_{k-1}^0 \b \epsilon(\b u_I)), \b\Pi_{k-1}^0 \b\epsilon(\b \xi_h)) \\
    & +C(\varepsilon) 
    \sum_{E\in\Omega_h} \int_E (|\PP0 \epsilon(\b u_I)|+|\b\epsilon(\b u)|)^{r-2}  
    |\b\epsilon(\b u) - \PP0 \epsilon(\b u_I)|^2  \,.   
\end{aligned}
\]
Notice that the constant $C(\varepsilon)$ may depend on $\b \sigma$, the degree $k$, the domain $\Omega$ but, given our mesh assumptions, it is independent of the particular mesh or mesh element within the family $\{ \Omega_h \}_h$. 
Recalling that $r-2<0$, and employing the fact that $|\b\epsilon(\b u) - \PP0 \epsilon(\b u_I)|\leq |\b\epsilon(\b u)| +| \PP0 \epsilon(\b u_I)|$, from the last equation we get
\[
T_1^B\le
    \gamma \varepsilon
    (\b\sigma(\cdot, \b\Pi_{k-1}^0 \b \epsilon(\b u_h))- \b\sigma(\cdot, \b\Pi_{k-1}^0 \b \epsilon(\b u_I)), \b\Pi_{k-1}^0 \b\epsilon(\b \xi_h))
    +C(\varepsilon) 
    \sum_{E\in\Omega_h} \| \b\epsilon(\b u) - \PP0 \epsilon(\b u_I)\|^r_{\mathbb{L}^r(E)} \, ,
\]
which, making use of Lemma \ref{lem:approx-interp}, and taking $\varepsilon= \frac{1}{2 \gamma}$ becomes
\begin{equation}
\label{eq:T1B}
T_1^B\le
    \frac{1}{2}
    (\b\sigma(\cdot, \b\Pi_{k-1}^0 \b \epsilon(\b u_h))- \b\sigma(\cdot, \b\Pi_{k-1}^0 \b \epsilon(\b u_I)), \b\Pi_{k-1}^0 \b\epsilon(\b \xi_h))
    + C h^{k_1r} | \b u |_{W^{k_1+1,r}(\Omega_h)}^r \, .
\end{equation}
Combining \eqref{eq:T1A} and \eqref{eq:T1B} in \eqref{eq:T10} we infer
\begin{equation}
\label{eq:T1}
\begin{aligned}
T_1 &\le
    \frac{1}{2}
    (\b\sigma(\cdot, \b\Pi_{k-1}^0 \b \epsilon(\b u_h))- \b\sigma(\cdot, \b\Pi_{k-1}^0 \b \epsilon(\b u_I)), \b\Pi_{k-1}^0 \b\epsilon(\b \xi_h)) +
    \\
    & \quad + C h^{k_1r} | \b u |_{\b W^{k_1+1,r}(\Omega_h)}^r +
    \frac{C}{\theta}  h^{2k_2} |\b\sigma(\cdot, \b \epsilon(\b u))|_{\mathbb{W}^{k_2,r'}(\Omega_h)}^2 + 
   \frac{\theta}{2} \Vert{\b \xi}_h\Vert_{\b W^{1,r}(\Omega)}^2 \, .
\end{aligned}    
\end{equation}
\noindent 
$\bullet$ \, Estimate of $T_2$. 
Recalling definitions \eqref{eq:Sglobal} and \eqref{eq:dofi} 
with $\delta=0$ we have
\[
    T_2 = - \sum_{E \in \Omega_h}S^E(\widetilde{\b u}_I, \widetilde{\b \xi}_h)
    \le \sum_{E\in \Omega_h} h_E^{2-r} |\DOF(\widetilde{\b u}_I)|^{r-2} \vert\DOF (\widetilde{\b u}_I) \vert \, \vert \DOF (\widetilde{\b \xi}_h)\vert \,.
\]
Employing \eqref{eq:extrass.continuity} in Lemma \ref{lm:Hirn} to the scalar function   
$\sigma(\tau)= |\tau|^{r-2} \tau$ (hence $\delta=0$ and $\eta =0$) we have
\[
T_2 \le \sum_{E \in \Omega_h} C \varphi'_{|\DOF (\widetilde{\b u}_I)|}(|\DOF (\widetilde{\b u}_I)|) \, \vert \DOF (\widetilde{\b \xi}_h)\vert \,.
\]
Employing \eqref{eq:Young1} we get that for every $\varepsilon>0$ there exists a positive constant $C(\varepsilon)$ such that 
\[
\begin{aligned} 
    T_2 & \le \varepsilon \sum_{E\in \Omega_h} h_E^{2-r} 
      \varphi_{|\DOF (\widetilde{\b u}_I)|}(|\DOF(\widetilde{\b \xi}_h)|)
    + C(\varepsilon) \sum_{E\in \Omega_h} h_E^{2-r}  \varphi_{|\DOF (\widetilde{\b u}_I)|}(|\DOF (\widetilde{\b u}_I)|) \,.
\end{aligned}
\]
We now use \eqref{eq:extrass.monotonicity} and, denoting with $\gamma$ the hidden constant, we infer
\begin{equation}
\label{eq:T20}
\begin{aligned} 
T_2 & \le \varepsilon \gamma \sum_{E\in \Omega_h} h_E^{2-r} 
      (|\DOF (\widetilde{\b u}_I)| + |\DOF (\widetilde{\b u}_h)|)^{r-2} |\DOF(\widetilde{\b \xi}_h)|^2 + C(\varepsilon) \sum_{E\in \Omega_h} h_E^{2-r}|\DOF (\widetilde{\b u}_I)|^r =: T_2^A + T^2_B \,.
\end{aligned}
\end{equation}
Employing \eqref{eq:strong1} in Lemma \ref{lm:monotonicity:local_stab} (with  $\b u_h = \widetilde{\b u}_h$ and  $\b w_h = \widetilde{\b u}_I$) and recalling definition \eqref{eq:Sglobal} we obtain
\begin{equation}
\label{eq:T2A}
T_2^A \le \varepsilon C  \gamma \sum_{E \in \Omega_h}
(S^E(\widetilde{\b u}_h, \widetilde{\b \xi}_h) - 
S^E(\widetilde{\b u}_I, \widetilde{\b \xi}_h) ) 
= \varepsilon C  \gamma
(S(\widetilde{\b u}_h, \widetilde{\b \xi}_h) - 
S(\widetilde{\b u}_I, \widetilde{\b \xi}_h) ) \,.
\end{equation}
Employing definition \eqref{eq:dofi}, Lemma \ref{Lem:vemstab-2} (with $\b v_h=\widetilde{\b u}_I$), a triangular inequality and Lemma \ref{lem:approx-interp} we infer
\begin{equation}
\label{eq:T2B}
\begin{aligned}
     T_2^B &\le
     C(\varepsilon) \sum_{E\in \Omega_h} S^E(\widetilde{\b u}_I, \widetilde{\b u}_I) \le 
     C(\varepsilon) \sum_{E\in \Omega_h} \vert \widetilde{\b u}_I\vert^r_{\b W^{1, r}(E)}
     \\
     &\le 
     C(\varepsilon) \sum_{E\in \Omega_h} \left(\vert \widetilde{\b u} - \widetilde{\b u}_I\vert^r_{\b W^{1, r}(E)} 
     + \vert  \widetilde{\b u} \vert^r_{\b W^{1, r}(E)} \right)
      \le C(\varepsilon) h^{k_1r}  | \b u |^r_{\b W^{k_1+1,r}(\Omega_h)} \, .
   \end{aligned} 
\end{equation}
Combining \eqref{eq:T2A} and \eqref{eq:T2B} in \eqref{eq:T20} and taking $\epsilon = \frac{1}{2 C \gamma}$ we infer
\begin{equation}
\label{eq:T2}
T_2 \le 
\frac{1}{2}
(S(\widetilde{\b u}_h, \widetilde{\b \xi}_h) - 
S(\widetilde{\b u}_I, \widetilde{\b \xi}_h) ) + 
C h^{k_1r}  | \b u |^r_{\b W^{k_1+1,r}(\Omega_h)} \, .
\end{equation}

\noindent 
$\bullet$ \, Estimate of $T_3$. We infer
\begin{equation}
\label{eq:T3}
\begin{aligned}
    T_3 &= \sum_{E \in \Omega_h}\int_E (\Pi^{0,E}_k \b f - \b f) \cdot ({\b \xi}_h - \Pi^{0,E}_0 ({\b \xi}_h)) 
    & \quad & \text{(def. \eqref{eq:forma fh} \& def. \eqref{eq:P0_k^E})}
    \\
   & \le C \sum_{E\in \Omega_h} h_E^{k_3+2} |\b f |_{\b W^{k_3+1,r'}(E)} |{\b \xi}_h|_{\b W^{1,r}(E)} 
    & \quad & \text{($(r',r)$-H\"older ineq. \& Lemma \ref{lem:approx-interp})}
    \\
   &
   \le \frac{C}{\theta}  h^{2k_3+4} |\b f |_{\b W^{k_3+1,r'}(\Omega_h)}^2 + 
   \frac{\theta}{2} \Vert{\b \xi}_h \Vert_{\b W^{1,r}(\Omega)}^2 
    & \quad & \text{($(r',r)$-H\"older ineq. \& Young ineq.)}
\end{aligned}
\end{equation}
Plugging the estimates in \eqref{eq:T1}, \eqref{eq:T2} and \eqref{eq:T3} in \eqref{eq:conv0} and recalling definition \eqref{eq:forma ah} we obtain
\begin{equation}
\label{eq:upper}
\begin{aligned}
\frac{1}{2} &(a_h(\b u_h, \b \xi_h) - a_h(\b u_I, \b \xi_h))    
    \le  \\
&    C h^{k_1r}  | \b u |^r_{\b W^{k_1+1,r}(\Omega_h)} +
    \frac{C}{\theta}  h^{2k_2} |\b\sigma(\cdot, \b \epsilon(\b u))|_{\mathbb{W}^{k_2,r'}(\Omega_h)}^2 +
    \frac{C}{\theta}  h^{2k_3+4} |\b f |_{\b W^{k_3+1,r'}(\Omega_h)}^2 + 
   \theta \Vert{\b \xi}_h \Vert_{\b W^{1,r}(\Omega)}^2 
   \, .
\end{aligned}
\end{equation}
We now derive a lower bound for the left-hand side of the equation above. 
Note that, since from bounds \eqref{eq:a-priori_cont.u} and \eqref{eq:a-priori.disc.u}, estimate \eqref{falcao} clearly holds true if $\mathcal{N}(\b f, \b g)= 0$, it is not restrictive to assume $\mathcal{N}(\b f, \b g) > 0$.
Employing Lemma \ref{lem:mono_ah}, a triangular inequality, the stability bounds \eqref{eq:a-priori_cont.u} and \eqref{eq:a-priori.disc.u} and the interpolation bound in Lemma \ref{lem:approx-interp} we get
\begin{equation}
\label{eq:lower}
\begin{aligned}
\frac{1}{2} (a_h(\b u_h, \b \xi_h) - a_h(\b u_I, \b \xi_h)) & \ge
C \left(\|\b u_h\|^r_{\b W^{1,r}(\Omega)} +
    \| \b u_I\|^r_{\b W^{1,r}(\Omega)} \right)^{\frac{r-2}{r}} \, \Vert  \b \xi_h \Vert^2_{\b W^{1,r}(\Omega)}
\\
&\ge
C \left(\|\b u_h\|^r_{\b W^{1,r}(\Omega)} + \|\b u\|^r_{\b W^{1,r}(\Omega)} +  \|\b u -  \b u_I\|^r_{\b W^{1,r}(\Omega)} \right)^{\frac{r-2}{r}} \, \Vert  \b \xi_h \Vert^2_{\b W^{1,r}(\Omega)}
\\
& \ge
C \left(\mathcal{N}(\b f, \b g)^{\frac{r}{r-1}} +  h^{k_1r}|\b u|^r_{\b W^{k_1+1,r}(\Omega)} \right)^{\frac{r-2}{r}} \, \Vert  \b \xi_h \Vert^2_{\b W^{1,r}(\Omega)}
\\
& \ge
C \mathcal{N}(\b f, \b g)^{\frac{r-2}{r-1}}  \, \Vert  \b \xi_h \Vert^2_{\b W^{1,r}(\Omega)}
\end{aligned}
\end{equation}
where in the last inequality we assume that $h$ is sufficiently small.
Combining  \eqref{eq:upper} and \eqref{eq:lower} and taking $\theta= \frac{1}{2} C \mathcal{N}(\b f, \b g)^{\frac{r-2}{r-1}}$ we obtain:
\begin{equation}
\label{eq:pressF}
\Vert  \b \xi_h \Vert^2_{\b W^{1,r}(\Omega)} \lesssim
C(\mathcal{N}(\b f, \b g))\left( 
 h^{k_1r}  | \b u |^r_{\b W^{k_1+1,r}(\Omega_h)} +
    h^{2k_2} |\b\sigma(\cdot, \b \epsilon(\b u))|_{\mathbb{W}^{k_2,r'}(\Omega_h)}^2 +
      h^{2k_3+4} |\b f |_{\b W^{k_3+1,r'}(\Omega_h)}^2
\right) \,.
\end{equation}
The proof follows by the previous bound, a triangular inequality and Lemma \ref{lem:approx-interp}. 
\end{proof}

%The following Corollary follows immediately combining a triangle inequality, Proposition \ref{prop:main} and Lemma \ref{lem:approx-interp}.
%\begin{corollary}
%Under the same assumptions of Proposition \ref{prop:main} it holds
%\begin{equation}
%    \|\b u -\b u_h\|_{W^{1,r}(\Omega)} \lesssim h^{kr/2} \, ,
%\end{equation}
%where, for the sake of conciseness, we included the regularity terms depending on $\b u$ in the symbol $\lesssim$.
%\end{corollary}
%%%%
\begin{proposition}\label{prop:main:2} 
Let $(\b u, p)$ be the solution of problem \eqref{eq:stokes.weak} and $(\b u_h, p_h)$ the solution of problem \eqref{eq:stokes.vem}. 
Assume that $\b u \in \b W^{k_1+1,r}(\Omega_h)$, 
$\b\sigma(\cdot, \b \epsilon(\b u)) \in \mathbb{W}^{k_2,r'}(\Omega_h)$, 
$\b f \in \b W^{k_3+1,r'}(\Omega_h)$
$p \in  W^{k_4,r'}(\Omega_h)$, for some $k_1$, $k_2$, $k_3$, $k_4 \le k$. Let the mesh regularity  in Assumption~\ref{ass:mesh} hold. Then, we have
\begin{equation*}
%\label{falcaop}
    \|p - p_h\|_{L^{r'}(\Omega_h)} \lesssim 
    \Vert \b u - \b u_h\Vert_{\b W^{1,r}(\Omega_h)}^{2/r'} +
h^{k_1(r-1)} R_1^{r-1} + h^{k_2} R_2 +  h^{k_3+2} R_3 +
    h^{k_4} R_4  \,,
\end{equation*}
where the regularity term $R_1$, $R_2$, $R_3$ are defined in \eqref{eq:reg:terms} and
$R_4 := |p |_{W^{k_4,r'}(\Omega_h)}$.
\end{proposition}

\begin{proof}
Let $p_I := \Pi^0_{k-1} p \in \Pk_{k-1}(\Omega_h)$  and let $\rho_h:=p_h-p_I$. Employing  \eqref{eq:stokes.weak} and  \eqref{eq:stokes.vem}, recalling the definition of the form $b(\cdot, \cdot)$ in \eqref{eq:a.b} and combining item $(ii)$ in \eqref{eq:v-loc} with the definition of $L^2$-projection, for all $\b w_h\in \b U_h$ we get
   \begin{equation}
   \label{eq:press0}
   \begin{aligned}
   b(\b w_h,\rho_h) &=a_h(\b u_h, \b w_h) -  (\b f_h,\b w_h)
   -a(\b u, \b w_h)  + (\b f,\b w_h) + b(\b w_h,p-p_I) & &
   \\
   & =a_h(\b u_h, \b w_h)-  a(\b u, \b w_h) + (\b f -\b f_h,\b w_h)
   \\
   & = \int_\Omega \left((\b \Pi^0_{k-1} - I)\b\sigma(\cdot, \b \epsilon(\b u)) \right): \b\epsilon( \b \w_h) +   (\b f -\b f_h,\b w_h) + 
       \\
      & + \int_\Omega \left(\b\sigma(\cdot, \b \Pi^0_{k-1}\b \epsilon(\b u_h)) -
        \b\sigma(\cdot, \b \epsilon(\b u))
        \right): \b \Pi^0_{k-1} \b\epsilon( \b \w_h) +
       S(\widetilde {\b u_h},\widetilde {\b w_h})  
    \\
    & =: T_1 + T_2 + T_3 + T_4 \,.
   \end{aligned}
   \end{equation}
We estimate separately each term in \eqref{eq:press0}.
Employing the H\"older inequality with exponents $(r',r)$ and polynomial approximation properties  we infer
   \begin{equation}
   \label{eq:pressT1-2}
   \begin{aligned}
       T_1  &\lesssim 
       \| (\b \Pi^0_{k-1} - I)\b\sigma(\cdot, \b \epsilon(\b u))\|_{\mathbb L^{r'}(\Omega_h)} \, \|\b \epsilon(\b w_h)\|_{\mathbb L^r(\Omega_h)}
       \lesssim  h^{k_2} | \b\sigma(\cdot, \b \epsilon(\b u))|_{\mathbb W^{k_2,r'}(\Omega_h)} \, \|\b w_h\|_{\b W^{1,r}(\Omega_h)} \,,
   \\
   T_2 &= (\b f - \b f_h, \b w_h - \Pi^0_0 \b w_h) 
   %\leq
   %\Vert \b f -\b f_h \Vert_{\b L^{r'}(\Omega_h)} \Vert\b (I- \Pi^0_0) \b w_h\Vert_{\b L^r(\Omega_h)}
   \lesssim h^{k_3+2}\vert {\bf f} \vert_{\b W^{k_3+1,r'}(\Omega_h)} \, \|{\b w}_h\|_{W^{1,r}(\Omega_h)}   \,.
   \end{aligned}
   \end{equation}
   Employing the H\"older inequality with exponents $(r',r)$, the $ W^{1,r}$-stability of $\b \Pi^0_{k-1}$ and Lemma \ref{Lemma:A} in the Appendix, we have
   \begin{equation}
   \label{eq:pressT3}
   \begin{aligned}
       T_3^{r'} &\lesssim
       \| \b\sigma(\cdot, \b \Pi^0_{k-1} \b \epsilon(\b u_h)) - \b\sigma(\cdot, \b \epsilon(\b u))\|_{\mathbb L^{r'}(\Omega_h)}^{r'} \, \|\b \Pi^0_{k-1}  \b \epsilon(\b w_h)\|_{\mathbb L^r(\Omega_h)}^{r'}
      \\
      & \lesssim
      \left( (\b\sigma(\cdot, \b\Pi_{k-1}^0 \b \epsilon(\b u_h))- \b\sigma(\cdot, \b\Pi_{k-1}^0 \b \epsilon(\b u_I)), \b\Pi_{k-1}^0 \b\epsilon(\b \xi_h))  + h^{k_1r} \vert \b u\vert_{\b W^{k_1+1,r}(\Omega_h)}^r \right) 
      \|\b \w_h\|_{\b W^{1,r}(\Omega_h)}^{r'}
   \end{aligned}    
   \end{equation}
   where $\b \xi_h = \b u_h - \b u_I$ (cf. proof of Proposition \ref{prop:main}). 
Employing Lemma \ref{Lemma:S_holderG}, Corollary \ref{Cor:vemstab},  the $W^{1,r}$-stability of $\Pi^0_{k-1}$
and finally Lemma \ref{Lemma:B} in the Appendix we obtain 
\begin{equation}
\label{eq:pressT4}
\begin{aligned}
    T_4^{r'} &\lesssim \|\widetilde {\b u}_h \|^{r}_{\b W^{1,r}(\Omega_h)}
    \|{\widetilde{\b w}}_h\|_{\b W^{1,r}(\Omega_h)}^{r'}
    \lesssim S(\widetilde {\b u}_h, \widetilde {\b u}_h)
    \|\b \w_h\|_{\b W^{1,r}(\Omega_h)}^{r'}
    \\
    & \lesssim
    \left( S(\widetilde {\b u}_h, \widetilde {\b \xi}_h)  -  
    S(\widetilde {\b u}_I, \widetilde {\b \xi}_h) + 
    h^{k_1r} \vert \b u\vert_{\b W^{k_1+1,r}(\Omega_h)}^r \right)
     \|\b w_h\|_{W^{1,r}(\Omega_h)}^{r'} \,.
\end{aligned}    
\end{equation}
Therefore recalling the definition of $a_h(\cdot, \cdot)$ in \eqref{eq:forma ah}, from \eqref{eq:pressT3} and \eqref{eq:pressT4}, and recalling \eqref{eq:upper} and \eqref{eq:pressF}, we have
\begin{equation}
\label{eq:pressT3-4}
\begin{aligned}
T_3^{r'} + T_4^{r'} &\lesssim 
\left(a_h(\b u_h, \b \xi_h)  - a_h(\b u_I, \b \xi_h) + h^{k_1r} \vert \b u\vert_{\b W^{k_1+1,r}(\Omega_h)}^r \right) 
     \|\b w_h\|_{W^{1,r}(\Omega_h)}^{r'}
\\
& \lesssim
\left(\Vert \b \xi_h\Vert_{\b W^{1,r}(\Omega_h)}^2 + h^{k_1r} \vert \b u\vert_{\b W^{k_1+1,r}(\Omega_h)}^r \right) 
     \|\b w_h\|_{W^{1,r}(\Omega_h)}^{r'} 
\\
& \lesssim
\left(\Vert \b u - \b u_h\Vert_{\b W^{1,r}(\Omega_h)}^2 + h^{2k_1} \vert \b u\vert_{\b W^{k_1+1,r}(\Omega_h)}^2
+ h^{k_1r} \vert \b u\vert_{\b W^{k_1+1,r}(\Omega_h)}^r \right)  
     \|\b w_h\|_{W^{1,r}(\Omega_h)}^{r'}  \,.
\end{aligned}     
\end{equation}
Collecting \eqref{eq:pressT1-2} and \eqref{eq:pressT3-4} in \eqref{eq:press0} we obtain
\[
b(\b w_h, \rho_h) \lesssim 
\left( \Vert \b u - \b u_h\Vert_{\b W^{1,r}(\Omega_h)}^{2/r'} +
h^{k_1(r-1)} R_1^{r-1} + h^{k_2} R_2 +  h^{k_3+2} R_3 
\right) 
     \|\b w_h\|_{W^{1,r}(\Omega_h)}
\quad \text{$\forall  \b w_h \in \VDG$.}     
\]
 Employing the discrete inf-sup (Lemma \ref{inf-sup:vem}), from the equation above we have
   \[
   \begin{aligned}
   \|\rho_h\|_{L^{r'}(\Omega_h)}
    &   \le \frac{1}{\bar{\beta}(r)}
       \sup_{\b w_h \in \VDG}
       \frac{b(\b w_h,\rho_h)}{\|\b w_h\|_{W^{1,r}(\Omega_h)}}
       \\
   &    \lesssim \frac{1}{\bar{\beta}(r)} 
       \left( \Vert \b u - \b u_h\Vert_{\b W^{1,r}(\Omega_h)}^{2/r'} +
h^{k_1(r-1)} R_1^{r-1} + h^{k_2} R_2 +  h^{k_3+2} R_3 
\right) 
       \,.
   \end{aligned}    
   \]
Now the thesis follows from triangular inequality and standard polynomial approximation property.
\end{proof}

\begin{remark}{(Rates of convergence)}
\label{rm:reg}
Notice that, assuming in Proposition \ref{prop:main} and Proposition \ref{prop:main:2} maximum regularity for the solution and for the data (that is $k_i=k$ for $i=1,\dots, 4$) we recover the following asymptotic behaviour 
\[
\Vert \b u - \b u_h\Vert_{\b W^{1,r}(\Omega_h)} \lesssim h^{kr/2} \,,
\qquad
\Vert p - p_h\Vert_{L^{r'}(\Omega_h)} \lesssim h^{k(r-1)} \,
\]
that are in agreement with the classical results presented, e.g., in \cite{Barrett.Liu:94} in the context of finite elements. 
\end{remark}

%\medskip

%%%%%%%%%%%%%%%%%%%%%%%%%%%%%%%%%%%%%%%%%%%%%%%%%%%%%%%%%%%%%%%%%%%%%%%%%%%%%%%%%%%%%%%%%%%%%%%%%%%%%%%%%%%%%%%%%%%%%%%%%%%%%%%%%%%%%%%%%%%%%%%%%%%%%%%%%%%%
\section{Numerical results}\label{sec:numerics}

In this section, we present two sets of numerical experiments that corroborate our theoretical findings and test the practical performances of the proposed scheme \eqref{eq:stokes.vem} in various scenarios. In particular, despite not covered by our theory, we will also explore the case $\delta\not=0$ in \eqref{eq:Carreau}.
\subsection{Fixed-point iteration}
\label{sub:fixpoint}
We first describe the linearization strategy based on a fixed-point 
iteration adopted to solve the non-linear Stokes equation.
We consider the Carreau--Yasuda model \eqref{eq:Carreau} and for any $r \in(1,2]$ we introduce the following notation:
\[
\begin{aligned}
\b{\sigma}_r(\b{x}, \GRADs(\b{z}); \GRADs(\b{v})) &:= 
\mu(\b x) (\delta^\alpha + |\GRADs(\b{z})|^\alpha)^{\frac{r-2}\alpha} \GRADs(\b{v}) 
&\quad &\text{for all $\b{z}$, $\b{v} \in \b U$}\,,
\\
a_{r}(\b z; \b v,  \b w) &:= 
\int_{\Omega}\b{\sigma}_r(\cdot,  \GRADs(\b{z}); \GRADs(\b{v})) : \GRADs(\b{w}) 
&\quad &\text{for all $\b z$, $\b v$, $\b w \in \b U$} \,,
\end{aligned}
\]
furthermore $a_{h, r}(\cdot; \cdot,  \cdot)$ will denote the VEM counterpart of $a_{r}(\cdot; \cdot,  \cdot)$ obtained generalizing in the linear setting the construction in Subsection \ref{sub:forms}.
Consider the discrete Problem \eqref{eq:stokes.vem} associated with the Carreau--Yasuda model with parameter $r \in (1, 2]$.
Then, in the light of the notation introduced above, Problem \eqref{eq:stokes.vem} can be formulated as follows:
find $(\b u_h, p_h) \in \VDG \times \QDG$ such that
\begin{equation}\label{eq:stokesL.vem}
  \begin{aligned}
     a_{r,h}(\b u_h; \b u_h,\b v_h)+b(\b v_h, p_h) &= \int_\Omega \b f_h \cdot \b v_h + \int_{\Gamma_N} \b g \cdot \b\gamma(\b v_h)&\qquad \forall \b v_h \in \VDG , \\
     -b(\b u_h, q_h) &= 0 &\qquad \forall q_h \in \QDG \,. 
  \end{aligned}
\end{equation}
To solve the non-linear equation \eqref{eq:stokesL.vem} we adopt the following strategy:

\vspace{0.2cm}

\noindent
\texttt{STEP 1.} Let $(\b u^{\rm S}_h, \, p^{\rm S}_h)$ be the solution of the linear Stokes equation associated with \eqref{eq:stokesL.vem} (corresponding to $r=2$).

\noindent
Let $\overline{r} := \frac{r+2}{2}$.
Starting from $(\b u^0_h, \, p^0_h) =(\b u^{\rm S}_h, \, p^{\rm S}_h)$, for $n \geq 0$, until convergence solve 
\[
  \begin{aligned}
     a_{\overline{r},h}(\b u_h^{n}; \b u_h^{n+1},\b v_h)+b(\b v_h, p_h^{n+1}) &= \int_\Omega \b f_h \cdot \b v_h + \int_{\Gamma_N} \b g \cdot \b\gamma(\b v_h)&\quad \forall \b v_h \in \VDG , \\
     -b(\b u_h^{n+1}, q_h) &= 0 &\quad \forall q_h \in \QDG \,.
  \end{aligned}
\]

\noindent
\texttt{STEP 2.} Let $(\b u^{\overline{r}}_h, \, p^{\overline{r}}_h)$ be the solution obtained in \texttt{STEP 1}.

\noindent
Starting from $(\b u^0_h, \, p^0_h) =(\b u^{\overline{r}}_h, \, p^{\overline{r}}_h)$, for $n \geq 0$, until convergence solve 
\[
  \begin{aligned}
     a_{r,h}(\b u_h^{n}; \b u_h^{n+1},\b v_h)+b(\b v_h, p_h^{n+1}) &= \int_\Omega \b f_h \cdot \b v_h + \int_{\Gamma_N} \b g \cdot \b\gamma(\b v_h)&\quad \forall \b v_h \in \VDG , \\
     -b(\b u_h^{n+1}, q_h) &= 0 &\quad \forall q_h \in \QDG \,.  
  \end{aligned}
\]

\subsection{Error quantities}
\label{sub:error}

To compute the VEM error between the exact solution $\b (u_{\rm ex}, p_{\rm ex})$ and the VEM
solution $(\b u_h, p_h)$ we consider the computable error quantities 
\[
\begin{aligned}
\texttt{err}(\b u_h, W^{1,r}) &:=
\frac{\|\GRAD \b u_{\rm ex} - \b \Pi^{0}_{k-1} \GRAD \b u_h\|_{L^r(\Omega)}}{\|\GRAD \b u_{\rm ex}\|_{\mathbb{L}^r(\Omega)}}  \,,
\\
\texttt{err}(p_h, L^{r'}) &:= \frac{\|p_{\rm ex} - p_h\|_{L^{r'}(\Omega)}}
{\|p_{\rm ex}\|_{L^{r'}(\Omega)}}
\\
\texttt{err}(\b \sigma, L^{r'}) &:=
\frac{\|\b\sigma(\cdot, \b \epsilon(\b u_{\rm ex})) - \b\sigma(\cdot, \b\Pi_{k-1}^0 \b \epsilon(\b u_h))\|_{L^{r'}(\Omega)}}
{\|\b\sigma(\cdot, \b \epsilon(\b u_{\rm ex}))\|_{\mathbb{L}^{r'}(\Omega)}}  \, .
% \\
% \texttt{err}(\b{\mathcal{F}}, L^2) &:=
% \frac{\|\b{\mathcal{F}}(\b\epsilon(\b u_{\rm ex})) - \b{\mathcal{F}}(\b \Pi^{0}_{k-1} \b\epsilon(\b u_h))\|_{L^2(\Omega)}}
% {\|\b{\mathcal{F}}(\b\epsilon(\b u_{\rm ex}))\|_{L^{2}(\Omega)} } \,.
\end{aligned}
\]

 In particular, we will compare the observed experimental convergence rates with the ones expected from the theory. In this last respect, we refer, for the case $\delta=0$,  to Propositions \ref{prop:main} and  \ref{prop:main:2} (see also Remark \ref{rm:reg}), while  we expect, for the case $\delta\not=0$,   optimal convergence rates (cf. \cite[Theorem 4.1]{Barrett.Liu:94}), i.e. 
\[
\Vert \b u - \b u_h\Vert_{\b W^{1,r}(\Omega_h)} \lesssim h^{k} \,,
\qquad
\Vert p - p_h\Vert_{L^{r'}(\Omega_h)} \lesssim h^{k}.
\]

In the forthcoming tests, we consider $k=2$.
\begin{figure}
\centering
\begin{overpic}[scale=0.23]{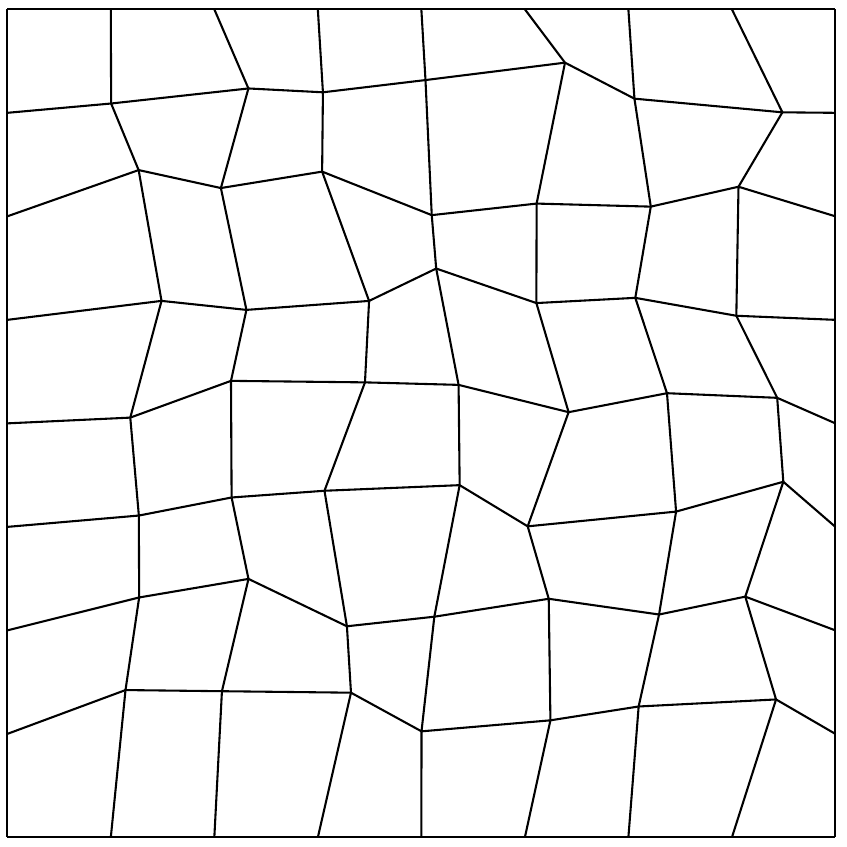}
\put(10,-15){{{\texttt{QUADRILATERAL}}}}
\end{overpic}
\qquad
\begin{overpic}[scale=0.23]{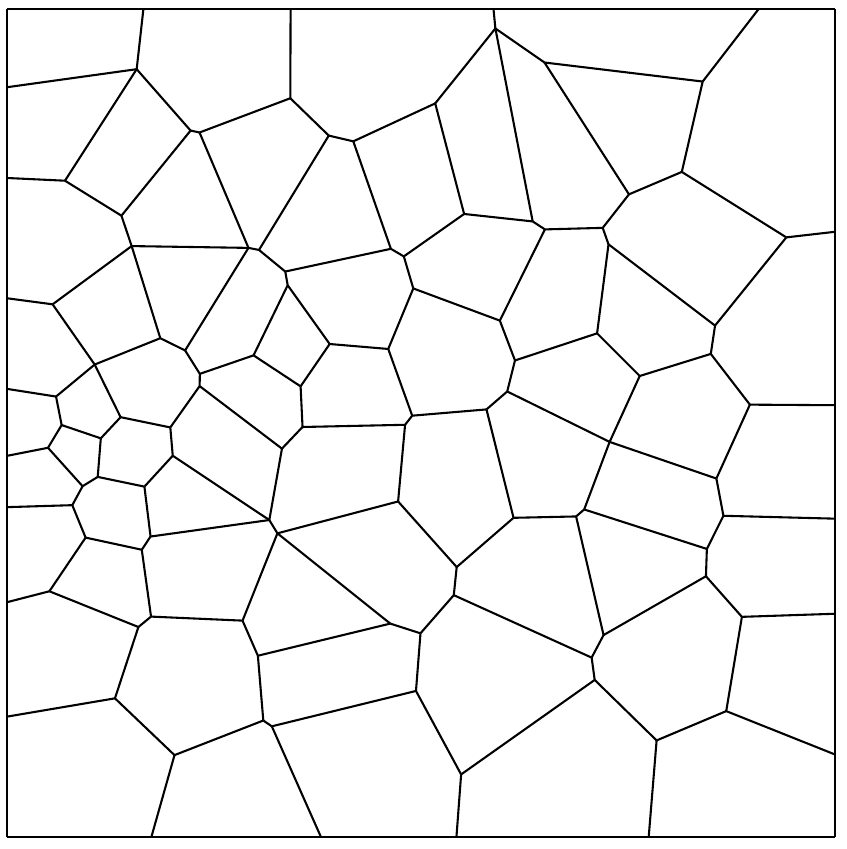}
\put(30,-15){{{\texttt{RANDOM}}}}
\end{overpic}
\qquad
\begin{overpic}[scale=0.20]{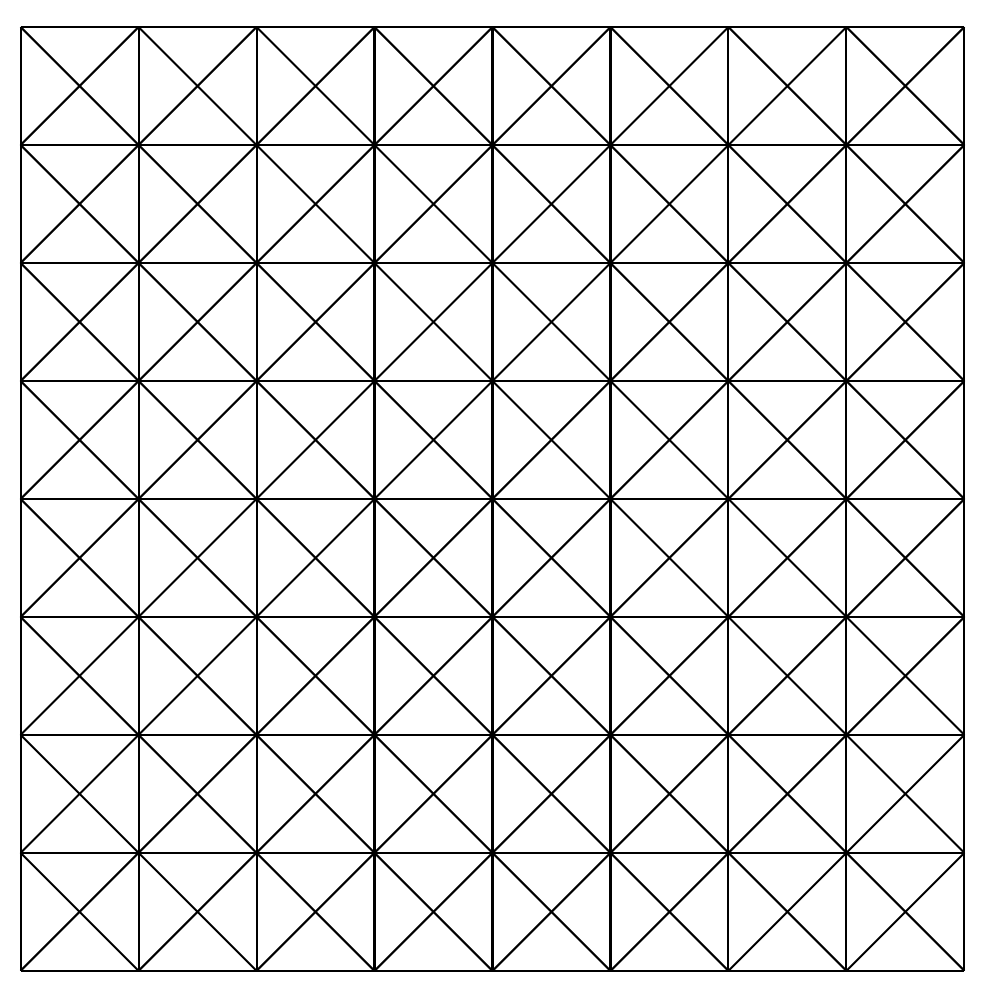}
\put(20,-15){{{\texttt{TRIANGULAR}}}}
\end{overpic}
\vspace{0.5cm}
\caption{Example of the adopted polygonal meshes.}
\label{fig:meshes}
\end{figure}

\paragraph{\textbf{Test 1. Performance w.r.t. $\texttt{r}$ and $\delta$.}}
In this test we examine the performance and the convergence properties of the proposed scheme \eqref{eq:stokes.vem} in the light of Proposition \ref{prop:main} and Proposition \ref{prop:main:2}.

\noindent
The aim of this test is to check the actual performance of the virtual element method assuming a Carreau--Yasuda model for different values of $r$ and $\delta$ in \eqref{eq:Carreau}. For simplicity we take $\mu=1$ and $\alpha=2$.
We consider Problem \eqref{eq:stokes.continuous} with full Dirichlet boundary conditions (i.e. $\Gamma_D = \partial \Omega$) on the unit square $\Omega = (0, 1)^2$. The load terms $\b f$ (depending on $r$ and $\delta$ in \eqref{eq:Carreau})  and the Dirichlet boundary conditions are chosen according to the analytical solution
\[
\b u_{\rm ex}(x_1,x_2) = 
\begin{bmatrix}
\sin\bigl(\frac{\pi}{2}x_1\bigr)\cos\bigl(\frac{\pi}{2}x_2\bigr)
\\
-\cos\bigl(\frac{\pi}{2}x_1\bigr)\sin\bigl(\frac{\pi}{2}x_2\bigr)
\end{bmatrix} \,,
\qquad
p_{\rm ex}(x_1,x_2) =  -\sin\biggl(\frac{\pi}{2}x_1\biggr)
\sin\biggl(\frac{\pi}{2}x_2\biggr) + \frac{4}{\pi^2} \,.
\]
The domain $\Omega$ is partitioned with the following sequences of polygonal meshes:
\texttt{QUADRILATERAL} distorted meshes and
\texttt{RANDOM} Voronoi meshes (see Fig. \ref{fig:meshes}).
For the generation of the Voronoi meshes we used the code \texttt{Polymesher} \cite{polymesher}.
For each family of meshes, we take the sequence with diameter $\texttt{h} =\texttt{1/4}$, $\texttt{1/8}$, $\texttt{1/16}$, $\texttt{1/32}$, $\texttt{1/64}$.
To highlight the behavior of the method for different situations we consider the following cases
\[
\texttt{r}=
\texttt{1.10}, \,
\texttt{1.15}, \,
\texttt{1.25}, \,
\texttt{1.50}, \,
\texttt{1.75}, \,
\texttt{2.00}, 
\qquad 
\delta=
\texttt{1}, \,
\texttt{0} \,.
\]

In Fig. \ref{tab:Q} and Fig. \ref{tab:V} we display the errors 
$\texttt{err}(\b u_h, W^{1,r})$ (first row),
$\texttt{err}(p_h, L^{r'})$  (second row), 
$\texttt{err}(\b \sigma, L^{r'})$  (third row),
 introduced in Subsection \ref{sub:error} for the sequences of aforementioned meshes. 
We notice that the theoretical predictions are confirmed:
for $\delta=\texttt{1}$ we recover the optimal rate of convergence, i.e. order \texttt{2}, whereas for $\delta=\texttt{0}$, we report on the plot the averaged order of convergence that is in agreement with the theory.

Furthermore, we display the number of iterations for the fixed-point procedure described in Subsection \ref{sub:fixpoint} where   $\texttt{N\_1|N\_2}$ denotes $\texttt{N\_1}$ iterations for \texttt{STEP 1} and $\texttt{N\_2}$ iterations for \texttt{STEP 2} are needed.
We observe, as expected, that smaller values of $\texttt{r}$ correspond to larger numbers of iterations, especially for $\delta=\texttt{0}$.

\begin{table}[!ht]
\centering
\begin{tabular}{cc}
\multicolumn{2}{c}
{\texttt{QUADRILATERAL MESHES}}
\\
 $\delta=\texttt{1}$
&$\delta=\texttt{0}$
\\
 \includegraphics[scale=0.2]{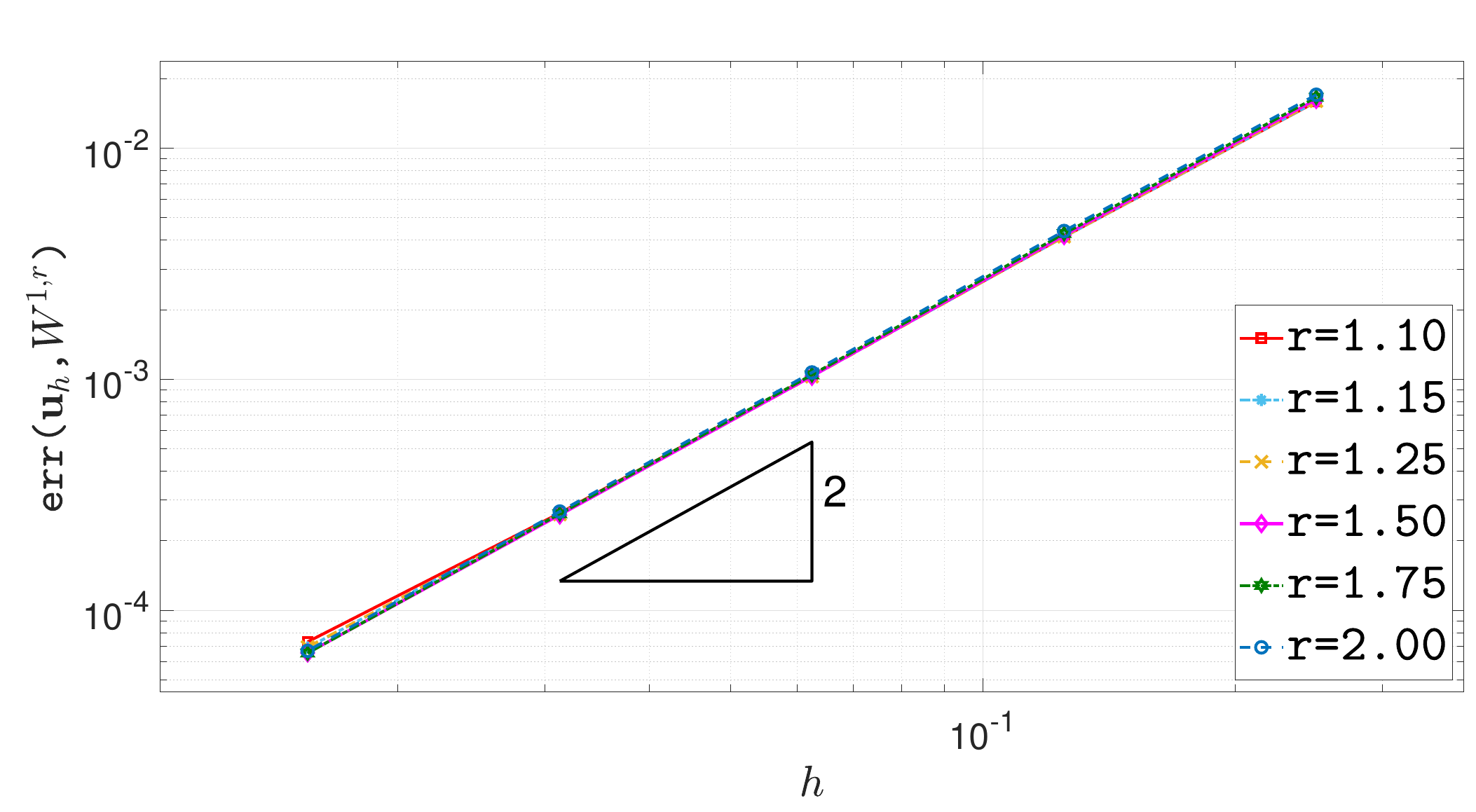}  
&\includegraphics[scale=0.2]{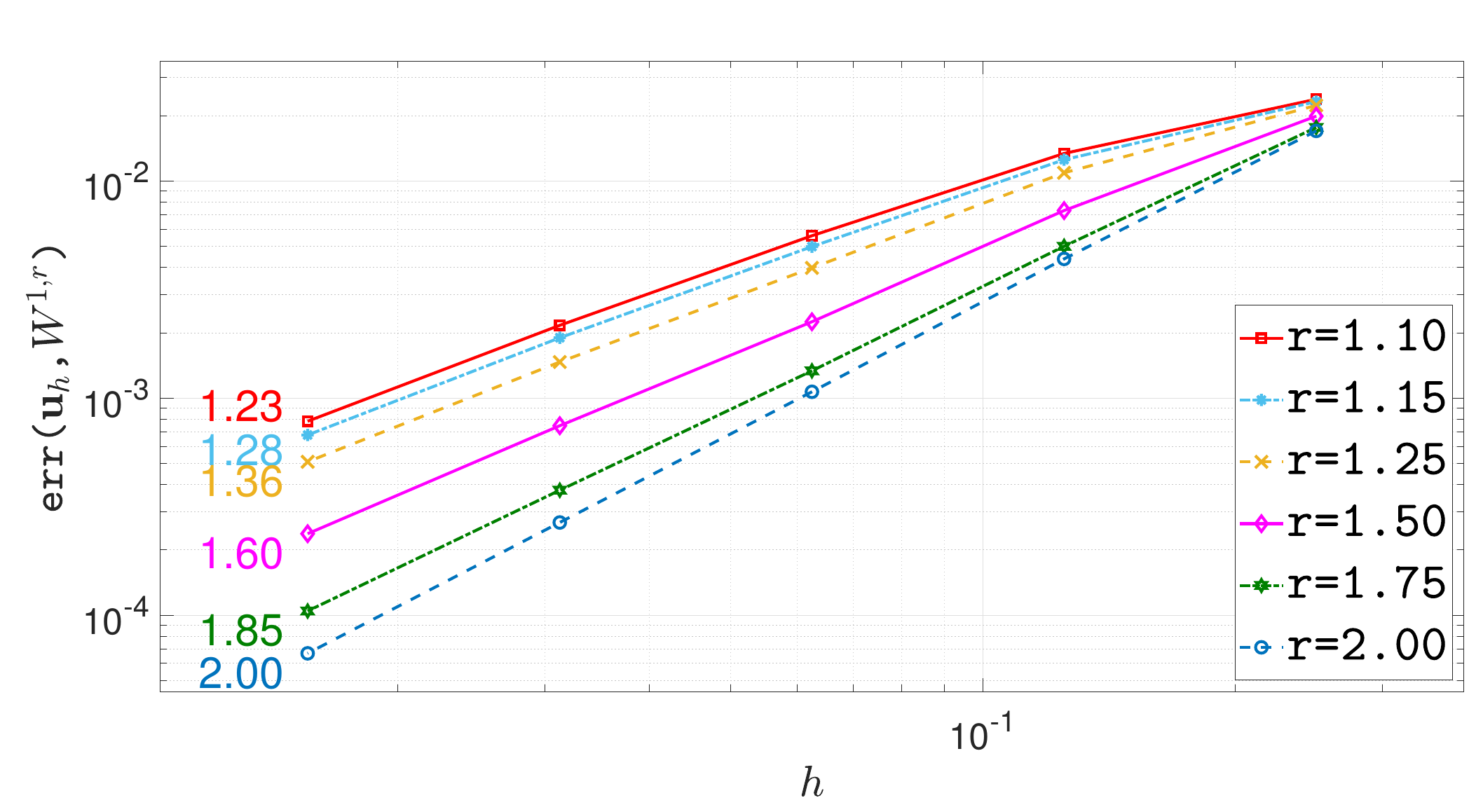}                        
\\
 \includegraphics[scale=0.2]{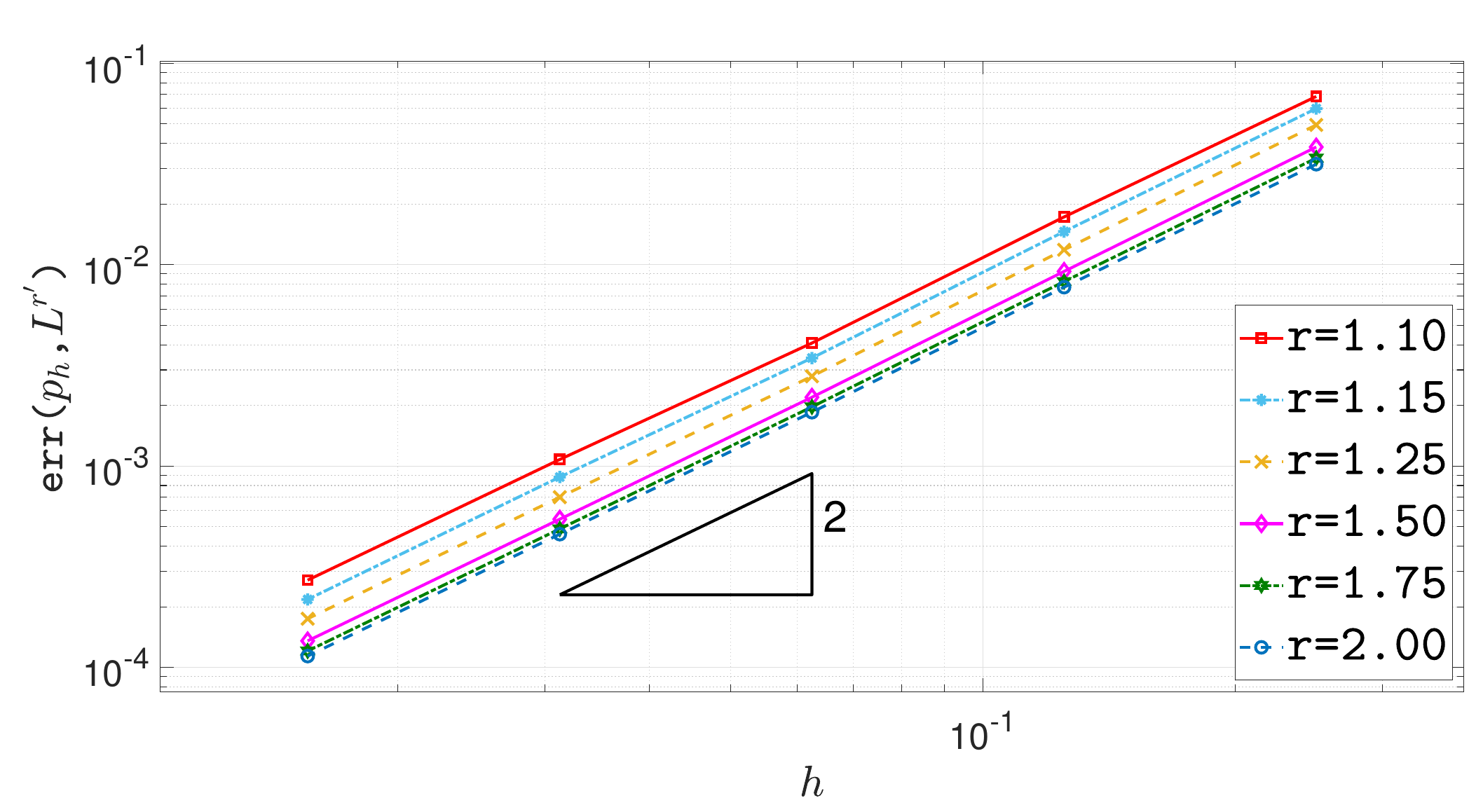}  
&\includegraphics[scale=0.2]{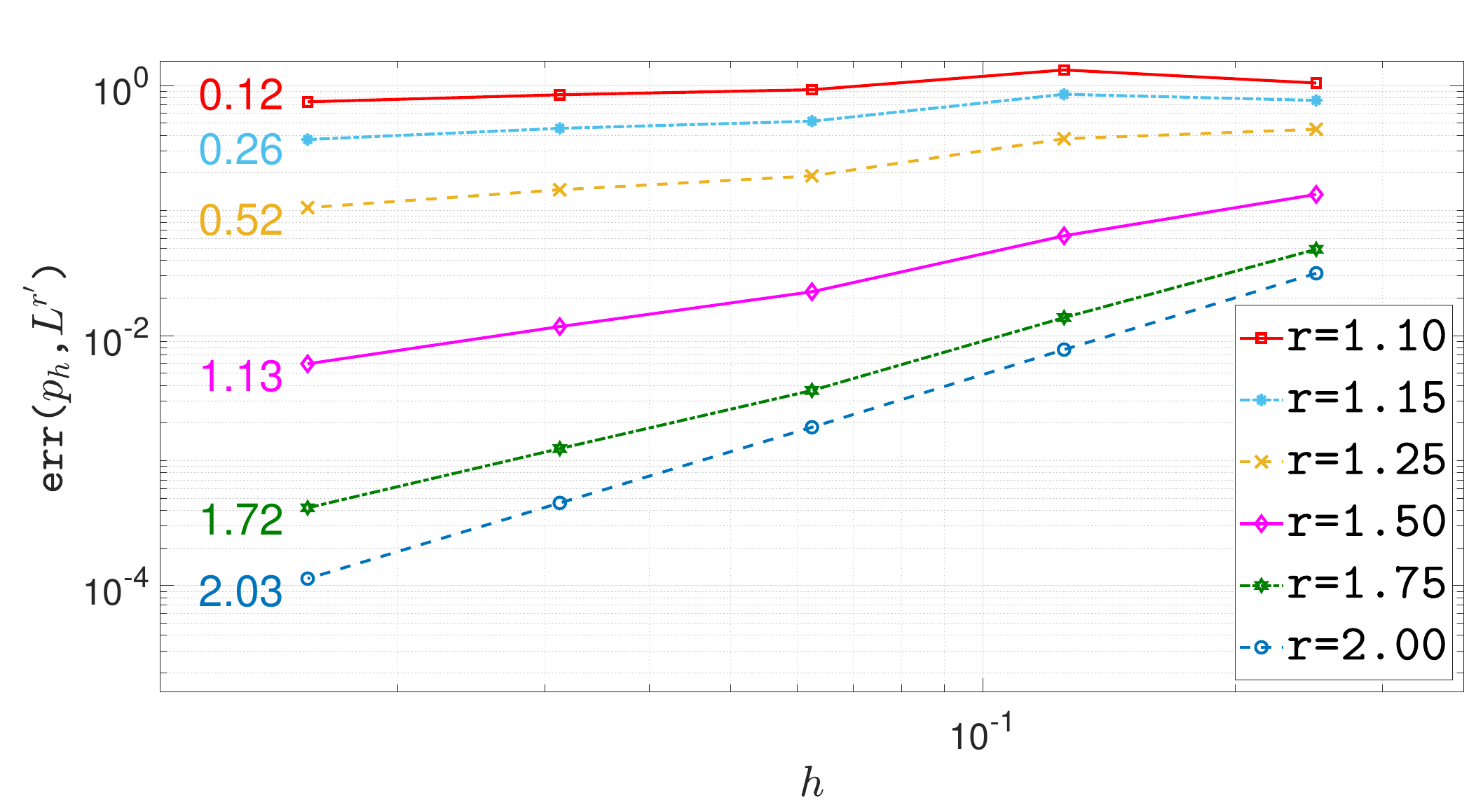}                        
\\
 \includegraphics[scale=0.2]{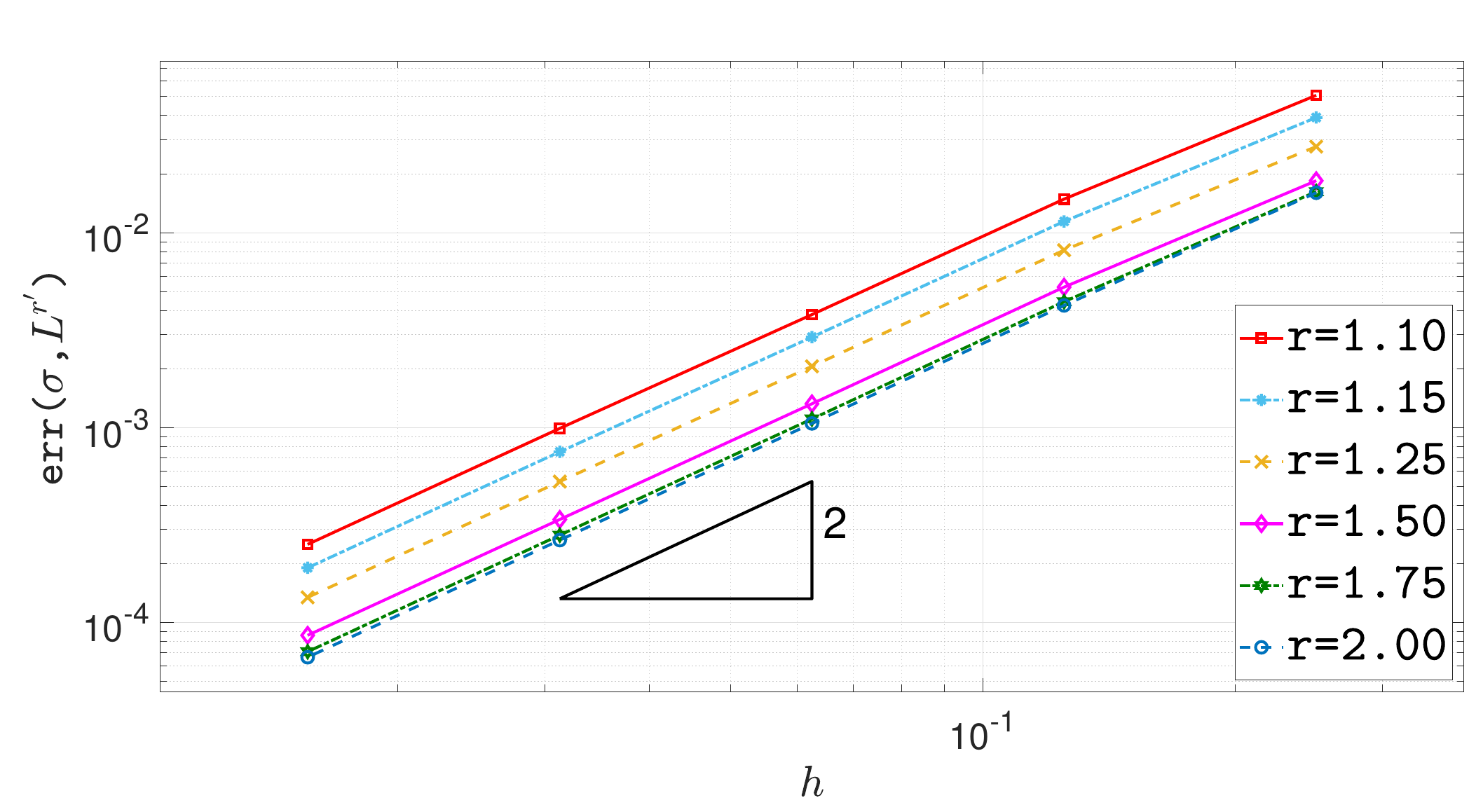}  
&\includegraphics[scale=0.2]{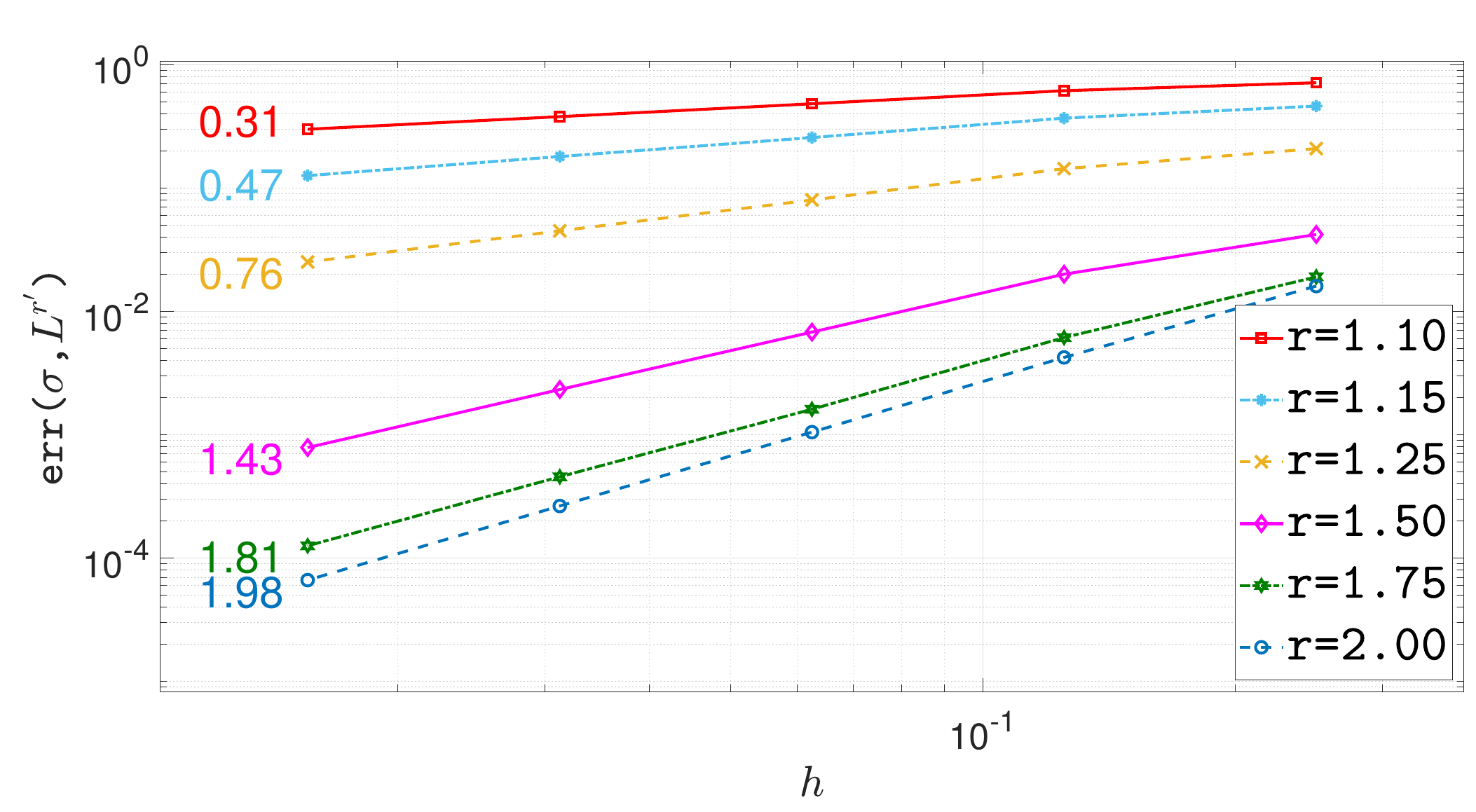}                        
%\\
% \includegraphics[scale=0.2]{QF_1e0.pdf}  
%&\includegraphics[scale=0.2]{QF_0.pdf}                        
\\
\centering
\begin{small}
\begin{tabular}{l|ccccc}
\toprule
%void
& \multicolumn{5}{c}{\texttt{r}}
\\
%\midrule
% void
   $\texttt{1/h}$
&  $\texttt{1.10}$      
&  $\texttt{1.15}$ 
&  $\texttt{1.25}$   
&  $\texttt{1.50}$        
&  $\texttt{1.75}$   
\\
\midrule          
  \texttt{4}
& \texttt{5|6}
& \texttt{5|6}
& \texttt{4|5}
& \texttt{4|5}
& \texttt{3|4}
\\
  \texttt{8}
& \texttt{5|6}
& \texttt{4|6}
& \texttt{4|5}
& \texttt{4|5}
& \texttt{3|4}
\\
  \texttt{16}
& \texttt{5|6}
& \texttt{4|6}
& \texttt{4|5}
& \texttt{4|5}
& \texttt{3|4}
\\
  \texttt{32}
& \texttt{4|6}
& \texttt{4|6}
& \texttt{4|5}
& \texttt{4|5}
& \texttt{3|4}
\\
  \texttt{64}
& \texttt{4|6}
& \texttt{4|6}
& \texttt{4|5}
& \texttt{4|5}
& \texttt{3|4}
\\
\bottomrule
\end{tabular}
\end{small}
&
\centering
\begin{small}
\begin{tabular}{l|ccccc}
\toprule
& \multicolumn{5}{c}{\texttt{r}}
\\
%\midrule
   $\texttt{1/h}$
&  $\texttt{1.10}$      
&  $\texttt{1.15}$ 
&  $\texttt{1.25}$   
&  $\texttt{1.50}$        
&  $\texttt{1.75}$   
\\
\midrule          
  \texttt{4}
& \texttt{7|29}
& \texttt{7|24}
& \texttt{7|16}
& \texttt{5|7}
& \texttt{4|5}
\\
  \texttt{8}
& \texttt{7|31}
& \texttt{7|23}
& \texttt{6|14}
& \texttt{5|7}
& \texttt{3|4}
\\
  \texttt{16}
& \texttt{7|28}
& \texttt{7|20}
& \texttt{6|12}
& \texttt{5|6}
& \texttt{3|4}
\\
  \texttt{32}
& \texttt{7|24}
& \texttt{7|17}
& \texttt{6|10}
& \texttt{4|6}
& \texttt{3|4}
\\
  \texttt{64}
& \texttt{7|19}
& \texttt{7|14}
& \texttt{6|09}
& \texttt{4|6}
& \texttt{3|4}
\\
\bottomrule
\end{tabular}
\end{small}
\end{tabular}

\caption{Test 1. Convergence histories of the VEM errors (cf. beginning of Subsection \ref{sub:error}) and the number of iterations of the fixed-point procedure for the mesh family \texttt{QUADRILATERAL}.}
\label{tab:Q}
\end{table}

\begin{table}[!ht]
\centering
\begin{tabular}{cc}
\multicolumn{2}{c}
{\texttt{RANDOM MESHES}}
\\
 $\delta=\texttt{1}$
&$\delta=\texttt{0}$
\\
 \includegraphics[scale=0.2]{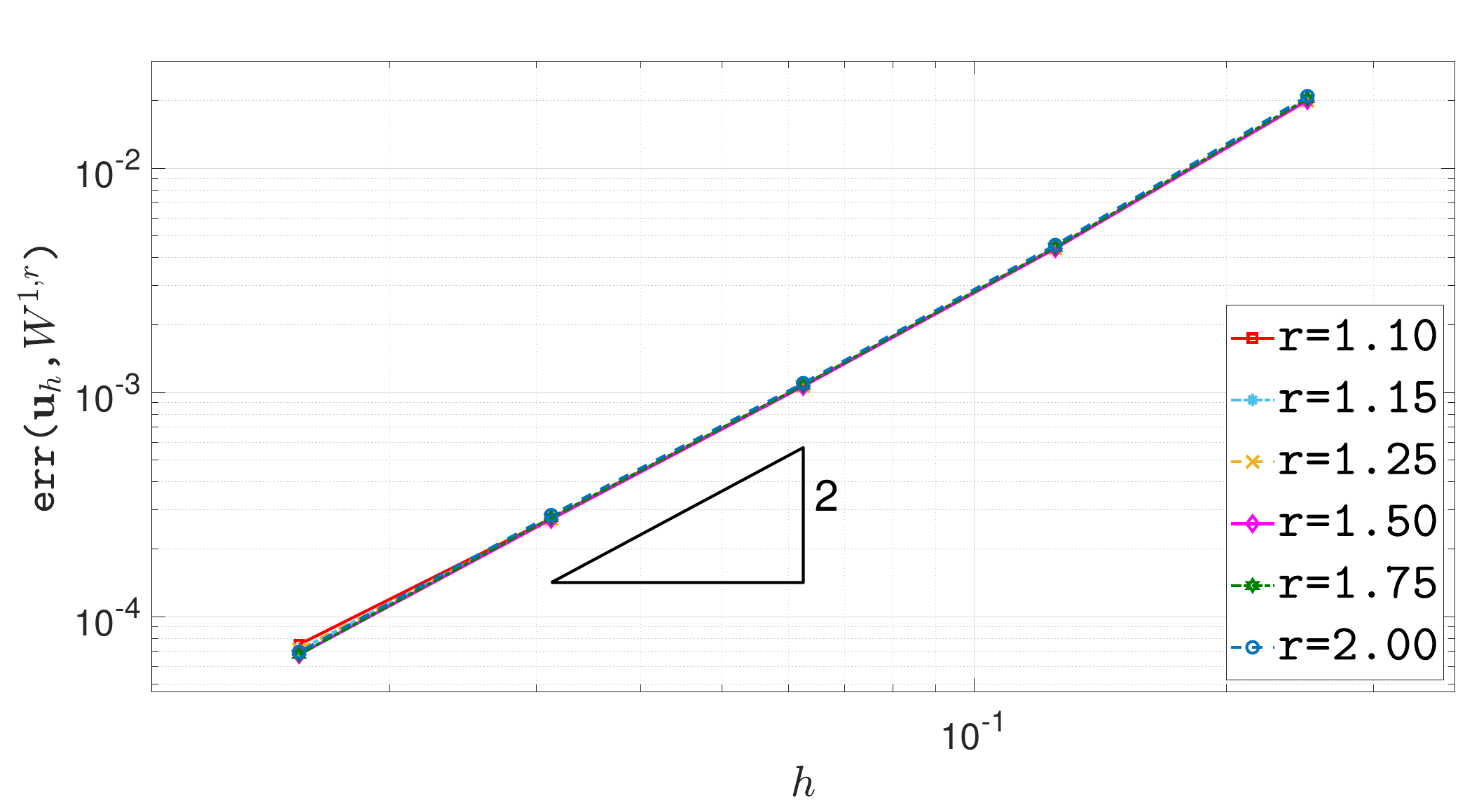}  
&\includegraphics[scale=0.2]{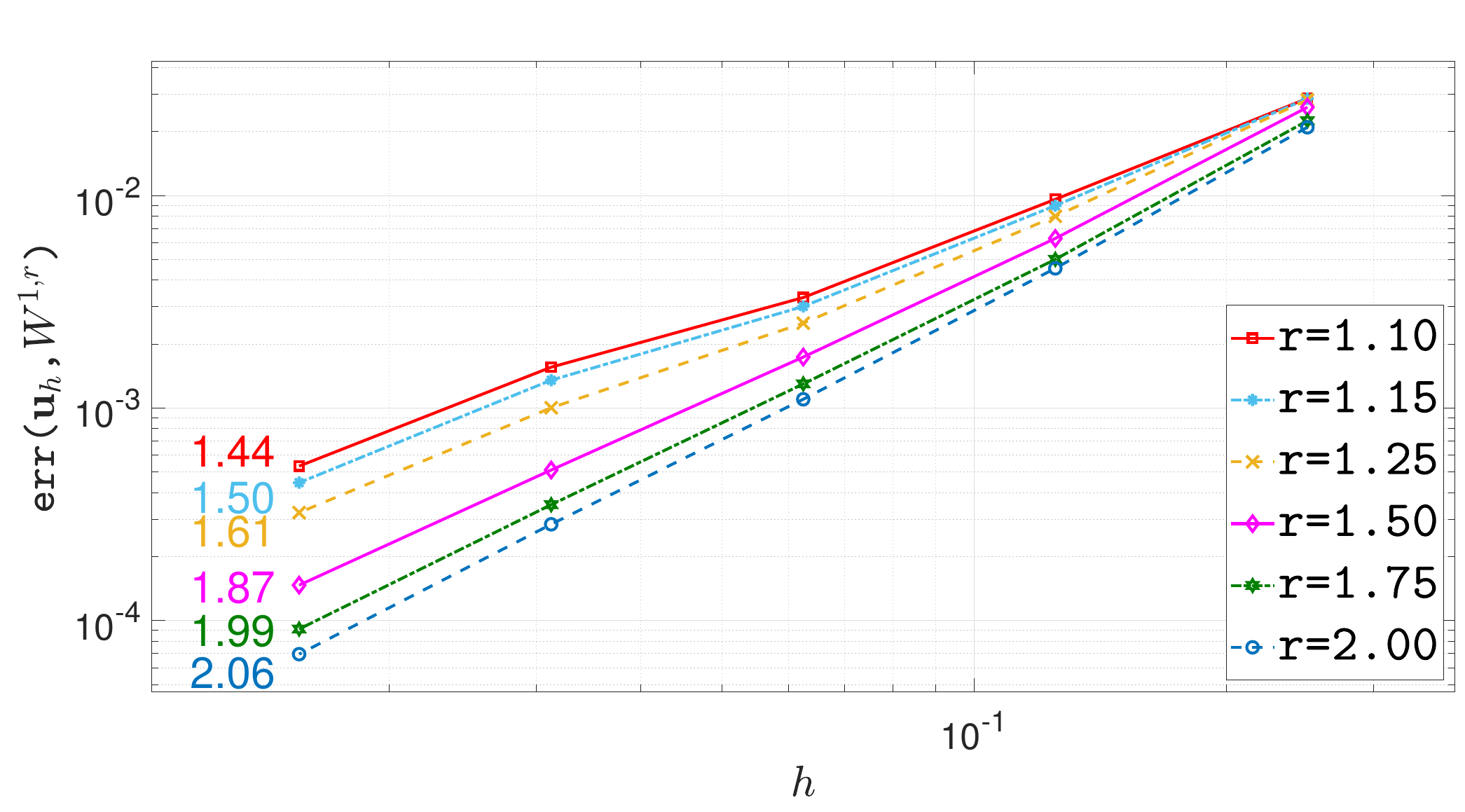}                        
\\
 \includegraphics[scale=0.2]{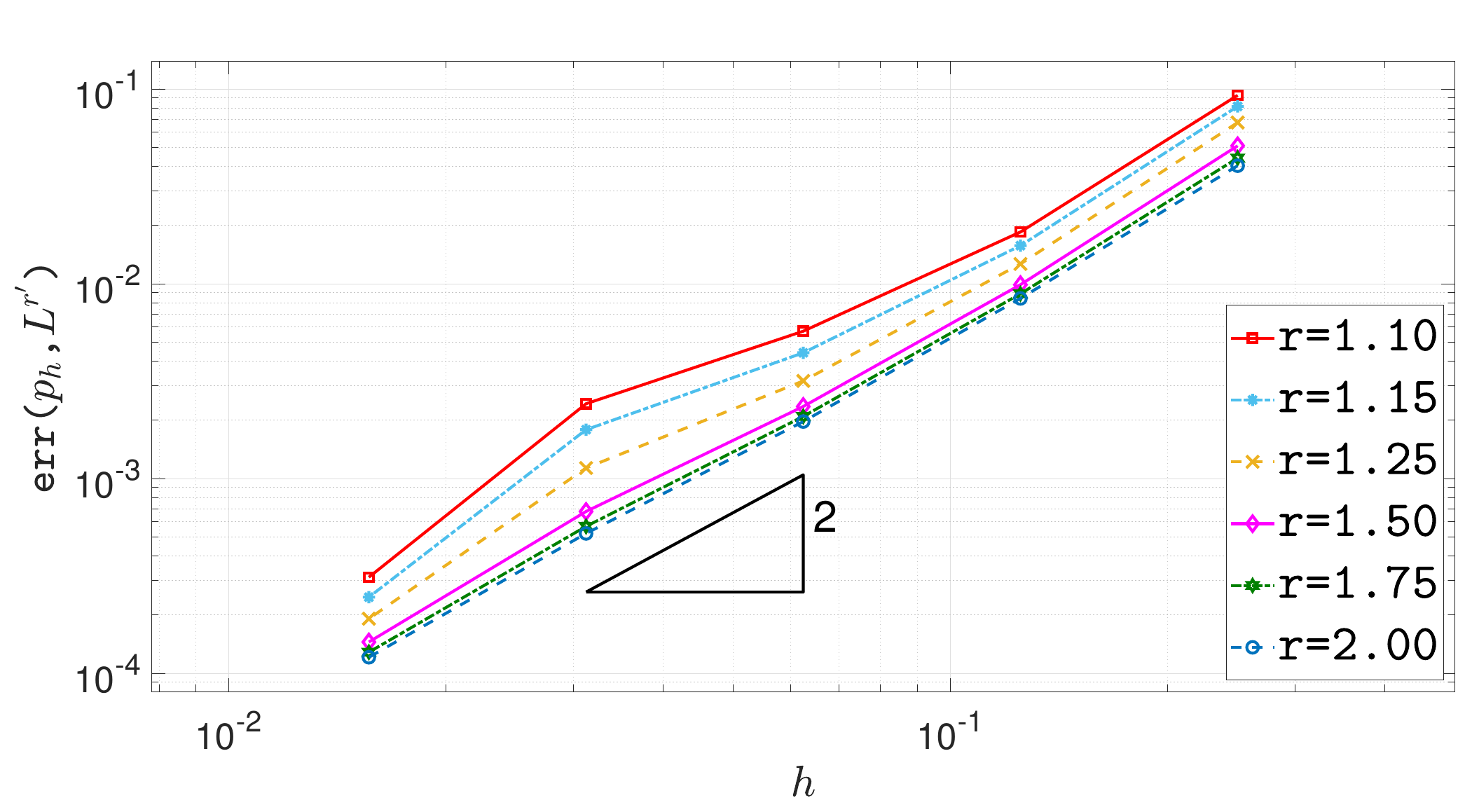}  
&\includegraphics[scale=0.2]{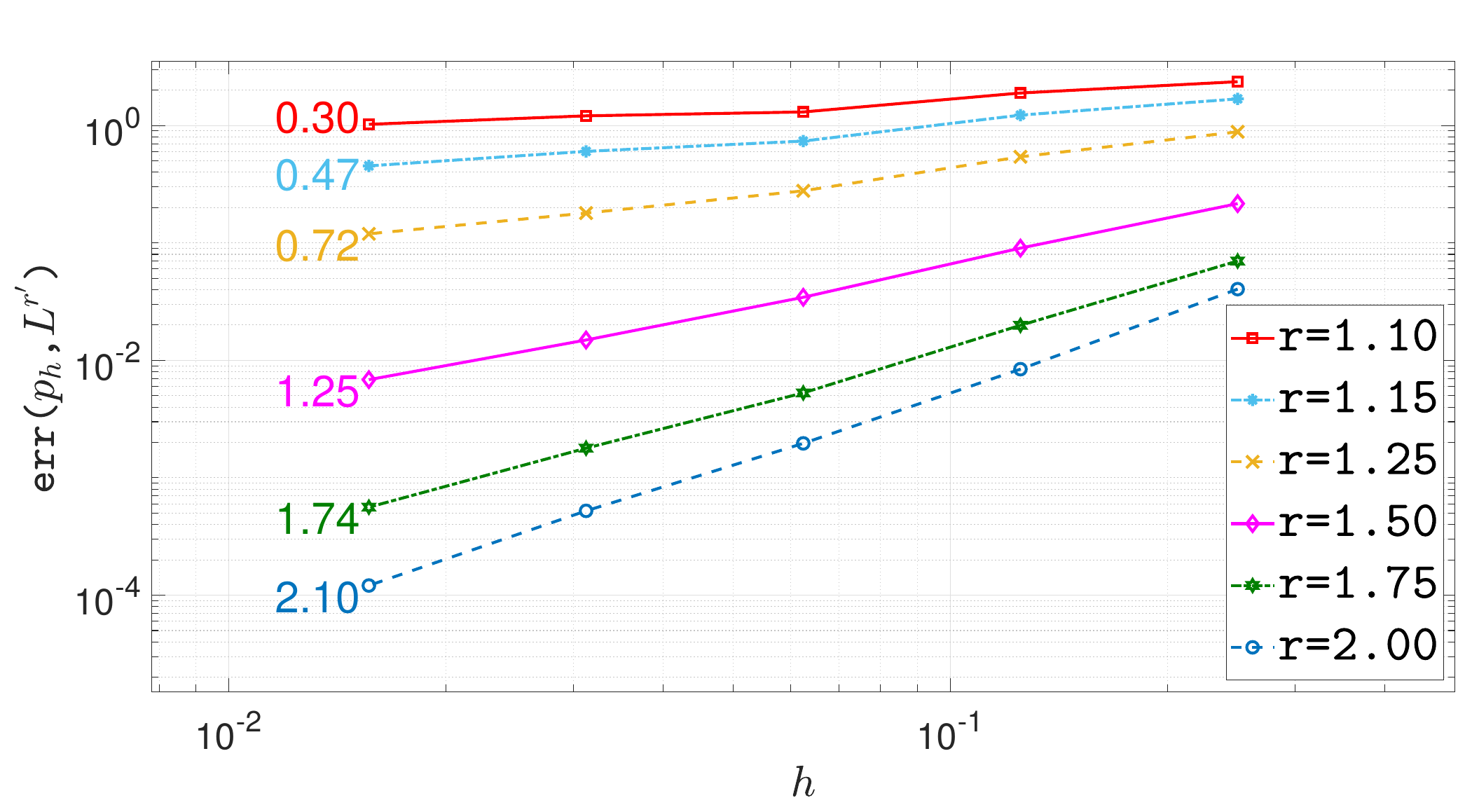}                        
\\
 \includegraphics[scale=0.2]{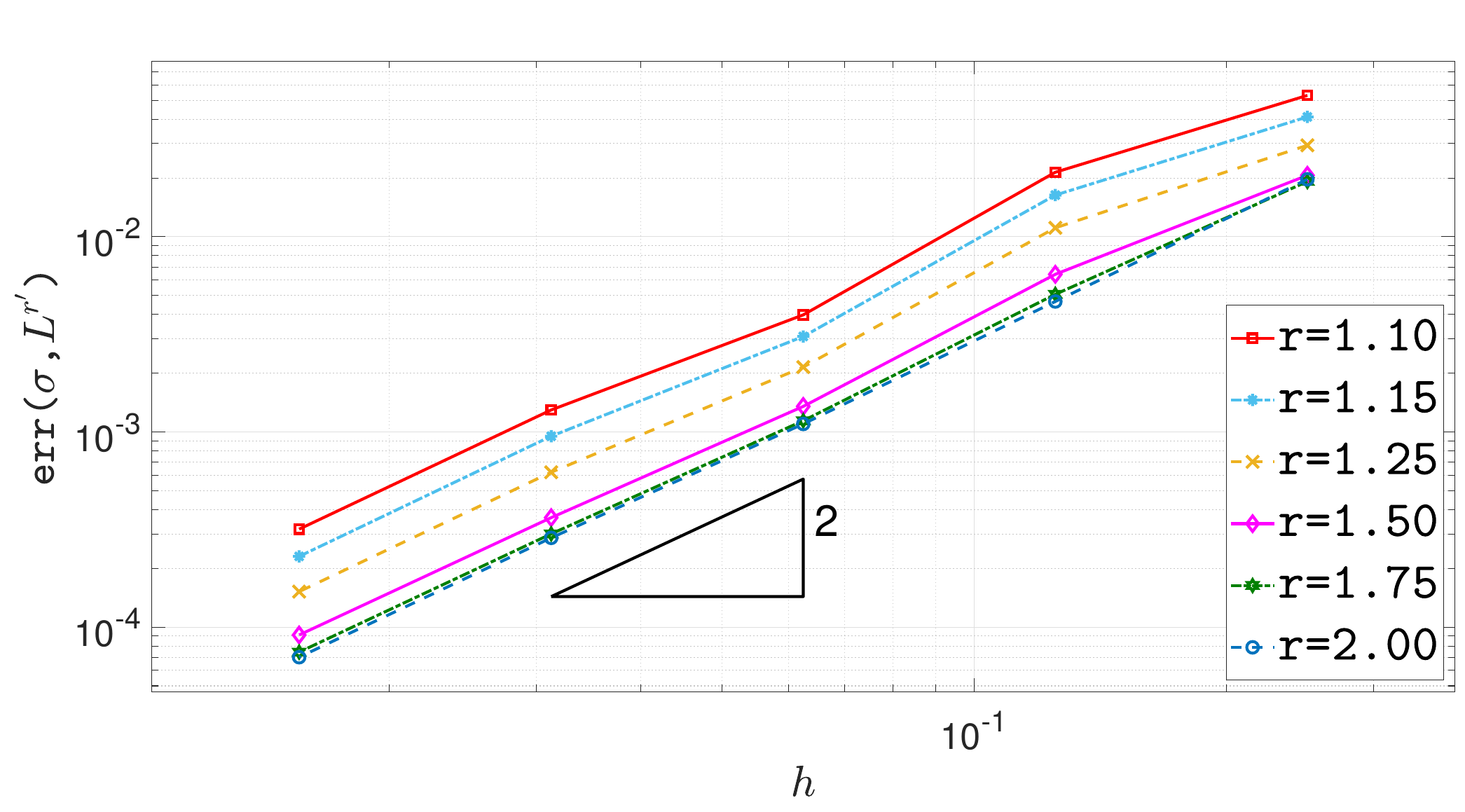}  
&\includegraphics[scale=0.2]{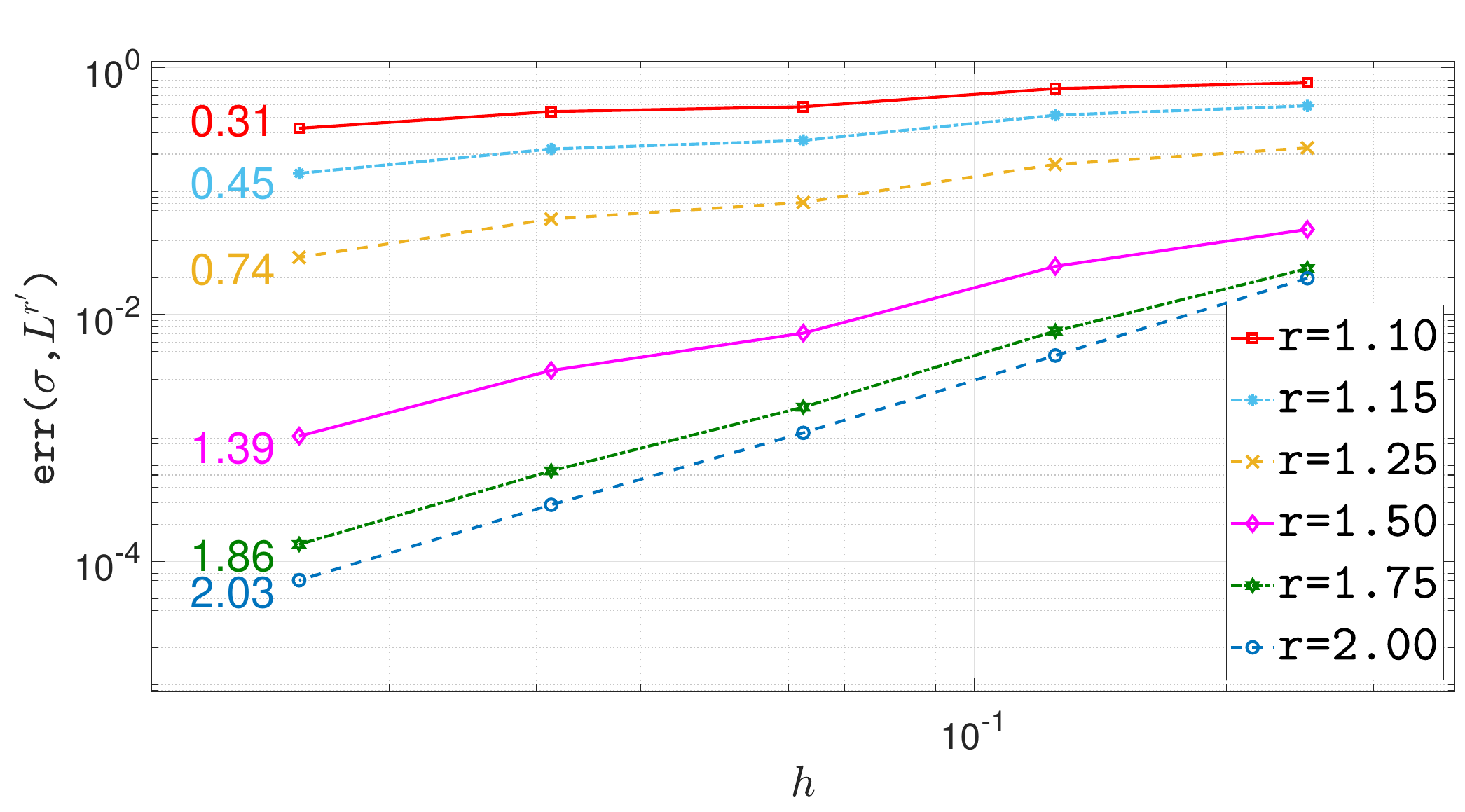}                        
%\\
% \includegraphics[scale=0.2]{VF_1e0.pdf}  
%&\includegraphics[scale=0.2]{VF_0.pdf}                        
\\
\centering
\begin{small}
\begin{tabular}{l|ccccc}
\toprule
&
\multicolumn{5}{c}{\texttt{r}}
\\
   \texttt{1/h}
&  $\texttt{1.10}$      
&  $\texttt{1.15}$ 
&  $\texttt{1.25}$   
&  $\texttt{1.50}$        
&  $\texttt{1.75}$   
\\
\midrule          
  \texttt{4}
& \texttt{5|6}
& \texttt{5|6}
& \texttt{4|6}
& \texttt{4|5}
& \texttt{3|4}
\\
  \texttt{8}
& \texttt{5|6}
& \texttt{4|6}
& \texttt{4|6}
& \texttt{4|5}
& \texttt{3|4}
\\
  \texttt{16}
& \texttt{5|6}
& \texttt{4|6}
& \texttt{4|5}
& \texttt{4|5}
& \texttt{3|4}
\\
  \texttt{32}
& \texttt{4|6}
& \texttt{4|6}
& \texttt{4|5}
& \texttt{3|4}
& \texttt{3|4}
\\
  \texttt{64}
& \texttt{4|6}
& \texttt{4|6}
& \texttt{4|5}
& \texttt{3|4}
& \texttt{3|4}
\\
\bottomrule
\end{tabular}
\end{small}
&
\centering
\begin{small}
\begin{tabular}{l|ccccc}
\toprule
&
\multicolumn{5}{c}{\texttt{r}}
\\
   \texttt{1/h}
&  $\texttt{1.10}$      
&  $\texttt{1.15}$ 
&  $\texttt{1.25}$   
&  $\texttt{1.50}$        
&  $\texttt{1.75}$   
\\
\midrule          
  \texttt{4}
& \texttt{7|22}
& \texttt{7|17}
& \texttt{7|12}
& \texttt{5|7}
& \texttt{4|5}
\\
  \texttt{8}
& \texttt{7|25}
& \texttt{7|18}
& \texttt{6|12}
& \texttt{5|6}
& \texttt{4|4}
\\
  \texttt{16}
& \texttt{7|20}
& \texttt{7|15}
& \texttt{6|10}
& \texttt{5|6}
& \texttt{4|4}
\\
  \texttt{32}
& \texttt{7|17}
& \texttt{7|13}
& \texttt{6|09}
& \texttt{5|6}
& \texttt{4|4}
\\
  \texttt{64}
& \texttt{7|14}
& \texttt{7|11}
& \texttt{6|07}
& \texttt{5|6}
& \texttt{4|4}
\\
\bottomrule
\end{tabular}
\end{small}
\end{tabular}

\caption{Test 1. Convergence histories of the VEM errors (cf. beginning of Subsection \ref{sub:error}) and number of iterations of the fixed-point procedure  for the mesh family \texttt{RANDOM}.}
\label{tab:V}
\end{table}

%%%%%%%%%%%%%%%%%%%%%%%%%%%%%%%%%%%%%%%%%%%%%%%%%%%%%%%%%%%%%%%%%%%%%%
%%%%%%%%%%%%%%%%%%%%%%%%%%%%%%%%%%%%%%%%%%%%%%%%%%%%%%%%%%%%%%%%%%%%%%
%%%%%%%%%%%%%%%%%%%%%%%%%%%%%%%%%%%%%%%%%%%%%%%%%%%%%%%%%%%%%%%%%%%%%%

\paragraph{\textbf{Test 2. Singular solution.}}

The scope of the present test is to show the performance of the method  in case of low regular solutions. 
We investigate the behavior of the proposed method for the test proposed in \cite[Section 7]{Belenki.Berselli.ea:12}.

We consider Problem \eqref{eq:stokes.continuous} on the square domain $\Omega = (-1, 1)^2$ where $\b f$ and the Dirichlet boundary conditions (posed on the whole $\partial \Omega$) are chosen in accordance with the analytical solution
\[
\b u_{\rm ex}(x_1,x_2) = 
\vert \b x \vert^{0.01}
\begin{bmatrix}
x_2
\\
-x_1
\end{bmatrix} \,,
\qquad
p_{\rm ex}(x_1,x_2) =  -\vert \b x \vert^{\gamma} + c_\gamma \,.
\]
where $\gamma=\frac{2}{r} - 1 + 0.01$ and $c_\gamma$ is s.t. $p_{\rm ex}$ is zero averaged.
Notice that for all $r \in (1, 2]$
\[
\b u_{\rm ex} \in \b W^{2/r+1, r}(\Omega) \,,
\quad
\b{\sigma}(\cdot, \b\epsilon(\b u_{\rm ex})) \in \mathbb{W}^{2/r',r'}(\Omega) \,,
\quad
\b f \in \b {W}^{2/r' -1,r'}(\Omega)\,,
\quad
p_{\rm ex} \in W^{1, r'}(\Omega) \,,
\]
therefore, with the notation of Proposition \ref{prop:main} and Proposition \ref{prop:main:2}: 
\[
k_1= \frac{2}{r} \,,
\quad
k_2= \frac{2}{r'} \,,
\quad
k_3= \frac{2}{r'} -2 \,,
\quad
k_4= 1 \,.
\]
We analyse the Carreau--Yasuda model \eqref{eq:Carreau} with $\mu=1$, $\alpha = 1$ and
\[
\texttt{r} =
\texttt{1.25}, \,
\texttt{1.33}, \,
\texttt{1.50}, \,
\texttt{1.67}, \,
\texttt{1.75}, \,
\texttt{2.00}, \,.
\]
We consider $\delta=\texttt{0}$ (differently from \cite[Section 7]{Belenki.Berselli.ea:12} where  $\delta=\texttt{1e-04}$).
The domain $\Omega$ is partitioned with a sequence of \texttt{TRIANGULAR} meshes with diameter $\texttt{h}$ =$\texttt{1/2}$, $\texttt{1/4}$, $\texttt{1/8}$, $\texttt{1/16}$ (see Fig.\ref{fig:meshes}).
In Tab. \ref{tab:TV} we show the errors $\texttt{err}(\b u_h, W^{1,r})$ for the proposed values of $\texttt{r}$. Notice that, accordingly with Proposition \ref{prop:main} the expected rate of convergence is $2/r'$, nevertheless we numerically observe linear convergence.
This means that, at least in the pre-asymptotic regime, the error velocity errors is dominated by the term $h^{k_1 r/2} = h$ (cf. equation \eqref{falcao}). 
In Tab. \ref{tab:TP} and Tab. \ref{tab:TS} we exhibit the errors
$\texttt{err}(p_h, L^{r'})$,
$\texttt{err}(\b \sigma, L^{r'})$, respectively.
Notice that the pressure errors are in accordance with Proposition \ref{prop:main:2} and the previous numerical evidence, that is
$\Vert p - p_h\Vert_{L^{r'}(\Omega)} \lesssim  \Vert \b u - \b u_h\Vert_{W^{1,r}(\Omega_h)}^{2/r'} + h^{2/r'} + h \lesssim h^{2/r'}$.

\begin{table}[!ht]
\centering
\begin{small}
\begin{tabular}{l|cccccc}
\toprule
  % void
& \multicolumn{6}{c}{\texttt{r}}
\\
%\midrule
   \texttt{1/h}
&  $\texttt{1.25}$      
&  $\texttt{1.33}$ 
&  $\texttt{1.50}$   
&  $\texttt{1.67}$        
&  $\texttt{1.75}$   
&  $\texttt{2.00}$       
\\
\midrule          
  \texttt{2}
& \texttt{7.5101e-04}
& \texttt{7.7297e-04}
& \texttt{8.2290e-04}
& \texttt{8.7303e-04}
& \texttt{8.9993e-04}
& \texttt{9.9850e-04}
\\
  \texttt{4}
& \texttt{2.8701e-04}
& \texttt{3.0353e-04}
& \texttt{3.4469e-04}
& \texttt{3.8999e-04}
& \texttt{4.1395e-04}
& \texttt{4.9925e-04}
\\
  \texttt{8}
& \texttt{1.1626e-04}
& \texttt{1.2176e-04}
& \texttt{1.4217e-04}
& \texttt{1.7106e-04}
& \texttt{1.8778e-04}
& \texttt{2.4837e-04}
\\
  \texttt{16}
& \texttt{7.4789e-05}
& \texttt{5.9046e-05}
& \texttt{6.1538e-05}
& \texttt{7.5020e-05}
& \texttt{8.4862e-05}
& \texttt{1.2339e-04}
\\
\midrule     
{\texttt{a.c.r.}}
& \texttt{1.1093e+00}
& \texttt{1.2368e+00}
& \texttt{1.2470e+00}
& \texttt{1.1802e+00}
& \texttt{1.1355e+00}
& \texttt{1.0055e+00}
\\
\midrule     
\texttt{$\frac{2}{r'}$}
& \texttt{0.40}
& \texttt{0.50}
& \texttt{0.66}
& \texttt{0.80}
& \texttt{0.86}
& \texttt{1.00}
\\
\bottomrule
\end{tabular}
\end{small}
\caption{Test 2. Errors $\texttt{err}(\b u_h, W^{1,r})$.}
\label{tab:TV}
\end{table}

\begin{table}[!ht]
\centering
\begin{small}
\begin{tabular}{l|cccccc}
\toprule
%void
& \multicolumn{6}{c}{\texttt{r}}
\\
%\midrule
% void
   \texttt{1/h}
&  $\texttt{1.25}$      
&  $\texttt{1.33}$ 
&  $\texttt{1.50}$   
&  $\texttt{1.67}$        
&  $\texttt{1.75}$   
&  $\texttt{2.00}$       
\\
\midrule          
  \texttt{2}
& \texttt{1.2116e-01}
& \texttt{1.0603e-01}
& \texttt{9.6989e-02}
& \texttt{9.9739e-02}
& \texttt{1.0303e-01}
& \texttt{1.7563e-01}
\\
  \texttt{4}
& \texttt{8.7515e-02}
& \texttt{6.7873e-02}
& \texttt{5.2335e-02}
& \texttt{5.0615e-02}
& \texttt{5.1764e-02}
& \texttt{8.7824e-02}
\\
  \texttt{8}
& \texttt{6.4955e-02}
& \texttt{4.5430e-02}
& \texttt{2.9065e-02}
& \texttt{2.5820e-02}
& \texttt{2.6045e-02}
& \texttt{4.3701e-02}
\\
  \texttt{16}
& \texttt{4.8749e-02}
& \texttt{3.1194e-02}
& \texttt{1.6625e-02}
& \texttt{1.3260e-02}
& \texttt{1.3135e-02}
& \texttt{2.1711e-02}
\\
\midrule     
  \texttt{a.c.r.}
& \texttt{4.3785e-01}
& \texttt{5.8840e-01}
& \texttt{8.4813e-01}
& \texttt{9.7035e-01}
& \texttt{9.9055e-01}
& \texttt{1.0053e+00}
\\
\midrule     
\texttt{$\frac{2}{r'}$}
& \texttt{0.40}
& \texttt{0.50}
& \texttt{0.66}
& \texttt{0.80}
& \texttt{0.86}
& \texttt{1.00}
\\
\bottomrule
\end{tabular}
\end{small}
\caption{Test 2. Errors $\texttt{err}(p_h, L^{r'})$.}
\label{tab:TP}
\end{table}

%%%%%%%%%%%%%%%%%%%%%%%%%%%%%%%%%%%%%%%%%%%%%%%%%%%
%%%%%%%%%%%%%%%%%%%%%%%%%%%%%%%%%%%%%%%%%%%%%%%%%%%

\begin{table}[!ht]
\centering
\begin{small}
\begin{tabular}{l|cccccc}
\toprule
%void

& \multicolumn{6}{c}{\texttt{r}}
\\
%\midrule
% void
   \texttt{1/h}
&  $\texttt{1.25}$      
&  $\texttt{1.33}$ 
&  $\texttt{1.50}$   
&  $\texttt{1.67}$        
&  $\texttt{1.75}$   
&  $\texttt{2.00}$       
\\
\midrule          
  \texttt{2}
& \texttt{1.4319e-01}
& \texttt{1.4386e-01}
& \texttt{1.4352e-01}
& \texttt{1.4369e-01}
& \texttt{1.4400e-01}
& \texttt{1.4592e-01}
\\
  \texttt{4}
& \texttt{1.0815e-01}
& \texttt{1.0136e-01}
& \texttt{9.0112e-02}
& \texttt{8.2299e-02}
& \texttt{7.9320e-02}
& \texttt{7.3008e-02}
\\
  \texttt{8}
& \texttt{8.1752e-02}
& \texttt{7.1451e-02}
& \texttt{5.6597e-02}
& \texttt{4.7136e-02}
& \texttt{4.3642e-02}
& \texttt{3.6328e-02}
\\
  \texttt{16}
& \texttt{6.1782e-02}
& \texttt{5.0421e-02}
& \texttt{3.5648e-02}
& \texttt{2.7099e-02}
& \texttt{2.4086e-02}
& \texttt{1.8048e-02}
\\
\midrule     
  \texttt{a.c.r.}
& \texttt{4.0425e-01}
& \texttt{5.0419e-01}
& \texttt{6.6980e-01}
& \texttt{8.0220e-01}
& \texttt{8.5995e-01}
& \texttt{1.0051e+00}
\\
\midrule     
\texttt{$\frac{2}{r'}$}
& \texttt{0.40}
& \texttt{0.50}
& \texttt{0.66}
& \texttt{0.80}
& \texttt{0.86}
& \texttt{1.00}
\\
\bottomrule
\end{tabular}
\end{small}
\caption{Test 2. Errors $\texttt{err}(\b \sigma, L^{r'})$.}
\label{tab:TS}
\end{table}

\section*{Acknowledgements}
%{PFA and LBdV have been partially funded by ERC Synergy Grant} n.\verb+101115663+ \emph{``NEMESIS: NEw generation MEthods for numerical SImulationS'' funded by European Union.}
This research has been partially funded by the European Union (ERC, NEMESIS, project number 101115663). Views and opinions expressed are however those of the author(s) only and do not necessarily reflect those of the European Union or the European Research Council Executive Agency. %Neither the European Union nor the granting authority can be held responsible for them.
PFA, LBdV and MV have been partially funded by 
% PRIN2017 n. \verb+201744KLJL+\emph{``Virtual Element Methods: Analysis and Applications''} and  
PRIN2020 n. \verb+20204LN5N5+\emph{``Advanced polyhedral discretisations of heterogeneous PDEs for multiphysics problems''} research grant, funded by the Italian Ministry of Universities and Research (MUR). 
GV has been partially funded by  
PRIN2022 n. \verb+2022MBY5JM+\emph{``FREYA - Fault REactivation: a hYbrid numerical Approach''} research grant, and
PRIN2022PNRR n. \verb+P2022M7JZW+\emph{``SAFER MESH - Sustainable mAnagement oF watEr Resources: ModEls and numerical MetHods''} research grant,
funded by the Italian Ministry of Universities and Research (MUR).
PFA is partially supported by ICSC -- Centro Nazionale di Ricerca in High Performance Computing, Big Data, and Quantum Computing funded by European Union -- NextGenerationEU. 
The present research is part of the activities of ``Dipartimento di Eccelllenza 2023-2027''.
The authors are members of INdAM-GNCS.

\nocitenames
\bibliographystyle{plain}
\bibliography{references}

%--------------------------------------------------------%
\appendix
\section{Appendix}
\label{sec:appendix}
We here show the proof of two technical lemmas.

\begin{lemma}\label{Lemma:A}
   Under the same assumptions of Proposition \ref{prop:main}, let $\b u_I \in \VDG$ be the interpolant of $\b u$ (cf. Lemma \ref{lem:approx-interp}) and $\b \xi_h :=\b u_h - \b u_I$, then the following holds
   \[
   \begin{aligned}
       \| (\b \sigma(\cdot, \b\Pi_{k-1}^0 \b \epsilon(\b u_h)) - \b\sigma(\cdot, \b \epsilon(\b u))\|_{\mathbb{L}^{r'}(\Omega)}^{r'}&\lesssim 
       h^{k_1 r} R_1^r  + \\
       & +
       \int_\Omega \left( \b \sigma(\cdot, \b\Pi_{k-1}^0 \b \epsilon(\b u_h)) - \b\sigma(\cdot, \b\Pi_{k-1}^0\b \epsilon(\b u_I))\right) : \b\Pi_{k-1}^0 \epsilon(\b \xi_h) \,.
    \end{aligned}
   \]
\end{lemma}
\begin{proof}
    Employing \eqref{eq:extrass.continuity} with $\delta=0$ we have 
    \begin{equation}
    \label{eq:lemmaA0}
    \begin{aligned}
       &\| (\b \sigma(\cdot, \b\Pi_{k-1}^0 \b \epsilon(\b u_h)) - \b\sigma(\cdot, \b \epsilon(\b u))\|_{\mathbb{L}^{r'}(\Omega)}^{r'}
        \lesssim
        \int_{\Omega_h} (\vert\b\Pi_{k-1}^0 \b \epsilon(\b u_h)\vert+ \vert\b \epsilon(\b u)\vert)^{(r-2)r'}\vert 
        \b\Pi_{k-1}^0 \b \epsilon(\b u_h)- \b \epsilon(\b u)
        \vert^{r'}
        \\
&=\int_{\Omega} \left( \left[(\vert\b\Pi_{k-1}^0 \b \epsilon(\b u_h)\vert+ \vert\b \epsilon(\b u)\vert)^{(r-2)(r'-1)}\vert         \b\Pi_{k-1}^0 \b \epsilon(\b u_h)- \b \epsilon(\b u) \vert^{r'-2}\right] \times \right .
\\
&\quad \quad  \quad \times \left.\left[(\vert\b\Pi_{k-1}^0 \b \epsilon(\b u_h)\vert+ \vert\b \epsilon(\b u)\vert)^{(r-2)}\vert 
\b\Pi_{k-1}^0 \b \epsilon(\b u_h)- \b \epsilon(\b u)        \vert^{2}\right]
\right) =: \int_\Omega T_1 \times T_2 \,,
\end{aligned}
\end{equation}
Following \cite{Barrett.Liu:94}, employing $r-2\le 0$ and $r' \ge 2$, together with $\vert X\vert + \vert Y\vert \ge \vert X-Y\vert$ for $X,Y\in \mathbb{R}^{d\times d}$ and noticing that  $(r-2)(r'-1)+r'-2=(r-1)r' -r=0$, we have
\begin{equation}
\label{eq:lemmaAT1}
T_1
\lesssim 
\vert \b\Pi_{k-1}^0 \b \epsilon(\b u_h)- \b \epsilon(\b u) \vert^{(r-2)(r'-1)+(r'-2)} = 1 \,.
\end{equation}
We now estimate the term $T_2$ in \eqref{eq:lemmaA0}.
In the following $C$ will denote a generic
positive constant independent of $h$ that may change at each occurrence, whereas the parameter $\varepsilon$ adopted in \eqref{eq:lemmaAT2A} and \eqref{eq:lemmaAT2B} will be specified later.
Using \eqref{eq:extrass.monotonicity} we have 
\begin{equation}
\label{eq:lemmaAT2}
\begin{aligned}
   T_2 &\le C
\bigl( \b\sigma(\cdot, \b\Pi_{k-1}^0 \b \epsilon(\b u_h))- \b\sigma(\cdot, \b \epsilon(\b u) ) \bigr) : \bigl( \Pi_{k-1}^0 \b \epsilon(\b u_h)- \b \epsilon(\b u) \bigr) 
\\
&= C
\bigl( \b\sigma(\cdot, \b\Pi_{k-1}^0 \b \epsilon(\b u_h))- \b\sigma(\cdot, \b \Pi^0_{k-1}\b \epsilon(\b u_I) ) \bigr) : \bigl( \Pi_{k-1}^0 \b \epsilon(\b u_h)- \b \epsilon(\b u) \bigr) +
\\
& \quad + C
\bigl( \b\sigma(\cdot, \b \Pi^0_{k-1}\b \epsilon(\b u_I) ) - \b\sigma(\cdot, \b \epsilon(\b u) ) \bigr) : \bigl( \Pi_{k-1}^0 \b \epsilon(\b u_h)- \b \epsilon(\b u) \bigr)
=:T_2^A+T_2^B \,.
\end{aligned}
\end{equation}
Using Lemma \ref{lm:Hirn} and Lemma \ref{lem:young_shifted} we obtain
\begin{equation}
\label{eq:lemmaAT2A}
\begin{aligned}
    T_2^A & \le
    C \varphi'_{|\b\Pi_{k-1}^0 \b \epsilon(\b u_h)|}(|\b\Pi_{k-1}^0 \b \epsilon(\b \xi_h)|) \, | \Pi_{k-1}^0 \b \epsilon(\b u_h)- \b \epsilon(\b u) |
    & \quad & \text{(by \eqref{eq:extrass.continuity})}
    \\
    & \le
    \varepsilon 
    \varphi_{|\b\Pi_{k-1}^0 \b \epsilon(\b u_h)|}(| \Pi_{k-1}^0 \b \epsilon(\b u_h)- \b \epsilon(\b u) |) +
    C(\varepsilon) 
    \varphi_{|\b\Pi_{k-1}^0 \b \epsilon(\b u_h)|}(|\b\Pi_{k-1}^0 \b \epsilon(\b \xi_h)|)
    & \quad & \text{(by \eqref{eq:Young1})}
    \\
    & \le
    \gamma \varepsilon T_2 + 
    C(\varepsilon) \bigl(
    (\b \sigma(\cdot, \b\Pi_{k-1}^0 \b \epsilon(\b u_h)) - \b\sigma(\cdot, \b\Pi_{k-1}^0 \b \epsilon(\b u_I))
    \bigr) :  \b\Pi_{k-1}^0 \b \epsilon(\b \xi_h))  
    & \quad & \text{(by \eqref{eq:extrass.monotonicity})}
\end{aligned}    
\end{equation}
where in the last line $\gamma$ denotes the associated uniform hidden positive constant.
Using analogous arguments we have
\begin{equation}
\label{eq:lemmaAT2B}
\begin{aligned}
    T_2^B & \le
    C \varphi'_{|\b \epsilon(\b u)|}(|\b \epsilon(\b u) - \b\Pi_{k-1}^0 \b \epsilon(\b u_I)|) \, | \b \epsilon(\b u)  -\Pi_{k-1}^0 \b \epsilon(\b u_h) |
    & \quad & \text{(by \eqref{eq:extrass.continuity})}
    \\
    & \le
    \varepsilon 
    \varphi_{|\b \epsilon(\b u)|}(| \b \epsilon(\b u)  -\Pi_{k-1}^0 \b \epsilon(\b u_h) |) +
    C(\varepsilon) 
    \varphi_{|\b \epsilon(\b u)|}(|\b \epsilon(\b u) - \b\Pi_{k-1}^0 \b \epsilon(\b u_I)|)
    & \quad & \text{(by \eqref{eq:Young1})}
    \\
    &\le
    \gamma \varepsilon T_2 + 
    C(\varepsilon) 
    ( |\b \epsilon(\b u)| + |\b\Pi_{k-1}^0 \b \epsilon(\b u_I)|)^{r-2}
    |\b \epsilon(\b u) - \b\Pi_{k-1}^0 \b \epsilon(\b u_I)|^2 
    & \quad & \text{(by \eqref{eq:extrass.monotonicity})}
    \\
    &\le
    \gamma \varepsilon T_2 + 
    C(\varepsilon) 
    |\b \epsilon(\b u) - \b\Pi_{k-1}^0 \b \epsilon(\b u_I)|^r
    & \quad & \text{($r-2\le 0$)}
\end{aligned}    
\end{equation}
Taking $\epsilon = \frac{1}{4 \gamma}$ in \eqref{eq:lemmaAT2A} and \eqref{eq:lemmaAT2B}, from \eqref{eq:lemmaAT2} we obtain
\[
T_2 \lesssim 
\bigl(\b \sigma(\cdot, \b\Pi_{k-1}^0 \b \epsilon(\b u_h)) - \b\sigma(\cdot, \b\Pi_{k-1}^0 \b \epsilon(\b u_I))
     :  \b\Pi_{k-1}^0 \b \epsilon(\b \xi_h)\bigr)  +
|\b \epsilon(\b u) - \b\Pi_{k-1}^0 \b \epsilon(\b u_I)|^r \,.   
\] 
The proof follows combining the bound above with \eqref{eq:lemmaAT1} in \eqref{eq:lemmaA0} and employing Lemma \ref{lem:approx-interp}.
\end{proof}

\begin{lemma}\label{Lemma:B}
Under the same assumptions of Proposition \ref{prop:main}, let $\b u_I \in \VDG$ be the interpolant of $\b u$ (cf. Lemma \ref{lem:approx-interp}) and $\b \xi_h :=\b u_h - \b u_I$, then the following holds
   \[
    S(\widetilde{\b u}_h, \widetilde{\b u}_h)\lesssim
    S(\widetilde{\b u}_h, \widetilde{\b \xi}_h) - 
    S(\widetilde{\b u}_I, \widetilde{\b \xi}_h) + 
       h^{k_1 r} R_1^r \,,
   \]
where $\widetilde{\b v} := (I - \Pi^0_k)\b v$ for
any $\b v \in \b L^2(\Omega)$.   
\end{lemma}

\begin{proof}
Let $\widehat{\b \sigma}(\b x):= |\b x|^{r-2} \b x$ for any $\b x \in \R^{N_E}$, where $N_E$ denotes the number of local DoFs (cf. Section~\ref{sub:forms}). 
Then simple computations yield
\begin{equation}
    \label{eq:lemmaB0}
    \begin{aligned}
S(\widetilde{\b u}_h, \widetilde{\b u}_h) &= 
\sum_{E \in \Omega_h} h_E^{2-r} |\DOF(\widetilde{\b u}_h)|^r =
\sum_{E \in \Omega_h} h_E^{2-r} \widehat{\b \sigma}(\DOF(\widetilde{\b u}_h)) \cdot \DOF(\widetilde{\b u}_h)
\\
&=
\sum_{E \in \Omega_h} h_E^{2-r}  \bigl(\widehat{ \b \sigma}(\DOF(\widetilde{\b u}_h)) -
\widehat{\b \sigma}(\DOF(\widetilde{\b u}_I)) \bigr)\cdot \DOF(\widetilde{\b u}_h) +
\sum_{E \in \Omega_h} h_E^{2-r} \widehat{\b \sigma} (\DOF(\widetilde{\b u}_I)) \cdot \DOF(\widetilde{\b u}_h)
\\
& =: T_1 + T_2 \,.
    \end{aligned}
\end{equation}
In the following $C$ will denote a generic
positive constant independent of $h$ that may change at each occurrence, whereas the parameter $\varepsilon$ adopted in \eqref{eq:lemmaBT1} and \eqref{eq:lemmaBT2} will be specified later.
Using Lemma \ref{lem:young_shifted} and Lemma \ref{lm:Hirn} we infer
\begin{equation}
    \label{eq:lemmaBT1}
    \begin{aligned}
T_1 & \le 
    C \sum_{E \in \Omega_h} h_E^{2-r} \varphi'_{|\DOF(\widetilde{\b u}_h)|}(|\DOF(\widetilde{\b \xi}_h)|) \,  |\DOF(\widetilde{\b u}_h)| 
    & \quad & \text{(by \eqref{eq:extrass.continuity})}
    \\
    & \le
    \varepsilon 
    \sum_{E \in \Omega_h} h_E^{2-r}
    \varphi_{|\DOF(\widetilde{\b u}_h)|}(|\DOF(\widetilde{\b u}_h)|)  +
    C(\varepsilon) \sum_{E \in \Omega_h} h_E^{2-r}
    \varphi_{|\DOF(\widetilde{\b u}_h)|}(|\DOF(\widetilde{\b \xi}_h)|)
    & \quad & \text{(by \eqref{eq:Young1})}
    \\
    &\le
    \gamma \varepsilon 
    \sum_{E \in \Omega_h} h_E^{2-r} |\DOF(\widetilde{\b u}_h)|^r + 
    C(\varepsilon) \sum_{E \in \Omega_h} h_E^{2-r}
    (|\DOF(\widetilde{\b u}_h)| + |\DOF(\widetilde{\b u}_I)|)^{r-2} \, |\DOF(\widetilde{\b \xi}_h)|^2
    & \quad & \text{(by \eqref{eq:extrass.monotonicity})}
    \\
    &=
    \gamma \varepsilon S(\widetilde{\b u}_h, \widetilde{\b u}_h)  + 
    C(\varepsilon) \sum_{E \in \Omega_h} h_E^{2-r}
    (|\DOF(\widetilde{\b u}_h)| + |\DOF(\widetilde{\b u}_I)|)^{r-2} \, |\DOF(\widetilde{\b \xi}_h)|^2 \,,    
    \end{aligned}
\end{equation}
where in the last line $\gamma$ denotes the uniform hidden positive constant from Lemma \ref{lm:Hirn}.
Using analogous arguments we have
\begin{equation}
    \label{eq:lemmaBT2}
    \begin{aligned}
T_2 & \le 
    C \sum_{E \in \Omega_h} h_E^{2-r} \varphi'(|\DOF(\widetilde{\b u}_I)|) \,  |\DOF(\widetilde{\b u}_h)| 
    & \quad & \text{(by \eqref{eq:extrass.continuity})}
    \\
    & \le
    \varepsilon 
    \sum_{E \in \Omega_h} h_E^{2-r}
    \varphi(|\DOF(\widetilde{\b u}_h)|)  +
    C(\varepsilon) \sum_{E \in \Omega_h} h_E^{2-r}
    \varphi(|\DOF(\widetilde{\b u}_I)|)
    & \quad & \text{(by \eqref{eq:Young1})}
    \\
    &\le
    \gamma \varepsilon 
    \sum_{E \in \Omega_h} h_E^{2-r} |\DOF(\widetilde{\b u}_h)|^r + 
    C(\varepsilon) \sum_{E \in \Omega_h} h_E^{2-r}
    |\DOF(\widetilde{\b u}_I)|^{r} 
    & \quad & \text{(by \eqref{eq:extrass.monotonicity})}
    \\
    &=
    \gamma \varepsilon S(\widetilde{\b u}_h, \widetilde{\b u}_h)  + 
    C(\varepsilon) \sum_{E \in \Omega_h} h_E^{2-r}
    |\DOF(\widetilde{\b u}_I)|^{r}  \,,    
    \end{aligned}
\end{equation}
Taking in \eqref{eq:lemmaBT1} and \eqref{eq:lemmaBT2} $\epsilon = \frac{1}{4 \gamma}$, from \eqref{eq:lemmaB0} we obtain
\[
S(\widetilde{\b u}_h, \widetilde{\b u}_h) \lesssim
\sum_{E \in \Omega_h} h_E^{2-r}
    (|\DOF(\widetilde{\b u}_h)| + |\DOF(\widetilde{\b u}_I)|)^{r-2} \, |\DOF(\widetilde{\b \xi}_h)|^2  +
\sum_{E \in \Omega_h} h_E^{2-r}
    |\DOF(\widetilde{\b u}_I)|^{r} \,.   
\] 
We now notice that the right-hand side in the previous bound coincides, up to a multiplicative constant,  with the right-hand side in \eqref{eq:T20}. Therefore employing \eqref{eq:T2A} and \eqref{eq:T2B} we finally obtain
\[
S(\widetilde{\b u}_h, \widetilde{\b u}_h) \lesssim
  S(\widetilde{\b u}_h, \widetilde{\b \xi}_h) -
  S(\widetilde{\b u}_I, \widetilde{\b \xi}_h) +
  h^{k_1 r} R_1^r \,.
\] 
\end{proof}
\end{document}